\documentclass[12pt,a4paper,leqno]{amsart}

\usepackage[top=1.1in, bottom=1.3in, left=1.5in, right=1.5in]{geometry}
\usepackage{amsthm}
\usepackage{amsfonts}         
\usepackage{amsmath}
\usepackage{amssymb}
\usepackage{breqn}
\usepackage{bm}
\usepackage{url}
\usepackage{esint}
\usepackage{fancyhdr}

\newcommand{\FF}{\mathbb{F}}
\newcommand{\R}{\mathbb{R}}
\newcommand{\C}{\mathbb{C}}
\newcommand{\QQ}{\mathbb{Q}}

\newcommand{\Z}{\mathbb{Z}}

\newcommand{\A}{\mathcal{A}}
\newcommand{\B}{\mathcal{B}}
\newcommand{\CC}{\mathcal{C}}

\newcommand{\PP}{\mathbb{P}}

\newcommand{\U}{\mathcal{U}}

\renewcommand{\b}{\mathfrak{b}}

\renewcommand{\d}{\mathfrak{d}}

\newcommand{\p}{\mathfrak{p}}

\theoremstyle{plain}
\newtheorem{theorem}{Theorem}
\newtheorem{lemma}[theorem]{Lemma}
\newtheorem{prop}[theorem]{Proposition}

\theoremstyle{remark}
\newtheorem{remark}{Remark}
\newtheorem{definition}{Definition}

\numberwithin{equation}{section}
\newcommand{\eps}{\varepsilon}
\renewcommand{\Re}{\text{Re}}
\renewcommand{\Im}{\text{Im}}
\begin{document}
\title{The polynomials $X^2+(Y^2+1)^2$ and $X^{2} + (Y^3+Z^3)^2$ also capture their primes}
\author{Jori Merikoski}
\address{Mathematical Institute,
University of Oxford,
Andrew Wiles Building,
Radcliffe Observatory Quarter,
Woodstock Road,
Oxford,
OX2 6GG}
\email{jori.merikoski@maths.ox.ac.uk}

\begin{abstract} 
We show that there are infinitely many primes of the form $X^2+(Y^2+1)^2$ and $X^2+(Y^3+Z^3)^2$. This extends the work of Friedlander and Iwaniec showing that there are infinitely many primes of the form $X^2+Y^4$. More precisely, Friedlander and Iwaniec obtained an asymptotic formula for the number of primes of this form. For the sequences $X^2+(Y^2+1)^2$ and $X^2+(Y^3+Z^3)^2$ we establish Type II information that is too narrow for an aysmptotic formula, but we can use Harman's sieve method to produce a lower bound of the correct order of magnitude for primes of form $X^2+(Y^2+1)^2$ and $X^2+(Y^3+Z^3)^2$. Estimating the Type II sums is reduced to a counting problem which is solved by using the Weil bound, where the arithmetic input is quite different from the work of Friedlander and Iwaniec for $X^2+Y^4$. We also show that there are infinitely many primes $p=X^2+Y^2$ where $Y$ is represented by an incomplete norm form of degree $k$ with $k-1$ variables. For this we require a Deligne-type bound for correlations of hyper-Kloosterman sums.  
\end{abstract}

\maketitle
\tableofcontents

\section{Introduction}
Friedlander and Iwaniec \cite{FI} famously showed that there are infinitely many primes of the form $a^2+b^4$. This result is striking as integers of this form are very sparse -- the number of integers of the form $a^2+b^4$ less than $X$ is of order $X^{3/4}$. Prior to this Fouvry and Iwaniec \cite{fouvryi} had shown that there are infinitely many primes of the form $a^2+p^2$ with $p$ also a prime. The result of \cite{FI} was extended to primes of the form $a^2+p^4$ with $p$ also a prime by Heath-Brown and Li \cite{hbli}. Pratt \cite{pratt} has shown that there are infinitely many primes of the form $a^2+b^2$ with $b$ missing any three fixed digits in its decimal expansion, which is also a very sparse sequence. 

Another key result concerning primes in sparse polynomial sets is the result of Heath-Brown of infinitely many primes of the form $a^3+2b^3$ \cite{hb}, which has been generalized to binary cubic forms by Heath-Brown and Moroz \cite{hbm} and to general incomplete norm forms by Maynard \cite{maynard}. Recently Li \cite{li} showed that there are infinitely many primes of the form $a^3+2b^3$ with $b$ relatively small.

The work of Friedlander and Iwaniec for primes of the form $a^2+b^4$ relies on the factorization over Gaussian integers $a^2+b^4= (b^2+ia)(b^2-ia)$ and the great regularity in the distribution of squares $b^2$. In particular, the argument fails to capture primes of the form $a^2+f(b)^2$ for non-homogeneous quadratic polynomials $f(b)$. Our first main result resolves this, with the caveat that instead of an asymptotic formula we get lower and upper bounds of correct order of magnitude for primes of this shape.
\begin{theorem} \label{maintheorem1}
There are infinitely many primes of the form $a^2+(b^2+1)^2$. More precisely, we have
\[
\sum_{p \leq X} \sum_{p=a^2+(b^2+1)^2}1 \asymp \frac{X^{3/4}}{\log X}.
\]
\end{theorem}
The methods developed in this paper are quite flexible and work for Gaussian primes with one coordinate given by a polynomial sequence that is not too sparse compared to the degree. We are also able to show the following.
\begin{theorem} \label{maintheorem2}
There are infinitely many primes of the form $a^2+(c^3+d^3)^2$ with $c,d > 0$. More precisely, we have
\[
\sum_{p\leq X} \sum_{\substack{p=a^2+(c^3+d^3)^2 \\ c,d > 0}} 1\asymp \frac{X^{5/6}}{\log X}.
\]
\end{theorem}
\begin{remark}
The proof of Theorem \ref{maintheorem1} is easily generalized to primes of the form $a^2+f(b)^2$, where $f(b)=b^2+D$ for any integer $D \neq 0$, with $D=0$ corresponding to \cite{FI}. The proof of Theorem \ref{maintheorem2} as given below generalizes (with minor technical nuisances) to primes of the form $a^2+(f_1(c) + f_2(d))^2$, where $f_1$ and $f_2$ are polynomials of degree at most 3. It is possible to handle also primes of the form $a^2+g(c,d)^2$ for a cubic polynomial $g(c,d)$ which could involve mixed terms, but this would require additional work (replacing the Weil bound by the Deligne bound for two dimensional sums to get an analogue of Lemma \ref{expsum2Dprelimlemma}).
\end{remark}
Finally, analogous to Theorem \ref{maintheorem2} we have the following.
\begin{theorem} \label{maintheorem3}
Let $k$ be sufficiently large. Let $K/\QQ$ be a Galois extension of degree $k$ and fix a basis $\omega_1,\dots, \omega_k$ for the ring of integers $O_K$. Define the incomplete form
\[
N(b_1,\dots, b_{k-1}) := N_{K/\QQ}(b_1 \omega_1+\cdots + b_{k-1} \omega_{k-1}).
\] 
Then there are infinitely many primes of the form $p=a^2+N(b_1,\dots, b_{k-1})^2$ with $b_i \in [0,p^{1/(2k)}]$. More precisely, we have
\[
\sum_{p\leq X} \sum_{\substack{p=a^2+N(b_1,\dots, b_{k-1})^2 \\b_i \in [0,p^{1/(2k)}]}} 1 \asymp \frac{X^{1-1/(2k)}}{\log X}.
\]
\end{theorem}
The assumption that $K/\QQ$ is Galois could be removed with more work. With a lot more precise numerical work it should be possible to give a non-trivial lower bound for all $k \geq 3$.  The result we prove actually approaches to an asymptotic formula as $k$ becomes large. The only reason for restricting to a norm form is that we need the fact that the number of representations
\[
b=N(b_1,\dots, b_{k-1})
\]
with $b_i \ll b^{1/k}$ is bounded by $\ll_\eps b^\eps $, which is not known for general forms.  For the diagonal form $b_1^k + \cdots +b_{k-1}^k$ we can obtain a result conditional to such a hypothesis. Working with a norm form also simplifies the bounding of the relevant exponential sums of Deligne type.

\subsection{Sketch of the arguments}

Let us focus on the case $a^2+(b^2+1)^2$ in this sketch, as the proofs of the other results are similar.  Using a sieve argument counting primes is reduced to estimating Type I and Type II sums of the form
\[
\sum_{d \leq D} \alpha(d) \sum_{\substack{n \sim X \\ n \equiv 0 \, (d)}} \sum_{n=a^2+(b^2+1)} 1 \quad \text{and} \quad \sum_{\substack{m \sim M \\ n \sim N}} \alpha(m) \beta(n) \sum_{mn=a^2+(b^2+1)} 1.
\]
with $MN=X$ with bounded coefficients $\alpha$ and $\beta$ with $\beta$ satisfying the Siegel Walfisz property (\ref{swproperty}). We can evaluate the Type I sums in the range $D \leq X^{3/4-\eta}$ by essentially the same argument as in \cite[Section 3]{FI}, applying the large sieve inequality for roots of quadratic congruences.

We will be able to handle the Type II sums in the range
\[
X^{1/4+\eta} \ll N \ll X^{1/3-\eta}
\] 
for any small $\eta >0$. This is in contrast to FI, where the range is essentially $X^{1/4+\eta} \ll N \ll X^{1/2-\eta}$. Our Type II information is not sufficient for an asymptotic formula, but using Harman's sieve method \cite{harman} we are nonetheless able to get a lower bound of the correct order of magnitude. In contrast to the work of Friedlander and Iwaniec \cite{fisieve}, our sieve argument is much more similar to the one in Heath-Brown's work \cite{hb}. 

In the situation of Theorem \ref{maintheorem2} we are able to handle Type I sums for $D \leq X^{5/6-\eta}$ and Type II sums in the range $X^{1/6+\eta} \ll N \ll X^{2/9-\eta}$.

For the Type II sums we first use a device of Heath-Brown to reduce to the special case where $\beta(n)$ satisfies the Siegel-Walfisz property with main term 0 (see  (\ref{swmain0})) -- this replaces the sieve argument of Friedlander and Iwaniec \cite{fisieve} in our work.  It then suffices to show that for $\beta$ satisfying (\ref{swmain0}) we have
\[
 \sum_{\substack{m \sim M \\ n \sim N}} \alpha(m) \beta(n)  \sum_{mn=a^2+(b^2+1)} 1 \ll_C X^{3/4} / \log^C X.
\]
The initial part of the argument is the same as in \cite[Sections 4 and 5]{FI}. Assuming for simplicity that $(m,n)=1$, we may write 
\[
m=|w|^2, \quad n=|z|^2, \quad \overline{w}z = b^2+1 + i a
\]
for Gaussian integers $w,z \in \Z[i]$, and denote $\alpha_w=\alpha(|w|^2)$ and $\beta_z=\beta(|z|^2)$. Then our Type II sums are essentially
\[
\sum_{|w|^2 \sim M} \alpha_w \sum_{|z|^2 \sim N} \beta_z \sum_{\substack{\Re(\overline{w}z) = b^2+1 }} 1.
\]
By applying Cauchy-Schwarz similarly to \cite[Sections 4-6]{FI}, the task is reduced to showing that for all $C > 0$ (for $\psi$ a fixed smooth function)
\[
 \sum_{z_1,z_2} \beta_{z_1} \overline{\beta_{z_2} } \sum_{\substack{b_1,b_2 \\ b_1^2+1 \equiv a (b_2^2+1) \, (\Delta)}} \psi(b_1/B_1) \psi(b_2/B_2) \ll_C X^{1/2} N \log^{-C} X , 
\]
for any $B_1,B_2 \ll X^{1/4}$, where for $z_j=r_j+is_j$
\[
\Delta := \Im(\overline{z_1} z_2) \quad \text{and} \quad a := z_2/z_1 = \frac{\overline{z_1 }z_2}{|z_1|^2} \equiv \frac{r_1r_2+s_1s_2}{r_1^2+s_1^2} \, (|\Delta|).
\]
Essentially the only difference to \cite[Section 6]{FI} here is that instead of the congruence $b_1^2 \equiv a b_2^2 \, (\Delta)$ we are counting solutions to 
\[
b_1^2+1 \equiv a (b_2^2+1) \, (\Delta).
\]
The contribution from the diagonal $\Delta=0$ is sufficiently small by trivial bounds provided that $N \gg X^{1/4+\eta}$.

For $\Delta \neq 0$ applying the Poisson summation formula to the sums over $b_1$ and $b_2$ we get a main term which is controlled by
\[
N(a;\Delta) := \frac{B_1 B_2}{\Delta^2} |\{x_1,x_2 \,(\Delta): x_1^2+1 \equiv a(x_2
^2+1) \,\, (\Delta)\}|,
\]
and an error term controlled by an average over the frequencies $(h_1,h_2) \neq (0,0)$ of the exponential sums
\[
S(a,h_1,h_2;\Delta) := \sum_{\substack{x_1,x_2 \, (\Delta) \\ x_1^2+1 \equiv a(x_2
^2+1) \, (\Delta)}} e_{|\Delta|}(h_1 x_1 + h_2 x_2).
\]
This sum runs over a non-singular curve (for the generic $a$), so that by the Weil bound we get (morally)
\[
S(a,h_1,h_2;\Delta) \ll |\Delta|^{1/2+\eps} \ll N^{1/2+\eps}.
\]
This bound is sufficient provided that $N \ll X^{1/3-\eta}$. For the analogous sum in \cite{FI}
\[
S_{\text{FI}}(a,h_1,h_2;\Delta) := \sum_{\substack{x_1,x_2 \, (\Delta) \\ x_1^2\equiv ax_2
^2 \, (\Delta)}} e_{|\Delta|}(h_1 x_1 + h_2 x_2)
\]
one morally gets $S_{\text{FI}}(a,h_1,h_2;\Delta)  \ll N^\eps$, which is why for $a^2+b^4$ one can handle the Type II sums up  to $N \ll X^{1/2-\eta}$. 

For the main term we have to evaluate $N(a;\Delta)$. Since we are counting points on a non-singular curve, this is a much easier problem than the corresponding one in \cite{FI} which is over a singular curve
\[
N_{\text{FI}}(a;\Delta):= |\{x_1,x_2 \,(\Delta): x_1^2 \equiv ax_2
^2 \,\, (\Delta)\}|.
\]
For simplicity assume that $\Delta$ is square-free. Writing
\[
N(a;p) := p + \eps_{p}(a), \quad \eps_d := \prod_{p| d} \eps_{p},
\]
we have by the Weil bound $\eps_p(a) \ll p^{1/2}$ (for generic $a$), which gives $\eps_d(a) \ll_\eps d^{1/2+\eps}$. By the Chinese remainder theorem
\[
N(a;\Delta) =  \prod_{p| d} N(a;p)  =   \prod_{p| d}(p + \eps_{p}(a)) = |\Delta| \sum_{\substack{d| \Delta }}  \frac{\eps_d(a)}{d},
\]
so that we find (ignoring the condition $(d,\Delta/d)=1$ for simplicity)
\[ 
B_1 B_2\sum_{z_1,z_2} \beta_{z_1} \overline{\beta_{z_2} } \frac{B_1 B_2}{\Delta^2} N(a;\Delta)  = B_1^2 B_2^2 \sum_{d \ll N} \sum_{\substack{z_1,z_2 \\ \Delta \equiv 0 \, (d)}} \beta_{z_1} \overline{\beta_{z_2}}\frac{1}{|\Delta|} \frac{\eps_d(a)}{d}.
\]
By using the bound $\eps_d(a) \ll_\eps d^{1/2+\eps}$ we can truncate the sum at $d\leq \log^{2C} X$, getting
\[
B_1^2 B_2^2\sum_{d \leq \log^{2C} X} \sum_{\substack{z_1,z_2 \\ \Delta \equiv 0 \, (d)}} \beta_{z_1} \overline{\beta_{z_2}} \frac{1}{|\Delta |}\frac{\eps_d(a)}{d} + O_C(X N \log^{-C} X),
\]
which can now be bounded by the same argument as in \cite[Section 16]{FI}, using the Siegel-Walfisz property (\ref{swmain0}). In contrast, Friedlander and Iwaniec \cite{FI} have a singular curve with
\[
N_{\text{FI}}(a;p) = p + \bigg(\frac{a}{p} \bigg) p,
\]
which produces the sum
\[
B_1^2 B_2^2\sum_{d \ll N} \sum_{\substack{z_1,z_2 \\ \Delta \equiv 0 \, (d)}} \beta_{z_1} \overline{\beta_{z_2}} \frac{1}{|\Delta|} \bigg(\frac{a}{\Delta} \bigg).
\]
This sum cannot be truncated and for large $d$ they have to show cancellation in the sums over $z_1$ and $z_2.$ Thus, for the main term having $(b^2+1)^2$ rather than $b^4$ turns out to be a friend rather than an enemy. Similar arguments apply for the sequence $a^2+(b^3+c^3)^2$

The article is structured as follows. In Section \ref{typeisection} we state and prove a Type I estimate (Proposition \ref{typeiprop}). Using this we prove a fundamental lemma of the sieve type result (Proposition \ref{flprop}). In Section \ref{typeiisection} we state our Type II information (Proposition \ref{typeiiprop}) and reduce it to a simpler case (Proposition \ref{typeiiproper}). In Section \ref{weilsection} we will state several corollaries of the Weil bound. After this preparation we give the proof of our Type II estimate in Section \ref{typeiiproofsection}. In Section \ref{sievesection} we apply Harman's sieve method with the gathered arithmetic information to prove Theorems \ref{maintheorem1} and \ref{maintheorem2}. In Sections \ref{setup3section}, \ref{ai3section}, and \ref{sieve3section} we give the proof of Theorem \ref{maintheorem3}, where we will assume that the reader is familiar with the techniques from the previous sections. The key estimates are Lemma \ref{deligneboundlemma} and Lemma \ref{sumoverhlemma}. The first gives essentially a square-root bound for the completed exponential sums unless the frequency parameters lie in some bad hyperplanes, as an application of bounds for correlations of hyper-Kloosterman sums \cite{polymath}. The second shows that on average over the frequency parameters are not too often in these bad hyperplanes.

\subsection{Notations}
For functions $f$ and $g$ with $g \geq 0$, we write $f \ll g$ or $f= O(g)$ if there is a constant $C$ such that $|f|  \leq C g.$ The notation $f \asymp g$ means $g \ll f \ll g.$ The constant may depend on some parameter, which is indicated in the subscript (e.g. $\ll_{\epsilon}$).
We write $f=o(g)$ if $f/g \to 0$ for large values of the variable. For summation variables we write $n \sim N$ meaning $N<n \leq 2N$, that is,
\[
\sum_{n \sim N} f_n = \sum_{N < n \leq 2N} f_n.
\] 

We let $C > 0$ denote a large constant. In particular, we write $f \ll_C g/ \log^C X$ meaning that this bound holds for any fixed $ C >0$.

For a statement $E$ we denote by $\bm{1}_E$ the characteristic function of that statement. We let $e(x):= e^{2 \pi i x}$ and $e_q(x):= e(x/q)$ for any integer $q \geq 1$.
We write $P^{+}(n)$ for the largest prime factor of $n$, with the convention that $P^{+}(1)=1$. We denote
\[
P(W) := \prod_{p < W} p \quad \text{and} \quad P(V,W):= \prod_{V \leq p < W} p.
\]
\subsection{Acknowledgements}
I am grateful to Kaisa Matom\"aki and James Maynard for helpful discussions and comments. I am also grateful for the anonymous referee for their careful reading of the manuscript, helpful suggestions, and spotting errors in a previous version of the manuscript. The author was supported by a grant from the Emil Aaltonen foundation. This project has
received funding from the European Research Council (ERC) under the European Union's Horizon 2020 research and innovation programme (grant agreement No 851318).

\section{The set-up for Theorems \ref{maintheorem1} and \ref{maintheorem2}}
Denote
\begin{equation} \label{omegadef}
\Omega(c,d):= 9 c^2d^2/(c^3+d^3).
\end{equation}
Then for large $Y$ we have
\[
 \sum_{\substack{c^3+d^3 \sim Y \\ c,d >0}} \frac{9c^2 d^2}{c^3+d^3} =(1+O(Y^{-\eta}))\sum_{\substack{c_1+d_1 \sim Y \\ c_1,d_1 >0}} \frac{1}{c_1+d_1} =(1+O(Y^{-\eta}))\sum_{n \sim Y} 1 =(1+O(Y^{-\eta}))Y,
\]
which follows from
\begin{equation}\label{omegaintegral}
\int_{0}^{Y^{1/3}} \int_0^{(Y-u^3)^{1/3}} \Omega(u,v) dvdu = Y \int_{0}^{1} \int_0^{(1-u^3)^{1/3}} \Omega(u,v) dvdu = Y,
\end{equation}
by using \eqref{omegaintegral} with $Y$ and $2Y$ and taking the difference. Heuristically $c_1^{-2/3}/3$ and $d_1^{-2/3}/3$ are the probabilities that $c_1$ and $d_1$ are perfect cubes, and $n=c_1+d_1$ is the number of representations as a sum. 
Also, set
\[
\rho_1(d):= |\{\nu \, (d): \, \nu^{2}+1 \equiv 0 \, (d)\}|, \quad \rho_2(d):= |\{\alpha,\beta \, (d): \, \alpha^3+\beta^3 \equiv 0 \, (d)\}| 
\]
and define the arithmetic factors (see Remark \ref{defexplanationremark})
\begin{equation} \label{kappadef}
\kappa_j := \sum_{c \geq 1} \frac{\mu(c) }{c^2} \bigg(\sum_{c \geq 1} \frac{\mu(c) \rho_j(c)}{c^{1+j}} \bigg)^{-1}.
\end{equation}
We define the normalized sequences
\[\begin{split}
a_n^{(1)} := \kappa_1 \sum_{\substack{n=a^2+(b^2+1)^2 \\ (a,b^2+1)=1 \\ b >0}} 2b \quad \text{and} \quad
a_n^{(2)} := \kappa_2 \sum_{\substack{n=a^2+(c^3+d^3)^2 \\ (a,c^3+d^3)=1 \\ c,d>0 }} \Omega(c,d).
\end{split}
\]
To prove Theorems \ref{maintheorem1} and \ref{maintheorem2} we need to show a lower and upper bounds for the sums along primes 
\[
\sum_{p \sim X} a^{(j)}_p \quad \text{for} \quad j \in \{1,2\} \]
and heuristically we expect that both should be asymptotic to
\[
\sum_{p \sim X} b_p \quad \text{for} \quad b_n := \sum_{\substack{n=a^2+b^2 \\ (a,b)=1 \\ b >0}} 1.
\]

We will use Harman's sieve method, where the basic idea is to try to compare the sum $\sum_{p \sim X} a^{(j)}_p $ to $\sum_{p \sim X} b_p $. This means that all of our arithmetic information (Type I and Type II sums) is given in  terms of a comparison, see Propositions \ref{typeiprop} and \ref{typeiiprop}.

\section{Type I information} \label{typeisection}
We have the following two propositions, the latter of which is a corollary to the first. We will prove both of these in this section.
\begin{prop}\emph{(Type I estimate).} \label{typeiprop}  Let $\eta,\eta' > 0$ and suppose that $\eta$ is sufficiently small in terms of $\eta'$ ($\eta=\eta'/100$ suffices). Let $D_1 \leq X^{3/4-\eta'}$ and $D_2 \leq X^{5/6-\eta'}$.  Then for any bounded coefficients $\alpha(d)$ we have
\[
\sum_{\substack{ d \sim D_j \\ n \sim X/d}} \alpha(d) (a^{(j)}_{dn} - b_{dn}) \ll_\eta X^{1-\eta}.
\]
\end{prop}
\begin{prop}\emph{(Fundamental lemma of the sieve).} \label{flprop} \label{fundamentallemma}Let $\eta,\eta' > 0$ and suppose that $\eta$ is sufficiently small in terms of $\eta'$. Let $W:=X^{1/(\log\log x)^2}$. Let $D_1 \leq X^{3/4-\eta'}$ and $D_2 \leq X^{5/6-\eta'}$. Then for any bounded coefficients $\alpha(d)$ we have for any $C > 0$
\[
\sum_{\substack{ d \sim D_j \\ n \sim X/d}} \alpha(d)\bm{1}_{(n,P(W))=1} (a^{(j)}_{dn} - b_{dn}) \ll_{\eta,C} X \log^{-C} X.
\]
where $\eta > 0$ depends on $\eta' >0.$ 
\end{prop}

\begin{remark} \label{defexplanationremark}
In the definitions of  $a_n^{(k)}$, and $b_n$ we have included the coprimality conditions such as $(a,b^2+1)=1$, and $(a,b)=1$ so that the Type I sums match. Without this Proposition \ref{typeiprop} would be false. For instance, if $d=5$, then we have
\[
\begin{split}
\sum_{\substack{a \sim A \\ b^2+1 \sim B \\ a^2+(b^2+1)^2 \equiv 0 \, (5)}} 2b &= \sum_{\substack{\mu,\nu  \, (5)\\ \mu^2+(\nu^2+1)^2 \equiv 0\,\,(5)}}  \sum_{\substack{a \sim A \\ a \equiv \mu \, (5)}}\sum_{\substack{b^2+1 \sim B \\ b \equiv \nu \, (5)}} 2b =(1+o(1)) \sum_{\substack{\mu,\nu  \, (5)\\ \mu^2+(\nu^2+1)^2 \equiv 0\,\,(5)}} \frac{AB}{25}  \\
&=(1+o(1)) \frac{8}{25}AB 
\end{split}
\]
and
\[
\begin{split}
\sum_{\substack{a \sim A \\ b \sim B \\ a^2+b^2 \equiv 0 \, (5)}} 1 &= \sum_{\substack{\mu,\nu  \, (5)\\ \mu^2+\nu^2 \equiv 0\,\,(5)}}  \sum_{\substack{a \sim A \\ a \equiv \mu \, (5)}}\sum_{\substack{b \sim B \\ b \equiv \nu \, (5)}} 1 =(1+o(1))  \sum_{\substack{\mu,\nu  \, (5)\\ \mu^2+\nu^2 \equiv 0\,\,(5)}} \frac{AB}{25} =(1+o(1)) \frac{9}{25} AB.
\end{split}
\]
If we replace the modulus $5$ by $3$, then the first sum is empty while the second is $\sim AB/9$. The issue arises from the potential common factors of $d$ and $\mu$. For every odd prime divisor $p|(d,\mu)$ there is exactly one $\nu \, (p)$ such that $p|\nu$ but there are either zero or two $\nu$ for which $p|(\nu^2+1)$. Inserting the coprimality conditions we get $(p,\mu)=1$ which removes this problem. To see that we need to normalize by a factor $\kappa_2$ consider $d=1$ for which
\[
\sum_{\substack{a\sim A\\ (b^2+1)\sim B \\ (a,b^2+1)=1}} 2b = \sum_{e}\mu(e) \sum_{\substack{a\sim A\\ (b^2+1)\sim B \\ e|a \\ e| (b^2+1)}} 2b =(1+o(1))  \sum_{e}\frac{\mu(e)\rho_2(e)}{e^2}\sum_{\substack{a\sim A\\ b\sim B }} 1=(1+o(1))  \kappa_2^{-1} \sum_{\substack{a\sim A\\ b\sim B \\ (a,b)=1}}  1.
\]
Similar arguments apply for $k \geq 3$.
\end{remark}

\subsection{Introducing finer-than-dyadic smooth weights} \label{smoothweightsection}
Here we describe a device that allows us to partition a sum smoothly into finer-than-dyadic intervals. Let $\delta \in (0,1/10)$ be small (we will use $\delta=X^{-\eps}$ or $\delta= \log^{-C} X$), and fix a non-negative $C^\infty$-smooth function $\psi$ supported on $[1-\delta,1+\delta]$ and satisfying
\[
|\psi^{(j)}| \ll_j \delta^{-j},\, \, j \geq 0 \quad \text{and} \quad \int_{1/2}^2 \psi(1/t) \frac{dt}{t} = \delta.
\]
Suppose that we want to introduce a smooth partition to bound a sum of the form $\sum_{n \leq N} f_n$. We can write (using a change of variables $t \mapsto t/n$)
\[
\sum_{n \leq N} f_n = \frac{1}{\delta} \sum_{n \leq N} f_n \int_{1/2}^2 \psi(1/t) \frac{dt}{t} =\frac{1}{\delta}   \sum_{n \leq N} \int_{1/2}^{2N} f_n \psi(n/t)  \frac{dt}{t}= \frac{1}{\delta} \int_{1/2}^{2N}  \sum_{n \leq N} f_n \psi(n/t) \frac{dt}{t},
\]
so that 
\[
\bigg| \sum_{n \leq N} f_n \bigg| \leq \frac{1}{\delta} \int_{1/2}^{2N} \bigg| \sum_{n \leq N} f_n \psi(n/t) \bigg| \frac{dt}{t}.
\]
Hence, at the cost of a factor $\delta^{-1} \log N$, it suffices to consider sums of the form
\[
 \sum_{n \leq N} f_n \psi(n/t)
\]
for $t \leq  2N$. Naturally, if the original sum is
$ \sum_{n \sim N} f_n$ (dyadic $n$), then it suffices to consider $\sum_{n \sim N} f_n \psi(n/t)$ for $t \asymp N$ at a cost $\delta^{-1}$. The effect is the same as with the usual smooth partition of unity except that we did not need to explicitly construct the partition functions $\psi$. We will denote $\psi_t (n):=\psi(n/t)$.
\subsection{Lemmas}
For the proof we need the large sieve inequality for roots of quadratic congruences (see \cite[Lemma 3.2]{FI} ).
\begin{lemma} \label{largesievelemma} For any complex numbers $\gamma_n$ we have
\[
\sum_{d\leq D} \sum_{\nu^2+1 \equiv 0 \, (d)} \bigg| \sum_{n \leq N} \gamma_n e_d (\nu n) \bigg|^2 \ll (D+N)  \sum_{n \leq N} |\gamma_n|^2.
\]
\end{lemma}
The following lemma gives a truncated version of the Poisson summation formula.
\begin{lemma}\emph{(Poisson summation).}  \label{poissonlemma} Let $\psi$ be as in Section \ref{smoothweightsection} for some $\delta \in (0,1/10)$ and denote $\psi_N(n):=\psi(n/N)$. Let $x \gg 1$ and let $q \sim Q$ be an integer. Let $\eps >0$ and denote
\[
H:= \delta^{-1} (QN)^{\eps} Q/N
\]
Then for any $A >0$
\begin{align*}
\sum_{n \equiv a \, (q)} \psi_N(n) = \frac{N}{q} \hat{\psi}(0) + \frac{N}{q} \sum_{1 \leq |h| \leq H} \widehat{\psi} \left( \frac{h N}{q}\right) e_q (-ah) + O_{A,\epsilon,\psi}((QN)^{-A}),
\end{align*}
where $\hat{f}(h):= \int f(u)e(hu) du$ is the Fourier transform.
\end{lemma}
\begin{proof}
By the usual Poisson summation formula we have
\[
\sum_{n \equiv a \, (q)} \psi_N(n)=\frac{N}{q} \sum_{h} \widehat{\psi} \left( \frac{h N}{q}\right) e_q (-ah)
\]
For $|h| > H$ we have by integration by parts $j\geq 2$ times
\[
\widehat{\psi} \left( \frac{h N}{q}\right) = \int \psi(u)e(u hN/q) du \ll_j \delta^{-j} (hN/q)^{-j} \ll_j (QN)^{-j\eps}(h/H)^{-2},
\]
which gives the result.
\end{proof}
We need the following lemma bounding the number of representations of an integer as a sum of two cubes.
\begin{lemma} \label{cubereplemma} The number of representations of $n$ as sum of two cubes $c^3+d^3$ with $c,d > 0$ is bounded by $\ll_\eps n^\eps$.
\end{lemma}
\begin{proof}
Let $n=c^3+d^3$. By factoring out the common factor it suffices to consider representations where $(c,d)=1$ and $(d,n)=1$. Any such representation produces a root of the equation $\omega^3 +1\equiv 0 \, (n)$ by $\omega=c\bar{d}$ where $d \bar{d} \equiv 1 \,\, (n)$. If we have another representation $n=c_1^3+d_1^3$ with $\omega=c_1\bar{d_1}$, then
\[
c d_1 \equiv c_1 d \,\, (n).
\]
Since $c d_1, c_1 d < c^3+d^3=n$, this implies  $c d_1 = c_1 d $. Since $(c,d)=1$ and $c,c_1,d,d_1 > 0$, we must have $c=c_1$ and $d= d_1$. The lemma now follows since the number of roots $\omega$ is by the Chinese remainder theorem at most $\ll_\eps n^\eps$. The claim also follows trivially from the factorization $c^3+d^3 = (c+d)(c^2-cd+d^2)$ and the divisor bound.
\end{proof}
In the proof of Proposition \ref{flprop} we use the following standard bound for exceptionally smooth numbers
(see \cite[Chapter III.5, Theorem 1]{Ten}, for instance). 
\begin{lemma} \label{smoothlemma} For any $2 \leq Z \leq Y$ we have
\begin{align*}
\sum_{\substack{n \sim Y \\ P^+(n) < Z}} 1 \, \ll \, Y e^{-u/2},
\end{align*}
where $u:= \log Y/\log Z.$
\end{lemma}
We also require the following elementary bound (see \cite[Lemma 1]{fisieve}, for instance).
\begin{lemma} \label{divisorlemma0} For every square-free integer $n$ and every $k \geq 2$ there exists some $d|n$ such that $d \leq n^{1/k}$ and
\begin{align*}
\tau(n) \leq 2^k \tau(d)^k.
\end{align*}
\end{lemma}
From this we get the more general version.
\begin{lemma} \label{divisorlemma} For every integer $n$ and every $k \geq 2$ there exists some $d|n$ such that $d \leq n^{1/k}$ and
\begin{align*}
\tau(n) \leq 2^{k^2} \tau(d)^{k^3}.
\end{align*}
\end{lemma}
\begin{proof}
Write $n=b_1 b_2^2 \cdots b_{k-1}^{k-1} b_k^k$ with $b_1, \dots, b_{k-1}$ square-free, by letting $b_k$ be the largest integer such that $b_k^k | n$, so that $n/b_k^k$ is $k$-free and splits uniquely into $b_1 b_2^2 \cdots b_{k-1}^{k-1}$ with $b_j$ square-free. We have
\[
\tau(n) \leq \tau(b_1) \tau(b_2)^2 \cdots \tau(b_k)^k.
\]
By Lemma \ref{divisorlemma0} for all $j \leq k-1$ there is some $d_j|b_j$ with $d_j \leq b_j^{1/k}$ and $\tau(b_j) \leq 2^k\tau(d_j)^k$. Hence, for $d=d_1\cdots d_{k-1} b_k$ we have 
\[
d \leq (b_1\cdots b_{k-1})^{1/k} b_k  \leq (b_1\cdots b_{k-1}b_k^k)^{1/k}  \leq n^{1/k}
\]
and
\[
\tau(n) \leq (2\tau(d_1)\cdots \tau(d_{k-1}) \tau(b_k))^{k^2} \leq 2^{k^2}\tau(d)^{k^3}.
\]
\end{proof}

\subsection{Proof of Proposition \ref{typeiprop}}
In this section all implicit constants depend on the parameter $\eta$, that is, $O$ and $\ll$ stand for $O_\eta$ and $\ll_\eta$ throughout this section. Let us first consider the case $j=1$, and denote
\[
\Delta(b):= 1 - \kappa_1\sum_{b=b_0^2+1}2|b_0| .
\]
Then the claim is that for $D_1 \leq X^{3/4-\eta'}$ we have for some $\eta > 0$
\[
\sum_{d \sim D_1} \alpha_d \sum_{\substack{a^2+b^2 \sim X \\
a^2+b^2 \equiv 0 \, (d) \\ (a,b)=1 }} \Delta(b) \ll X^{1-\eta}.
\]
We split the sums over $a$ and $b$ into short sums using Section \ref{smoothweightsection} with $\delta=X^{-\eps}$ with $\eps=4\eta$. Letting $A,B \leq 4 X^{1/2}$ with $A^2+B^2 \sim X$, we consider
\[
\sum_{d \sim D_1} \alpha_d \sum_{\substack{a^2+b^2 \sim X \\
a^2+b^2 \equiv 0 \, (d) \\ (a,b)=1  }} \Delta(b) \psi_A(a)\psi_B(b),
\]
which is trivially bounded by $\ll_\eps \delta^{-1} AB X^{\eps}$ (using $\tau(a^2+b^2) \ll X^\eps$). The contribution from $A \leq X^{1/2-\eta-2\eps}$ or $B \leq X^{1/2-\eta-2\eps}$ is 
\[
\frac{1}{\delta^2} \iint_{\substack{ A,B \leq 4 X^{1/2}\\ \min\{A,B\} \leq X^{1/2-\eta-2\eps} }} \sum_{a,b} |\Delta(b)| \psi_A(a)\psi_B(b) \tau(a^2+b^2) \frac{dAdB}{AB} \ll_\eps  X^{1-\eta},
\]
so that we may restrict to $A,B > X^{1/2-\eta-2\eps}$. The condition $a^2+b^2 \sim X$ is trivially true except in the diagonal cases when
\[
A^2+B^2 = X + O(X^{1-\eps}) \quad \text{or}\quad  A^2+B^2 = 2X + O(X^{1-\eps}),
\]
and the total contribution from these diagonal cases is at most
\[
\sum_{k \in \{1,2\}} \frac{1}{\delta^2} \iint_{\substack{ A,B \leq 4 X^{1/2}\\ \max\{A,B\} > X^{1/2-\eta-2\eps} \\  A^2+B^2 = kX + O(X^{1-\eps})}} \sum_{a,b} |\Delta(b)| \psi_A(a)\psi_B(b) \tau(a^2+b^2) \frac{dAdB}{AB} \ll_\eps  X^{1-\eps/2}.
\]
Hence, it suffices to show that for all $A,B \in [X^{1/2-\eta},10 X^{1/2}]$ we have
\begin{equation} \label{smoothedclaim}
\sum_{d \sim D_1} \alpha_d \sum_{\substack{a^2+b^2 \equiv 0 \, (d) \\ (a,b)=1 } } \Delta(b) \psi_A(a)\psi_B(b) \ll X^{1-2\eps-\eta}.
\end{equation}
The sum can be written as
\[
\begin{split}
\sum_{d \sim D_1} &\alpha_d \sum_{b} \Delta(b)  \psi_B(b) \sum_{\substack{a^2\equiv - b^2 \, (d)  \\ (a,b)=1 }} \psi_A(a) =\sum_{d \sim D_1} \alpha_d \sum_{b} \Delta(b)\psi_B(b) \sum_{\nu^2+1 \equiv 0 \, (d)} \sum_{\substack{a\equiv \nu b \, (d) \\   (a,b)=1 } } \psi_A(a)
\end{split}
\]
Expanding the condition $(a,b)=1$ using the M\"obius function gives
\[
S(A,B):=\sum_{d \sim D_1} \alpha_d \sum_{b} \Delta(b)\psi_B(b) \sum_{c|b} \mu(c) \sum_{\nu^2+1 \equiv 0 \, (d)} \sum_{\substack{a\equiv \nu b \, (d) \\ a \equiv 0 \, (c) } } \psi_A(a).
\]
The contribution from $c > X^\eps$ is trivially bounded by
\[
\begin{split}
&\ll  \sum_{b} |\Delta(b)|\psi_B(b) \sum_{\substack{c|b \\ c > X^\eps}} \sum_{a \equiv 0 \, (c)} \psi_A(a) \tau(a^2+b^2)^2 \\
&\ll X^{\eps/4} \sum_{b} |\Delta(b)|\psi_B(b) \sum_{\substack{c|b \\ c > X^\eps}} \sum_{a \equiv 0 \, (c)} \psi_A(a) \ll X^{1-2\eps-\eps/2}.
\end{split}
\]
For $c \leq X^\eps$ in $S(A,B)$ we apply the Poisson summation formula (Lemma \ref{poissonlemma}) with $H:=X^{\eps'} cD_1/A$ for some small $\eps' >0$ to the sum over $a$ to get 
\[
\begin{split}
&\sum_{d \sim D_1} \alpha_d \sum_{b} \Delta(b)\psi_B(b)\sum_{\substack{c|b \\ c \leq X^\eps}} \mu(c)  \frac{\hat{\psi}(0) A\rho_1(d) }{cd} \\
& \hspace{40pt} +\sum_{d \sim D_1} \alpha_d \sum_{b} \Delta(b)\psi_B(b) \sum_{\substack{c|b \\ c \leq X^\eps}} \mu(c)  \sum_{\nu^2+1 \equiv 0 \, (d)} \frac{X^{\eps'}}{H} \frac{D_1}{d} \sum_{0 < |h| \leq H} \hat{\psi}\left( \frac{h A}{cd}\right) e_d(hb \overline{c} \nu )  \\
&=\sum_{d \sim D_1} \alpha_d  \sum_{\substack{c \leq X^\eps}}\mu(c) \sum_{b} \Delta(bc)\psi_B(bc)  \frac{\hat{\psi}(0) A\rho_1(d) }{d} \\
& \hspace{40pt} +\sum_{\substack{c \leq X^\eps}} \mu(c)\sum_{d \sim D_1} \alpha_d \sum_{b} \Delta(bc)\psi_B(bc)  \sum_{\nu^2+1 \equiv 0 \, (d)} \frac{X^{\eps'}}{H} \frac{D_1}{d} \sum_{0 < |h| \leq H} \hat{\psi}\left( \frac{h A}{cd}\right) e_d(hb  \nu ) .
\end{split}
\]
We can absorb the factor $D_1/d \asymp 1$ into the coefficient $\alpha_{d}$ for the second sum. 

\subsubsection{Bounding the main term} \label{typeimaintermsection} 
We get by Poisson summation (Lemma \ref{poissonlemma}, using $B > X^{1/2-\eta}$) 
\[
\sum_{b}\Delta(bc)\psi_B(bc)  = \sum_{b}\psi_B(bc) -  \sum_{\substack{b_0 \\ b_0^2+1 \equiv 0 \, (c)}}\psi_B((b_0^2+1)) 2 b_0 =  \hat{\psi(0)} B \left(  \frac{1}{c}  - \frac{\rho_1(c)}{c}\right)+ O(X^{-10}).
\]
Summing over $c$ and recalling the definition (\ref{kappadef}) of $\kappa_1$ we see that the main terms cancel in
\[
\sum_{c \leq X^{\eps} } \frac{\mu(c)}{c^2} (1 - \kappa_1\rho_1(c)) =\sum_{c  } \frac{\mu(c)}{c^2} - \kappa_1 \sum_{c  } \frac{\mu(c) \rho_1(c)}{c^2}  + O(X^{-\eps}) \ll X^{-\eps},
\]
so that the total contribution from the main term is
\[
\ll\hat{\psi}(0)^2 AB X^{-\eps/2} \ll X^{1-2\eps-\eps/2},
\]
which is sufficient for (\ref{smoothedclaim}).
\subsubsection{Bounding the non-zero frequencies} \label{nonzerotypeisection}
Recall that we need to show that
\[
\begin{split}
\sum_{\substack{c \leq X^\eps}} \sum_{d \sim D_1} \sum_{\nu^2+1 \equiv 0 \, (d)}\bigg|\sum_{b} \Delta(bc)\psi_B(bc)   \frac{X^{\eps'}}{H}  \sum_{0 < |h| \leq H} \hat{\psi}\left( \frac{h A}{cd}\right) e_d(hb  \nu ) \bigg| \ll X^{1-2\eps -\eta}.
\end{split}
\]
We can remove the smooth weight $\hat{\psi}\left( \frac{h A}{cd}\right)$ by applying Mellin transform to separate the variables. Thus, we need to show for any bounded coefficients $c_h$ we have
\[
\sum_{\substack{c \leq X^\eps}}\sum_{d \sim D_1} \sum_{\nu^2 +1 \equiv 0 \, (d)}  \bigg|\sum_{b} \Delta(bc) \psi_B(bc)  \sum_{h \leq H} c_h e_{d}(hb\nu) \bigg| \ll H X^{1-5\eps -\eta}.
\]
Write $\Delta(b)= \Delta_1(b) - \Delta_2(b)$ with
\[
\Delta_1(b) = 1  \quad \text{and}  \quad \Delta_2(b)= \kappa_1\sum_{b=b_0^2+1}2|b_0|.
\]
It then suffices to show that for all $c \leq X^{\eps}$ we have (recalling $H = X^{\eps'} c D_1 / A$ and $A \leq 2 X^{1/2}$)
\begin{equation} \label{fouriersmoothclaim}
\sum_{d \sim D_1}   \sum_{\nu^2 +1 \equiv 0 \, (d)} \bigg|\sum_{b} \Delta_k (bc) \psi_B(bc) \sum_{h\leq H} c_h  e_{d}(hb\nu)\bigg| \ll D_1 X^{1/2-3\eps -\eta} 
\end{equation}
Let us consider this for $k=2$, the case $k=1$ being similar but easier. Let $N:= H B/c$ and denote
\[
\gamma_n := \sum_{n=hb} c_h\Delta_2(bc) \psi_B(bc) \bm{1}_{h \leq H}
\]
Then by Cauchy-Schwarz and Lemma \ref{largesievelemma} the left-hand side in (\ref{fouriersmoothclaim}) is at most
\[
\ll D_1^{1/2}(D_1 + HB)^{1/2} \bigg( \sum_{n \ll N} |\gamma_n|^2 \bigg)^{1/2}.
\]
We have
\begin{equation} \label{divisorreduction}
\sum_{n \ll N} |\gamma_n|^2 \ll B^{1/2}  \sum_{n \ll N} \tau(n) \gamma_n  \ll B \sum_{b_0 \leq 2 B^{1/2}} \sum_{h \leq 2 H'} \tau(h(b_0^2+1)) \ll B^{3/2} H X^{\eps}.
\end{equation}
Thus, the final contribution to  (\ref{fouriersmoothclaim}) is (using $B \leq 2 X^{1/2}$, $H \leq c X^{\eps'}D_1 /A$, $c \leq X^\eps$, and $A \geq X^{1/2-\eps}$)
\[
\ll X^{\eps} (D_1 (B^{3/2} H)^{1/2} + D_1^{1/2} B^{5/4} H) \ll X^{6\eps} D_1^{3/2} X^{1/8} .
\]
Hence, (\ref{fouriersmoothclaim}) is satisfied provided that
\[
 D_1^{3/2} X^{1/8} \ll D_1 X^{1/2-9\eps -\eta},
\]
which holds if
\[
D_1 \ll X^{3/4-3\eta'}.
\]
for $\eps=4\eta$ and $\eta'=100\eta$.
\subsubsection{Case of $a^2+(c^3+d^3)^2$} \label{typeicubesection}
The details are essentially the same, with the only change that we let
\[
\Delta(b) := 1- \kappa_2 \sum_{\substack{b=c_0^3+d_0^3 \\ c_0,d_0 >X^{1/6-\eps}}} \Omega(c_0,d_0),
\]
bounding the contribution from $\min\{c_0,d_0\} \leq X^{1/6-\eps}$ trivially. The contribution of the main term from Poisson summation is handled by the same argument as in Section \ref{typeimaintermsection}, applying Poisson summation to the variables $c_0,d_0$ to extract $\rho_2(c)$.

For the non-zero frequencies, the argument is essentially the same as in Section \ref{nonzerotypeisection}. Let 
\[
\Delta'_2(b) := \kappa_2 \sum_{\substack{b=c_0^3+d_0^3 \\ c_0,d_0 >X^{1/6-\eps}}} \Omega(c_0,d_0) \quad \text{and} \quad 
\gamma'_n := \sum_{n=hb}c_h \Delta'_2(bc) \psi_B(bc) \psi_{H'}(h).
\]
To get the bound corresponding to (\ref{divisorreduction}) one also has to use Lemma \ref{cubereplemma} to bound the number of representations as a sum of two cubes, which gives
 \[
\sum_{n \sim N} |\gamma'_n|^2 \ll X^{\eps\/2} B^{1/3} \sum_{n \sim N} \tau(n) \gamma'_n  \ll X^{\eps/2} B^{2/3} \sum_{c_0,d_0 \leq 2 B^{1/3}} \sum_{h \leq 2 H'} \tau(h(c_0^3+d_0^3)) \ll B^{4/3} H' X^{\eps}. 
 \]
The end bound is (using $B \leq 2 X^{1/2}$, $H_1 \leq c X^{\eps'}D_1 /A$, $c \leq X^\eps$, and $A \geq X^{1/2-\eps}$)
\[
\ll X^{2\eps} (D_1 (B^{4/3} H')^{1/2} + D_1^{1/2} B^{7/6} H') \ll X^{6\eps} D_1^{3/2} X^{1/12},
\]  
which is sufficient for the bound (\ref{fouriersmoothclaim}) if $D_2 \leq X^{5/6-\eta'}$ with $\eps=4\eta$ and $\eta'=100\eta$. \qed

\subsection{Proof of Proposition \ref{flprop}} \label{flpropsection}
To prove Proposition \ref{flprop} we write
\[
\bm{1}_{(n,P(W))=1} = \sum_{e|(n,P(W))} \mu (e) =\sum_{\substack{e|(n,P(W)) \\ e \leq X^{\eta}}} \mu (e)+\sum_{\substack{e|(n,P(W)) \\ e > X^{\eta}}} \mu (e).
\]
For $e \leq X^{\eta}$ we get by Proposition \ref{typeiprop} with level $D_j X^{\eta}$
\[
\sum_{e \leq X^{\eta/2}} \mu (e) \sum_{\substack{ d \sim D_j \\ n \sim X/de}} \alpha(d) (a^{(j)}_{den} - b_{de n}) \ll_C X \log^{-C} X,
\]
since the coefficient of $d'=ed$ is bounded by $\ll_\eta X^{\eta/2}$.

For $e > X^{\eta}$ we note that $e|P(W)$ implies that there is a divisor $e_1| e$ with $e_1 \in [X^{\eta},X^{2\eta}]$. Thus, by crude bounds, Lemma \ref{divisorlemma}, and Proposition \ref{typeiprop} we get
\[
\begin{split}
 \sum_{\substack{e > X^{\eta}\\e | P(W)}} \mu (e)& \sum_{\substack{ d \sim D_j \\ n \sim X/de}} \alpha(d) (a^{(j)}_{den} - b_{de n}) \ll \sum_{\substack{e_1 \in [X^{\eta},X^{2\eta}] \\ P^{+}(e_1) < W}} (a^{(j)}_{e_1n} +b_{e_1 n}) \tau(n)^2 \\
& \ll  \sum_{\substack{e_1 \in [X^{\eta},X^{2\eta}] \\ P^{+}(e_1) < W}} \sum_{d \leq X^{1/10}} \tau(d)^{O(1)}(a^{(j)}_{de_1n} +b_{d e_1 n}) \\ 
& \ll \sum_{\substack{e_1 \in [X^{\eta},X^{2\eta}] \\ P^{+}(e_1) < W}} \sum_{d \leq X^{1/10}} \tau(d)^{O(1)} b_{de_1 n} + O(X^{1-\eta})
\\
& \ll X  \sum_{\substack{e_1 \in [X^{\eta},X^{2\eta}] \\ P^{+}(e_1) < W}} \sum_{d \leq X^{1/10}} \frac{\tau(d)^{O(1)} \tau(d) \tau(e_1)}{de_1}+O(X^{1-\eta})  \\
& \ll X (\log^{O(1)} X) \sum_{f \leq X^{\eta/10} } \frac{\tau(f)^{O(1)} }{f}\sum_{\substack{e_2 \in [X^{\eta}/f,X^{2\eta}/f] \\ P^{+}(e_1) < W}}  \frac{1}{e_2}+O(X^{1-\eta}) \\
&\ll_C X \log^{-C} X 
\end{split}
\]
by applying Lemma \ref{smoothlemma}  to get the last bound. \qed
\section{Type II information: initial reductions} \label{typeiisection}
Let
\[
\varphi_{\Z[i]}(d) := |\{ v \in \Z[i]/d \Z[i]: \, (v,d)=1  \}|.
\]
 For coefficients $\beta(n)$ we denote $\beta_z := \beta(|z|^2)$ for $z \in \Z[i]$ and for any $Y > 0$ write $\psi_Y(y):= \psi(y/Y)$. It is convenient for us to give the following definition of the Siegel-Walfisz property over Gaussian integers, where the variable is weighted by a smooth function supported on a polar box. 
\begin{definition}(Siegel-Walfisz property on $\Z[i]$). We say that coefficients $\beta(n)$, supported on $n \sim N$, satisfies \emph{the Siegel-Walfisz property} if the following holds for any $C>0$ and any smooth weight $\psi$ as in Section \ref{smoothweightsection} with $\delta^{-1}=\log^{O(1)} N$. For any $\theta\in[\pi,5\pi],$ $N_1 \sim N$, $v \in \Z[i]$ and $d = \log^{O(1)}N$ we have
\begin{equation} \label{swproperty}\begin{split}
\sum_{\substack{ z \equiv v \,(d)}} \beta_{z}\psi_{N_1}(|z|^2) \psi_\theta(\arg z+2\pi) =\frac{\widehat{\psi_\theta}(0)\bm{1}_{(v,d)=1}}{\varphi_{\Z[i]}(d)}\sum_{z} \beta_{z} \psi_{N_1}(|z|^2)   + O_C( N \log^{-C} N).
\end{split}
\end{equation}
\end{definition}

For all coefficients $\beta(n)$ that we consider this property follows by standard arguments (similar to \cite[Section 16]{FI}). More precisely, we will need this for $\beta(n)$ which have a large prime factor, that is, for $\beta(n)$ such that for some $PK\asymp N$ with $P \gg W$ and for some divisor bounded coefficients $\gamma(k)$ supported on $(k,P(W))=1$ we have
\begin{align} \label{betageneric}
    \beta(n) = \psi_N(n)\sum_{\substack{n=pk \\ p \equiv 1\, (4)\\ k \sim K}} \gamma(k) .
\end{align}
Fixing $|u|^2=k$  and relabeling, it then suffices to show that for $P\gg W$  we have
\[
\sum_{\substack{   z \equiv v  \,(d)}} \Lambda(|z|^2)\psi_{P}(|z|^2) \psi_\theta(\arg z+2\pi) =\frac{\widehat{\psi_\theta}(0)\bm{1}_{(v,d)=1}}{\varphi_{\Z[i]}(d)}\sum_{z} \Lambda(|z|^2) \psi_{P}(|z|^2)   + O_C( P \log^{-C} N).
\]
Using Dirichlet characters to detect $z\equiv v\, (d)$ and the Fourier series expansion to $\psi_\theta$ the contribution from the principal character and the zeroth frequency  in the Fourier expansion of $\psi_\theta$ cancels the main term on the right-hand side of \eqref{swproperty}.  For non-principal characters or non-zero frequencies we use Melling inversion for $\psi_P$ and shift the contour past $s=1$. Shifting the contour is justified by applying the Siegel-Walfisz bound \cite[Lemma 16.1]{FI} with the character
\[
\psi(z) :=\chi(z) \left(\frac{z}{|z|}\right)^{k}
\]
where $\chi$ a Dirichlet character of $(\Z[i]/d \Z[i])^\times$.

Our Type II information for Theorems \ref{maintheorem1} and \ref{maintheorem2} is given by the following.
\begin{prop}\emph{(Type II information).}\label{typeiiprop}
Let $W=X^{1/(\log \log X)^2}$. For $j \in \{1,2\}$ let $M_j N_j\asymp X$ and for any small $\eta >0$ let
\[
X^{1/4+\eta} \ll N_1 \ll X^{1/3-\eta}
\]
and
\[
X^{1/6+\eta} \ll N_2 \ll X^{2/9-\eta}.
\]
Let $\alpha(m)$ and $\beta(n)$ be bounded coefficients.  Assume that $\beta(n)$ is supported on square-free numbers with $(n,P(W))=1$ and satisfies the Siegel-Walfisz property (\ref{swproperty}) and assume that $\alpha(m)$ is supported on $(m,P(W))=1$. Then for any $C > 0$
\[
\sum_{\substack{ m \sim M_j \\ n \sim N_j}} \alpha(m) \beta(n) (a^{(j)}_{mn} - b_{mn}) \ll_{\eta,C} X \log^{-C} X.
\]
\end{prop}
We begin by reducing to a case where $\beta(n)$ is oscillatory. For this we make the following definition.
\begin{definition}(Siegel-Walfisz property with main term equal to $0$). We say that coefficients $\beta(n)$, supported on $n \sim N$, satisfies \emph{the Siegel-Walfisz property with main term equal to $0$} if the following holds for any $C>0$ and any smooth weight $\psi$ as in Section \ref{smoothweightsection} with $\delta^{-1}=\log^{O(1)} N$. For any $\theta\in[\pi,5\pi],$ $N_1 \sim 2N$, $v \in \Z[i]$ and $d = \log^{O(1)}N$ we have
\begin{equation} \label{swmain0}\begin{split}
\sum_{\substack{ z \equiv v \,(d)}} \beta_{z}\psi_{N_1}(|z|^2) \psi_\theta(\arg z+2\pi) \ll_C N \log^{-C} N.
\end{split}
\end{equation}
\end{definition}
 We show in Section \ref{reductionsection} that for Proposition \ref{typeiiprop} it suffices to prove the following.
\begin{prop} \label{typeiiproper}
Let $W=X^{1/(\log \log X)^2}$. For $j \in \{1,2\}$ let $M_j N_j\asymp X$ and for any small $\eta >0$ let
\[
X^{1/4+\eta} \ll N_1 \ll X^{1/3-\eta}
\]
and
\[
X^{1/6+\eta} \ll N_2 \ll X^{2/9-\eta}.
\]
Let $\alpha(m)$ and $\beta(n)$ be bounded coefficients. Assume that $\beta(n)$  is supported on square-free numbers with $(n,P(W))=1$ and satisfies the Siegel-Walfisz property with main term equal to 0, that is, (\ref{swmain0}).  Assume that $\alpha(m)$ is supported on $(m,P(W))=1$.  Then for any $C > 0$
\[
\sum_{\substack{ m \sim M_j \\ n \sim N_j}} \alpha(m) \beta(n) a^{(j)}_{mn}  \ll_{\eta,C} X \log^{-C} X
\]
and
\[
\sum_{\substack{ m \sim M_j \\ n \sim N_j}} \alpha(m) \beta(n) b_{mn}  \ll_C X \log^{-C} X.
\]
\end{prop}
We will give the proof of this proposition in Section \ref{typeiiproofsection}. 

\subsection{Proof that Proposition \ref{typeiiproper} implies Proposition \ref{typeiiprop}} \label{reductionsection}
The idea is similar to one that appears in Heath-Brown's work \cite{hb}. Let $W:=X^{1/(\log\log x)^2}$. For a given $\delta>0$  let $\delta'=\delta (\log X)^{-C}$ for some large $C>0$ and let $\psi'$ be a smooth function as in Section \ref{smoothweightsection} with parameter $\delta'$. Let 
\begin{align}\label{approxconstruction}
 \CC(\beta, N') := \bigg( \sum_{(z,P(W))=1} \beta(|z|^2) \psi'_{N'}(|z|^2) \bigg) \bigg( \sum_{(z,P(W))=1}  \psi'_{N'}(|z|^2) \bigg)^{-1}   
\end{align}
Given the coefficients $\beta(n)$ we write
\begin{equation*} 
  \beta^{\#}(n):= \bm{1}_{(n,P(W))=1} \int_{N' \sim N_j} \frac{1}{\delta'}\psi'_{N'}(n)  \CC(\beta, N')  \frac{dN'}{N' },
\end{equation*}
approximates $\beta(n)$ by $\bm{1}_{(n,P(W))=1}$  while mimicking the distribution of $\beta(|z|^2)$ over short intervals near $n$. If $\beta(n)$  is bounded and supported on $(n,P(W))=1$, then the coefficient $\beta^\#(n)$ is also bounded since $|C(\beta,N')|\, \ll 1$ and $\int \psi'(1/y) \frac{dy}{y}= \delta'$. By Proposition \ref{fundamentallemma} the claim of Proposition \ref{typeiiprop} holds for $\beta$ replaced by $\beta^\#$, that is, if we let $S(\alpha,\beta)$ denote the Type II sum, then
\begin{equation} \label{typeiifl}
\begin{split}
S(\alpha,\beta^\#) = \sum_{\substack{ m \sim M_j \\ n \sim N_j}}& \alpha(m) \beta^\#(n) (a^{(j)}_{mn} - b_{mn})  \\
& = \frac{1}{\delta'}\int C(\beta,N')\sum_{\substack{ m \sim M_j \\ n \sim N_j}} \alpha(m)  \bm{1}_{(n,P(W))=1}\psi'_{N'}(n) (a^{(j)}_{mn} - b_{mn}) \frac{dN'}{N'} \\
& \ll_C X \log^{-C} X,
\end{split}
\end{equation}
since the weight $\psi'_{N'}(n)$ may be removed by partial summation.
Furthermore, by the fundamental lemma of the sieve (arguing similar to Section \ref{flpropsection}) it follows that $\beta^\#(n)$ satisfies the Siegel-Walfisz condition  \eqref{swproperty}. Thus, if $\beta(n)$ also satisfies the Siegel-Walfisz property, then the coefficient $(\beta - \beta^\#)(n)$ satisfies the Siegel-Walfisz property (\ref{swmain0}) with a main term equal to 0, as can be seen from
\begin{align} \nonumber
&\sum_{\substack{ z \equiv v \,(d)}} (\beta - \beta^\#)(|z|^2)\psi_{N_1}(|z|^2) \psi_\theta(\arg z+2\pi) \\ \nonumber
&=\frac{\widehat{\psi_\theta}(0)\bm{1}_{(v,d)=1}}{\varphi_{\Z[i]}(d)}\sum_{z }  (\beta - \beta^\#)(|z|^2) \psi_{N_1}(|z|^2) + O_C( N \log^{-C} N) \\ \nonumber
&= \frac{\widehat{\psi_\theta}(0)\bm{1}_{(v,d)=1}}{\varphi_{\Z[i]}(d)} \int_{N' \sim N_j} \frac{1}{\delta'} \sum_{z }  (\beta(|z|^2)  - C(\beta,N') \mathbf{1}_{(z,P(W))=1})\psi'_{N'}(|z|^2)\psi_{N_1}(|z|^2) \frac{dN'}{N'}  + O_C( N \log^{-C} N) \\
&\hspace{70pt}\ll_C N \log^{-C} N,\label{sw0ii}
\end{align}
where the last bound holds by the construction \eqref{approxconstruction} since the parameter $\delta'$ (associated to $\psi'_{N'}$) is small in comparison to the parameter $\delta$ (associated to $\psi_{N_1}$).
 Thus, the using decomposition
 \[
 \beta = \beta^\# + (\beta - \beta^\#), \quad S(\alpha,\beta) = S(\alpha,\beta^\#) + S(\alpha,\beta-\beta^\#) 
 \]
  and \eqref{typeiifl} and \eqref{sw0ii}, Proposition \ref{typeiiprop} follows from Proposition \ref{typeiiproper} applied to bound $S(\alpha,\beta-\beta^\#) $.

\section{Applications of the Weil bound} \label{weilsection}
For the proof of Propostion \ref{typeiiproper} we require the Weil bound both for counting points on certain curves over finite fields as well as for bounding algebraic exponential sums. 

Let $p$ be a prime and fix an algebraic closure $\overline{\FF}_p$ of the finite field $\FF_p$. Let $C$ be the projective curve defined by a polynomial equation $F(X,Y,Z) = 0$ with a homogeneous polynomial $F \in \FF_p [X,Y,Z]$. Recall that $C$ is \emph{non-singular} if for all  $(x,y,z) \in \PP^2(\overline{\FF}_p)$ (that is, $(x,y,z) \neq (0,0,0)$ ) with $F(x,y,z)=0$ we have $\partial_X F(x,y,z) \neq 0$ or $\partial_Y  F(x,y,z) \neq 0 $ or $\partial_Z F(x,y,z)  \neq 0$. 

For Theorem \ref{maintheorem1} we need to understand affine curves over $\FF_p$ defined by the equation 
\[
x_1^2+1 = a(x_2^2+1)
\]
with varying $a$. For $a \neq 0,1$ the homogenization $(1-a)x_0^2+ x_1^2-ax_2^2 =0$ defines a non-singular curve. For Theorem \ref{maintheorem2} we end up with affine varieties defined by the equation
\[
x_1^3+x_2^3 = a(x_3^3+x_4^3)
\]
with varying $a$. By fixing $x_3$ and $x_4$ this reduces to understanding curves defined by the equation
\[
x_1^3+x_2^3 = \gamma,
\]
which is non-singular for $\gamma \neq 0$.
\subsection{Points on curves}
We set
\[
N_1(a;d):=|\{x_1,x_2 \,(d): x_1^2+1 \equiv a(x_2
^2+1) \,\, (d)\}|,
\]
and define 
\[
\eps_p(a) := N_1(a;p)-p.
\]
For $k \geq 2$ we define
\[
\eps_{p^k}(a) := N_1(a,p^k) - p N_1(a,p^{k-1}) = N_1(a,p^k) - (p^k +p^{k-1}\eps_p(a) +\cdots + p \eps_{p^{k-1}}(a))
\]
and for any $d$
\[
\eps_d(a) := \prod_{p^k||d} \eps_{p^k}(a).
\]
Then by the Chinese remainder theorem
\begin{align*}
   N_1(a,D) = \prod_{p^k||D} N_1(a,p^k) =  \prod_{p^k||D} \bigg(\sum_{j=0}^k p^{k-j} \eps_{p_j}(a)\bigg) = \sum_{d|D} \frac{D}{d} \eps_d(a).
\end{align*}

\begin{lemma} \label{pointsweil1lemma}
For $a\not\equiv 0,1 \, (p)$ we have
\[
\eps_p(a) \ll p^{1/2}.
\]
For any $a$ and $k \geq 1$ we have
\[
N_1(a;p^k) \leq 2p^k.
\]
For any $d$ we have
\[
|\eps_d(a)| \,\leq d \tau(d)^{O(1)}.
\]
\end{lemma}
\begin{proof}
For $a \not \equiv 0,1 \, (p)$ the projective curve over $\FF_p$ defined by
\[
(1-a)x_0^2+ x_1^2-ax_2^2=0
\]
is non-singular, so the claim follows from the classical Hasse-Weil bound (\cite{weil} or \cite[Excercise V.1.10]{hartshorne}, for instance).
The last claim follows from
\[
|\eps_{p^k}(a)| = |N_1(a,p^k) - p N_1(a,p^{k-1})| \leq \max\{N_1(a,p^k) , p N_1(a,p^{k-1})\} \leq 2p^k.
\]
\end{proof}
Recall that an integer $c$ is said to be powerful if $p|c$ implies $p^2|c$. The next lemma gives an evaluation of $N_1(a;D)$ for a generic $a$ unless $D$ has a large powerful factor. When we apply this we take $Y=\log^C X$ and make use of the fact that $D$ which have a large powerful factor are very sparse.
\begin{lemma} \label{points1Dlemma}
Let $Y\geq 1$. Let $P_Y(d)$ denote the characteristic function of $d$ such that there exists a powerful divisor $c|d$ with $c> Y$. Then
\[
N_1(a;D) = \sum_{\substack{d|D \\ d <Y^4 }} \frac{D}{d} \eps_d(a) + D\tau(D)^{O(1)}  O(Y^{-1}+ P_Y(D)+ \bm{1}_{(a(a-1),D) > Y}).
\]
\end{lemma}
\begin{proof}
By the Chinese remainder theorem we have
\[
N_1(a;D) = \prod_{p^k| |D} N_1(a;p^k). 
\]
If $P_Y(D)=1$  or  $(a(a-1),D) > Y$, then we use the bound $N_1(a;p^k) \leq 2p^k$ to get
\[N_1(a;D) \ll D\tau(D). \]
Suppose then that $P_Y(D)=0$  and  $(a(a-1),D) \leq Y$.  We have 
\[
\begin{split}
N_1(a;D)  = \prod_{p^k|| D} (p^k+\eps_{p^k}(a))  = \sum_{d|D} \frac{D}{d} \eps_d(a) = \sum_{\substack{d|D \\ d <Y^4 }} \frac{D}{d} \eps_d(a) + \sum_{\substack{d|D \\ d \geq Y^4 }} \frac{D}{d} \eps_d(a).
\end{split}
\]
Let us write $D=D_1D_2$ where $D_1$ is the largest square-free divisor with $(D_1,D/D_1)=1$ so that $D_2$ is powerful. Then for any $d|D$ with $d \geq Y^4$ we have $d_1:=(d,D_1) \geq Y^3$ and 
\[
d_0:= \frac{d_1}{(a(a-1),d_1)} \geq Y^2.
\] 
By Lemma \ref{pointsweil1lemma} we have $\eps_{d_0}(a) \ll \tau(D)^{O(1)} d_0^{1/2}$, so that for $d > Y^4$ we have
\[\eps_{d_0}(a) / d_0 \ll\tau(D)^{O(1)} Y^{-1}.\]
and by the trivial bound $\eps_{d/d_0}d_0/d \ll \tau(D)^{O(1)}$ we get
\[
\frac{\eps_d(a)}{d} =  \frac{\eps_{d_0}(a)}{d_0} \frac{\eps_{d/d_0}(a)}{d/d_0}  \ll \tau(D)^{O(1)} Y^{-1}.
\]
Therefore,  the contribution from  $d \geq Y^4$ is
\[
\ll \sum_{\substack{d|D \\ d>Y^4}} \frac{D}{d} \eps_d(a) \ll D\tau(D)^{O(1)}Y^{-1}.
\]
\end{proof}

We now give a similar treatment for the curve corresponding to Theorem \ref{maintheorem2}. Similarly as in Lemma \ref{pointsweil1lemma} we get the following by the Weil bound.
\begin{lemma}
Denote
\[
N_2(\gamma;D):= |\{x_1,x_2 \, (D): x_1^3+x_2^3 \equiv \gamma \, (D)\}| 
\]
Let $\gamma \not \equiv 0 \, (p)$. Then
\[
N_2(\gamma;p)= p + O(p^{1/2})
\]
and
\[
N_2(0;p) \ll p.
\]
Furthermore, $N_2(\gamma,p^k) \ll p^k$.
\end{lemma}

\begin{lemma} \label{points2Dplemma}
Let 
\[
N_3(a;D) := |\{x_1,x_2,x_3,x_4 \, (D) : \, x_1^3+x_2^3 \equiv a(x_3^3+x_4^3) \, (D)\}|
\]
and
define $\eps'_{p^k}(a)$ by
\[
p^2 \eps'_p(a) :=N_3(a;p) - p^3, \quad p^{2k} \eps'_{p^{k}}(a):=  N_3(a;p^k) - p^{3} N_3(a;p^{k-1}).
\]
Set
\[
\eps'_d(a) := \prod_{p^k|| d} \eps_{p^k}(a).
\]
For $a \not \equiv 0 \, (p)$ we have
\[
N_3(a;p) = p^3 +O( p^{5/2}),
\]
and for any $a$ we have $N_3(a,p) \ll p^3.$ Furthermore,
\[
\eps'_d(a) \ll d \tau(d)^{O(1)}.
\]
\end{lemma}
\begin{proof}
Clearly
\[
N_3(a;p) = \sum_{x_3,x_4 \, (p)} N_2(a(x_3^3+x_4^3);p).
\]
The contribution from $x_3^3+x_4^3\equiv 0 \, (p)$ is trivially bounded by $\ll p^2$, and from  $x_3+x_4^3 \not \equiv 0 \, (p)$ we get a contribution
\[
(p^2+O(p))(p+ O(p^{1/2})) = p^3 + O(p^{5/2}).
\]
The last claim follows since by $N_2(\gamma,p^k) \ll p^k$ we have
$N_3(a,p^k) \ll p^{3k}$ and therefore $\eps'_{p^k}(a) \ll p^k$.
\end{proof}

Note that by the Chinese remainder theorem
\[
N_3(a;D) =\prod_{p^k||D} N_3(a,p^k) =  D^2\prod_{p^k||D} \bigg(\sum_{j=0}^k p^{k-j} \eps'_{p_j}(a)\bigg) = \sum_{d|D} \frac{D^3}{d} \eps'_d(a).
\]
Then by a similar argument as in Lemma \ref{points1Dlemma} we get the following.
\begin{lemma} \label{points2Dlemma}
We have for $\eps'_d(a) \ll d \tau(d)^{O(1)}$ that
\[
N_3(a;D) = \sum_{\substack{d|D \\ d <Y^4 }} \frac{D^3}{d} \eps'_d(a) + D^3\tau(D)^{O(1)} O( Y^{-1} +   P_Y(D)  +   \bm{1}_{(a,D) > Y} ).
\]
\end{lemma}
\subsection{Algebraic exponential sums}
\begin{lemma} \label{expsumplemma}
Let $\bm{h}=(h_1,h_2)$. We have for $a \not \equiv 0,1 \, (p)$ 
\[
\sum_{\substack{x_1,x_2 \, (p) \\ x_1^2+1 \equiv a(x_2^2+1) \, (p)}} e_p ( \bm{h} \cdot \bm{x}) \ll (h_1,h_2,p)^{1/2}  p^{1/2},
\]
and for $a \equiv 1 \, (p)$ we have
\[
\sum_{\substack{x_1,x_2 \, (p) \\ x_1^2+1 \equiv a(x_2^2+1) \, (p)}} e_p ( \bm{h} \cdot \bm{x}) \ll (h_1-h_2,p) +  (h_1+h_2,p)
\]
\end{lemma}
\begin{proof}
For $a\equiv 1 \, (p)$ we have
\begin{align*}
    \sum_{\substack{x_1,x_2 \, (p) \\ x_1^2 \equiv x_2^2 \, (p)}} e_p ( \bm{h} \cdot \bm{x}) =&  -1+ \sum_{x \, (p)} e_p ((h_1+h_2)x) +  e_p ((h_1-h_2)x)   \\
    \ll& \, (h_1-h_2,p) +  (h_1+h_2,p).
\end{align*}
Consider $a \not \equiv 0, 1 \, (p)$ and fix $j \in \{1,2\}$. If $(h_1,h_2,p)\neq 1$ we apply the trivial bound, so we may assume $(h_1,h_2,p) =1$.  The claim now follows from the Weil bound (see \cite[Theorem 5]{bombieri}, for instance), applied to the projective (hence complete) curve $X$ defined by the the (homogenized) polynomial
\[
(1-a)x_0^2 + x_1^2-ax_2^2,
\]
which is non-singular for $a \not \equiv 0,1 \, (p).$
\end{proof}
By using the Chinese remainder theorem as in the proof of Lemma \ref{points1Dlemma}, this implies the following.
\begin{lemma} \label{expsums1Dlemma}
Let $D_1$ be the largest square-free divisor of $D$ such that $(D_1,D/D_1)=1$. Let $(a,D)=1$. Then 
\[\sum_{\substack{x_1,x_2 \, (D) \\ x_1^2+1 \equiv a(x_2
^2+1) \, (D)}} e_{D}(\bm{h} \cdot \bm{x}) \ll \tau(D)^{O(1)} D D_1^{-1/2} (h_1,h_2,D_1)^{1/2} (a-1,D_1)^{1/2}
\]
\end{lemma}
\begin{proof}
 By the Chinese remainder theorem we have
 \[
 \sum_{\substack{x_1,x_2 \, (D) \\ x_1^2+1 \equiv a(x_2
^2+1) \, (D)}} e_{D}(\bm{h} \cdot \bm{x}) = \prod_{p^k || D} \sum_{\substack{x_1,x_2 \, (p^k) \\ x_1^2+1 \equiv a(x_2
^2+1) \, (p^k)}} e_{p^k}( (D/p^k)^{-1} \bm{h} \cdot \bm{x}),
 \]
 where $(D/p^k)^{-1}$ denotes the multiplicative inverse modulo $p^k$. For $p^k|| (D/D_1)$ we use the trivial bound
 \[
 \sum_{\substack{x_1,x_2 \, (p^k) \\ x_1^2+1 \equiv a(x_2
^2+1) \, (p^k)}} e_{p^k}( (D/p^k)^{-1} \bm{h} \cdot \bm{x}) \ll p^k.
 \]
 For $p|D_1$ we have by Lemma \ref{expsumplemma}
 \[
  \sum_{\substack{x_1,x_2 \, (p) \\ x_1^2+1 \equiv a(x_2
^2+1) \, (p)}} e_{p}( (D/p)^{-1} \bm{h} \cdot \bm{x}) \ll p^{1/2}(h_1,h_2,p)^{1/2}(a-1,p)^{1/2}
 \]
\end{proof}
We now turn to the corresponding sum for Theorem \ref{maintheorem2}. Here we apply a slightly different argument to the above.
\begin{lemma} \label{expsum2Dprelimlemma}
Let $\bm{h}=(h_1,h_2,h_3,h_4)$. We have for $a \not \equiv 0 \, (p)$ 
\[
\sum_{\substack{x_1,x_2,
x_3,x_4 \, (p) \\ x_1^3+x_2^3 \equiv a(x_3^3+x_4^3) \, (p)}} e_p ( \bm{h} \cdot \bm{x}) \ll (h_1,h_2,h_3,h_4,p)p^2
\]
\end{lemma}
\begin{proof}
If $(h_j,p)\neq 1$ for all $j \leq 4$ then we apply the trivial bound, so we may assume that $(h_1,h_2,h_3,h_4,p) =1$. We expand the congruence condition to get
\[ \begin{split}
\sum_{\substack{x_1,x_2,
x_3,x_4 \, (p) \\ x_1^3+x_2^3 \equiv a(x_3^3+x_4^3) \, (p)}} e_p ( \bm{h} \cdot \bm{x}) &= \frac{1}{p} \sum_{w \, (p)} \sum_{\substack{x_1,x_2,
x_3,x_4 \, (p)}}  e_p ( \bm{h} \cdot \bm{x} + w( x_1^3+x_2^3 - a(x_3^3+x_4^3)))  \\
&= \frac{1}{p} \sum_{w \, (p)} \prod_{j=1}^4 \bigg(\sum_{x \, (p)} e_p (h_j x + w a_j x^3) \bigg),
\end{split}
\]
where $a_1=a_2=1$ and $a_3=a_4=-a$. The contribution from $w=0$ is 0 (as $(h_j,p)=1$ for some $j\in\{1,2,3,4\}$), and for $w \neq 0$ we apply the Weil bound (\cite[Theorem 5]{bombieri} for $X=\mathbb{P}^1$ or \cite{weil2}) to each of the four sums to get
\[
\sum_{\substack{x_1,x_2,
x_3,x_4 \, (p) \\ x_1^3+x_2^3 \equiv a(x_3^3+x_4^3) \, (p)}} e_p ( \bm{h} \cdot \bm{x}) \ll\frac{1}{p} \sum_{w \, (p)}  (p^{1/2})^4 \ll p^2.
\]
\end{proof}
\begin{remark}
It should be possible to obtain the Deligne bound $ \ll p^{3/2}$ for the generic $\bm{h}$ and $a$, but since this is not necessary for our application nor sufficient for an asymptotic version of Theorem \ref{maintheorem2}, we do not pursue this issue here.
\end{remark}
Using Lemma \ref{expsum2Dprelimlemma} we get the following by the Chinese remainder theorem.
\begin{lemma} \label{expsums2Dlemma}
Let $D_1$ be the largest square-free divisor of $D$ such that $(D_1,D/D_1)=1$. Let $(a,D)=1.$ Then 
\[
\sum_{\substack{x_1,x_2,
x_3,x_4 \, (D) \\ x_1^3+x_2^3 \equiv a(x_3^3+x_4^3) \, (D)}} e_D ( \bm{h} \cdot \bm{x}) \ll \tau(D)^{O(1)} D^3 D_1^{-1} (h_1,h_2,h_3,h_4,D_1) .
\]
\end{lemma}
\begin{proof}
    By the Chinese remainder theorem we have
 \[
\sum_{\substack{x_1,x_2,
x_3,x_4 \, (D) \\ x_1^3+x_2^3 \equiv a(x_3^3+x_4^3) \, (D)}} e_D ( \bm{h} \cdot \bm{x}) = \prod_{p^k || D} \sum_{\substack{x_1,x_2,
x_3,x_4 \, (p^k) \\ x_1^3+x_2^3 \equiv a(x_3^3+x_4^3) \, (p^k)}} e_{p^k} ((D/p^k)^{-1} \bm{h} \cdot \bm{x}),
 \]
 where $(D/p^k)^{-1}$ denotes the multiplicative inverse modulo $p^k$. For $p^k|| (D/D_1)$ we use the trivial bound
 \[
\sum_{\substack{x_1,x_2,
x_3,x_4 \, (p^k) \\ x_1^3+x_2^3 \equiv a(x_3^3+x_4^3) \, (p^k)}} e_{p^k} ((D/p^k)^{-1} \bm{h} \cdot \bm{x})\ll p^{3k}.
 \]
 For $p|D_1$ we have by Lemma \ref{expsumplemma}
 \[
\sum_{\substack{x_1,x_2,
x_3,x_4 \, (p) \\ x_1^3+x_2^3 \equiv a(x_3^3+x_4^3) \, (p)}} e_{p} ((D/p)^{-1} \bm{h} \cdot \bm{x}) \ll p^{2}(h_1,h_2,h_3,h_3,p).
 \] 
\end{proof}

\section{Proof of Proposition \ref{typeiiproper}} \label{typeiiproofsection}
We need the following simple bound.
\begin{lemma}\label{Deltaboundlemma}
Let $\Delta=\Delta(z_1,z_2):= \emph{\Im}(z_2\overline{z_1})$ for $z_1,z_2 \in \Z[i]$.  Then
\[
\sum_{\substack{|z_1|^2,|z_2|^2 \sim N \\ \Delta =D}} 1 \ll N \log^{2} N.  
\]
\end{lemma}
\begin{proof}
Writing $z_j=r_j+is_j$, we have $\Delta=r_2 s_1-r_1s_2$. Since $\Delta \neq 0$ we have $s_1 \neq 0$ or $s_2\neq 0$, so that we can assume that $s_2\neq 0$. Summing over $r_2$ with $r_2s_1 \equiv D \, (s_2)$ we have 
\begin{equation*} 
\sum_{\substack{|z_1|^2,|z_2|^2 \sim N \\ \Delta =D}} 1 \ll N^{1/2}\sum_{s_1,s_2 \leq 2 N^{1/2}}  \frac{(s_1,s_2)}{s_2} \ll N \log^{2} N. 
\end{equation*}
\end{proof}
\subsection{Application of Cauchy-Schwarz}
Since $mn$ is not divisible by any prime $p\equiv 3 \, (4)$ by the support of $a^{(j)}_{mn}$ and $b_{mn}$, we can write $m=|w|^2$ and $n=|z|^2$ for some Gaussian integers $w,z \in \Z[i]$. Since the expression $\alpha(m)\beta(n)$ is supported on $(mn,P(W))=1$, the conditions $(a,b^2+1)=1$, $(a,c^3+d^3)=1$, and $(a,b)=1$  in the definitions of $a_n^{(1)}$, $a_n^{(2)}$, and $b_n$  may be dropped with a negligible error term. 
Let
\[
\begin{split}
S^{(1)}_1 (w,z) &= \sum_{\substack{\Re (\bar{w}z) = b^2+1\\ b > 0}} b, \\
S^{(2)}_1 (w,z) &= \sum_{\substack{\Re (\bar{w}z) = c^3+d^3 \\  c,d>0}}\Omega(c,d), \\
S^{(1)}_2 (w,z)&=S^{(2)}_2 (w,z) = \bm{1}_{\Re (\bar{w}z)  > 0}.
\end{split}  
\]
Denote $\alpha_w=\alpha(|w|^2)$ and $\beta_z = \beta(|z|^2).$ Note that since $\beta(n)$ is supported on square-free numbers, $\beta_z$ is supported on $z\in \Z[i]$ with $(z, \overline{z})=1$.  Then Proposition \ref{typeiiproper} follows once we show that for $j,k \in \{1,2\}$
\[
\sum_{ \substack{|w|^2 \sim M_j \\ |z|^2 \sim N_j }} \alpha_w \beta_z S^{(j)}_k (w,z) \ll X \log^{-C} X.
\]
Let $F_{M_j}(m)=F(m/M_j)$ denote a $C^{\infty}$-smooth majorant of $m \sim M_j$, for some fixed smooth $F$ supported on $[1/2,4]$.  Applying Cauchy-Schwarz we get
\begin{equation} \label{dispersion}
\sum_{ \substack{|w|^2 \sim M_j \\ |z|^2 \sim N_j }} \alpha_w \beta_z S^{(j)}_k (w,z)\ll M_j^{1/2} (U^{(j)}_{k} )^{1/2}
\end{equation}
for
\[
U^{(j)}_{k}:= \sum_{|z_1|^2,|z_2|^2 \sim N_j} \beta_{z_1}\beta_{z_2} \sum_{w} F_{M_j}(|w|^2)  S^{(j)}_k (w,z_1) S^{(j)}_k (w,z_2).
\]
Then our task is reduced to showing that for $j,k\in \{1,2\}$ we have
\begin{equation} \label{typeiiclaimaftercs}
U^{(j)}_{k} \ll_C X N_j\log^{-C} X.
\end{equation}
Let us first restrict only to the case $j=k=1$.  Similarly as in \cite[Section 6]{FI}, let
\[
\Delta =\Delta(z_1,z_2) := \Im( \overline{z_1} z_2) = r_1s_2-r_2s_1 = |z_1z_2|\sin (\arg z_2 - \arg z_1).
\]
Let $a_i, b_i$ denote integers such that
\[
\Re(\overline{w}z_i) = (b_i^2+1) \quad \text{and} \quad \Im(\overline{w}z_i) = a_i.
\]
Then
\[
z_2\Re(\overline{w}z_1) -z_1\Re(\overline{w}z_2) = \frac{1}{2} w (z_2 \overline{z_1} -z_1\overline{z_2}) = i \Delta w,
\]
that is,
\[
i \Delta w = z_2(b_1^2+1)-z_1(b_2^2+1).
\]
We first bound separately the contribution from the diagonal terms where $\Delta=0$ or $(z_1,z_2) > 1$.
\subsection{Diagonal terms with $\Delta=0$} \label{diag1section}
By $\Delta=0$  we have
\[
z_1(b_2^2+1) =  z_2(b_1^2+1).
\]
Multiplying both sides by $\overline{w}$ and taking the imaginary parts we find
\[
a_1(b_2^2+1)=a_2 (b_1^2+1).
\]
Hence, from $\arg z_1=\arg z_2$ we get a contribution to $U_{1}^{(1)}$ of size
\[
 \sum_{\substack{a_1,a_2 \ll X^{1/2} \\ b_1,b_2 \ll X^{1/4} \\ a_1(b_2^2+1)=a_2(b_1^2+1)}} b_1b_2 \ll X^{1/2}\sum_{\ell \ll X^{3/4}} \bigg(\sum_{\ell=a(b^2+1)}1 \bigg)^{2}  \ll_\eps X^{5/4+\eps}  \ll X^{1-\eta} N_j.
\]
by using the bound $\tau(\ell) \ll_\eps \ell^\eps$ and the assumption $N_1 \gg  X^{1/4+\eta}$.  The same bound holds for $U^{(1)}_2$, so that the diagonal contribution is sufficiently small in terms of (\ref{typeiiclaimaftercs}). 

Similar arguments apply in the case $j=2$, with $b_i^2+1$ replaced by $c_i^3+d_i^3$, and by using the trivial bound $\Omega(c_j,d_j) \ll \max\{c_j,d_j\} \ll X^{1/6}$ we get a contribution
\[
\sum_{\substack{a_1,a_2 \ll X^{1/2} \\ c_1,c_2,d_1,d_2 \ll X^{1/6} \\ a_1(c_2^3+d_2^3)=a_2(c_1^3+d_1^3)}}  \Omega(c_1,d_1)\Omega(c_2,d_2)\ll X^{1/3}  \sum_{\substack{a_1 \ll X^{1/2} \\ c_2,d_2\ll X^{1/6}}}  \bigg( \sum_{\substack{a_2 \ll X^{1/2} \\ c_1,d_1 \ll X^{1/6} \\ a_1(c_2^3+d_2^3)=a_2(c_1^3+d_1^3)}} 1 \bigg)^{2} 
\]
Using Lemma \ref{cubereplemma} to bound the number of representations as a sum of two cubes we get
\[
\sum_{\substack{a_1,a_2 \ll X^{1/2} \\ c_1,c_2,d_1,d_2 \ll X^{1/6} \\ a_1(c_2^3+d_2^3)=a_2(c_1^3+d_1^3)}}  \Omega(c_1,d_1)\Omega(c_2,d_2) \ll_\eps X^{7/6+\eps} \ll X^{1-\eta} N_2
\]
since $N_2 \gg X^{1/6+\eta}$. 

\subsection{Diagonal terms with $(z_1,z_2) > 1$} \label{diag2section}
Consider first the case $j=1$. Recall that, by assumptions in Proposition \ref{typeiiproper}, $\beta(n)$ is supported on $n$ with $(n,P(W))=1$ with $W=X^{1/(\log\log X)^2}$. Therefore, if $z_0:=(z_1,z_2) > 1$, then $|z_0|^2 \geq W$. Making the change of variables $z_i \mapsto z_0z_i$, we have
\[
\Re(\overline{w}z_0z_i) = (b_i^2+1).
\]
Denoting $\Delta'= \Im(\overline{z_1}z_2) \neq 0$ and combining $w'=w \overline{z_0}$ we have
\[
i \Delta' w' = z_2(b_1^2+1)-z_1(b_2^2+1) \neq 0,
\]
so that fixing $z_1,z_2,b_1,b_2$ fixes $w'$. We have
\[
\sum_{\substack{W \leq |z_0|^2 \ll N_1\\ (|z_0|^2,P(W))=1} } \bm{1}_{\overline{z_0} | w} \ll  2^{\log X/\log W} \ll W^{1/2}.
\]
Hence, by a dyadic decomposition we see that the part  $(z_1,z_2) > 1$ contributes at most
\[
\begin{split}
&\ll \sum_{\substack{W \leq |z_0|^2 \ll N_1\\ (|z_0|^2,P(W))=1} } \sum_{|w'|^2  \sim M_1 |z_0|^2} \sum_{\substack{|z_1|^2,|z_2|^2 \sim N_1/|z_0|^2 \\ (z_1,z_2)=1}} S^{(1)}_1 (w',z_1)S^{(1)}_1 (w',z_2) \\
& \ll X^{1/2}W^{1/2}(\log X )\sup_{W \leq P \leq N_1} \sum_{b_1,b_2 \ll X^{1/4}} \sum_{\substack{|z_1|^2,|z_2|^2 \sim N_1/P \\ (z_1,z_2)=1}} \bm{1}_{\Delta' | z_2(b_1^2+1)-z_1(b_2^2+1) \neq 0}.
\end{split}
\]
Let $z_i:=r_i+is_i$. Then we have 
\[
\Delta' | s_2(b_1^2+1)-s_1(b_2^2+1) \quad \text{and} \quad \Delta' | r_2(b_1^2+1)-r_1(b_2^2+1) 
\]
and $z_2(b_1^2+1)-z_1(b_2^2+1) \neq 0$ so that by symmetry we may assume $s_2(b_1^2+1)-s_1(b_2^2+1) \neq 0$ and thus by symmetry we may also assume that $s_2 \neq 0$. For a given $\Delta',s_1,s_2$ we have 
\[
\sum_{\substack{ r_1,r_2 \ll N_1^{1/2}/P^{1/2}\\r_1s_2-r_2s_1= \Delta'}} 1 \ll \sum_{\substack{ r_2 \ll N_1^{1/2}/P^{1/2}\\ -r_2s_1= \Delta' \, (s_2)}} \ll \frac{N_1^{1/2}(s_1,s_2)}{P^{1/2} s_2}. 
\]
Hence, denoting $s_0=(s_1,s_2)$ and $s_i'=s_0 s_i$, we get
\[
\begin{split}
\ll \frac{X^{1/2}N_1^{1/2} W^{1/2}\log X }{P^{1/2}} \sum_{s_0 \ll N_1^{1/2}} &\tau(s_0) \sum_{\substack{s_1',s_2' \ll N_1^{1/2}P^{-1/2}/s_0 \\ (s_1',s_2')=1 \\ s_2' \neq 0}}\frac{1}{s_2'} \\
& \sum_{\substack{b_1,b_2 \ll X^{1/4} \\ s'_2(b_1^2+1)-s'_1(b_2^2+1) \neq 0}} \tau( s'_2(b_1^2+1)-s'_1(b_2^2+1)).
\end{split}
\]
By Lemma \ref{divisorlemma} there exists $d | s'_2(b_1^2+1)-s'_1(b_2^2+1)$ such that $d \ll X^{\eps}$ and \[\tau( s'_2(b_1^2+1)-s'_1(b_2^2+1)) \ll \tau(d)^{O_\eps(1)}.\] Thus, the last expression is bounded by
\[
\begin{split}
&\ll\frac{X^{1/2}N_1^{1/2} W^{1/2}\log X }{P^{1/2}}\sum_{s_0 \ll N_1^{1/2}} \tau(s_0) \sum_{d \ll X^{\eps}} \tau(d)^{O_\eps(1)} \\
& \hspace{140pt} \sum_{\substack{s_1',s_2' \ll N_1^{1/2}P^{-1/2}/s_0 \\ (s_1',s_2')=1 \\ s_2' \neq 0}} \frac{1}{s_2'} \sum_{\substack{b_1,b_2 \ll X^{1/4} \\ s'_2(b_1^2+1)-s'_1(b_2^2+1) \equiv 0 \, (d)}} 1 \\ 
&\ll \frac{X N_1^{1/2} W^{1/2}\log X }{P^{1/2}}  \sum_{s_0 \ll N_1^{1/2}} \tau(s_0) \sum_{d \ll X^{\eps}} \frac{\tau(d)^{O_\eps(1)} }{d} \sum_{\substack{s_1',s_2' \ll N_1^{1/2}P^{-1/2}/s_0 \\   s_2' \neq 0}}\frac{(d,s_2')}{s_2'} 
\\ 
&\ll \frac{X N_1^{1/2} W^{1/2}\log^{O(1)} X }{P^{1/2}} \sum_{s_0 \ll N_1^{1/2}} \tau(s_0)  \sum_{\substack{s_1',s_2' \ll N_1^{1/2}P^{-1/2}/s_0 \\   s_2' \neq 0}} \frac{\tau(s_2')^{O_\eps(1)}}{s_2'} \\
&\ll W^{-1/2} X N_1 \log^{O(1)} X  \ll_C X N_1 \log^{-C} X,
\end{split}
\]
where we have used the bounds
\[
\sum_{\substack{b_1,b_2 \ll X^{1/4} \\ s'_2(b_1^2+1)-s'_1(b_2^2+1) \equiv 0 \, (d)}} 1  \ll X^{1/2} (d,s_2') /d 
\]
and
\[
\sum_{d \sim D} \tau(d)^{O(1)} (d,s)  \leq \sum_{v| s} \sum_{\substack{d \sim D \\ d \equiv 0\, (v)}} \tau(d)^{O(1)}  \ll \tau(s)^{O(1)} D \log^{O(1)} D.
\]
Similar arguments again apply when $j=2$ or $k=2$.
\subsection{The off-diagonal for $U^{(1)}_{1}$} \label{offdiagsection}
Let $U^{(1)}_{11}$ denote the part of $U^{(1)}_{1}$ where $\Delta \neq 0$ and $(z_1,z_2)=1$, so that we have 
\[(z_1,z_2)=(z_1,\overline{z_1})=(z_2,\overline{z_2})=(\Delta,|z_2z_1|^2)=1.\] 
To see the last equality, suppose that $p|(\Delta,|z_1|^2)$. Then using $(z_1,\overline{z_1})$=1 we have $p \equiv 1 \, (4)$ so that $p=\pi \bar{\pi}$ for some Gaussian prime $\pi | z_1$. By $(z_1,z_2)=1$ we have $\pi \nmid z_2$, so that $\bar{\pi} | z_2 \overline{z_1}$ but $p \nmid  z_2 \overline{z_1}$. This is a contradiction, since $p | \Im(z_2\overline{z_1})$ and $\bar{\pi} | z_2 \overline{z_1}$ imply that $p| \Re(z_2 \overline{z_1})$ and thus $p| z_2 \overline{z_1}$. Similarly we see that $(\Delta,|z_2|^2)=1$.

In particular, we have
\[
1 \leq |\Delta| \ll N_1.
\]
Recall that
\[
i \Delta w = z_2(b_1^2+1)-z_1(b_2^2+1).
\]
That is, $w$ is fixed for given $z_1,z_2,b_1,b_2,$ and we have
\[
z_2(b_1^2+1)\equiv z_1(b_2^2+1) \,\, (|\Delta|).
\]
Note that ($\overline{z_1 }$ denoting the complex conjugate)
\[
a = a(z_1,z_2):= z_2/z_1 = \frac{\overline{z_1 }z_2}{|z_1|^2} \equiv \frac{r_1r_2+s_1s_2}{r_1^2+s_1^2} \, (|\Delta|)
\]
is a congruence in $\Z$ despite the fact that $z_1,z_2 \in \Z[i]$, and note that $(a,\Delta)=1$. Hence, we have
\[
U^{(1)}_{11} = \sum_{\substack{|z_1|^2,|z_2|^2  \sim N_1 \\ \Delta \neq 0 \\ (z_1,z_2)=1}} \beta_{z_1}\beta_{z_2} \sum_{\substack{b_1,b_2 \\ b_2^2+1 \equiv a(b_1^2+1) \, (\Delta)}} b_1b_2 F_{M_1} \bigg( \bigg| \frac{z_2(b_1^2+1)-z_1(b_2^2+1)}{\Delta} \bigg|^2 \bigg).
\]
Let $\Delta_2$ denote the largest powerful divisor of $\Delta$ and let $Y=(\log X)^{C'}$ for some large $C'>0$.  Note that $(a-1,\Delta)=(z_2-z_1,\Delta).$ We now wish to discard the part where $\Delta_2 > Y$ or $(z_2-z_1,\Delta) > Y$ and  write
\[
U_{11}^{(1)}= U_{11\leq}^{(1)} +  U_{11>}^{(1)}
\]
with
\begin{align*}
    U_{11\leq }^{(1)} := \sum_{\substack{|z_1|^2,|z_2|^2  \sim N_1 \\ \Delta \neq 0 \\ (z_1,z_2)=1 \\ \max\{\Delta_2,(z_2-z_1,\Delta)\} \leq Y }} \beta_{z_1}\beta_{z_2} \sum_{\substack{b_1,b_2 \\ b_2^2+1 \equiv a(b_1^2+1) \, (|\Delta|)}} b_1b_2 F_{M_1} \bigg( \bigg| \frac{z_2(b_1^2+1)-z_1(b_2^2+1)}{\Delta} \bigg|^2 \bigg), \\
        U_{11 > }^{(1)} := \sum_{\substack{|z_1|^2,|z_2|^2  \sim N_1 \\ \Delta \neq 0 \\ (z_1,z_2)=1 \\ \max\{\Delta_2,(z_2-z_1,\Delta)\} > Y}} \beta_{z_1}\beta_{z_2} \sum_{\substack{b_1,b_2 \\ b_2^2+1 \equiv a(b_1^2+1) \, (|\Delta|)}} b_1 b_2 F_{M_1} \bigg( \bigg| \frac{z_2(b_1^2+1)-z_1(b_2^2+1)}{\Delta} \bigg|^2 \bigg).
\end{align*}
We will show in Section \ref{boundinglargeysection} by crude estimates that
\begin{align} \label{baddeltaerror}
 U_{11 > }^{(1)} \ll _C XN_1 \log^{-C} X.   
\end{align}
For $U_{11\leq }^{(1)}$ we futher separate terms near the diagonal by writing
\[
U^{(1)}_{11 \leq} = U^{(1)}_{111} + U^{(1)}_{112},
\]
where $U^{(1)}_{112}$ is the part where for $\delta=\log^{-C_1} X$ with some large $C_1 > 0$
\[
\arg z_1 = \arg z_2 + O(\sqrt{\delta}) \pmod{\pi}.
\]
We will show  in Section \ref{smalldeltasection} that for $C_1$ large enough in terms of $C$ we have
\begin{equation} \label{pseudodiagerror}
U^{(1)}_{112} \ll_C X N_1 \log^{-C} X.
\end{equation}
For now we consider $U^{(1)}_{111}$, where $|\arg z_1 - \arg z_2 \pmod{\pi}| > \sqrt{\delta}$.
To simplify the evaluation we split the sums smoothly into finer-than-dyadic intervals. Let $\delta=\log^{-C_1} X$ be as above and let $\psi$ be a $C^\infty$-smooth function as in Section \ref{smoothweightsection}. We will apply the procedure of Section \ref{smoothweightsection} six times, to introduce smooth weights for the variables $z_1,z_2,\arg z_1, \arg z_2, b_1,b_2$, which will cost us a factor of  $\delta^{-6} \log^2 X$. 
For any $Y\geq 1$ let $\psi_Y(y):= \psi(y/Y)$. Let $N_{11 },N_{12}\asymp N_1$, $\theta_1,\theta_2 \in [\pi,5\pi)$, and $B_{1},B_{2} \ll X^{1/4}$ and denote 
\[
\beta^{(\ell)}_{z} := \beta_{z} \psi_{N_{1\ell}}(|z|^2)\psi_{\theta_\ell}(\arg z + 2\pi) .
\] 
For $z_1,z_2$ such that $\beta^{(1)}_{z_1}\beta^{(2)}_{z_2} \neq 0$ we have using  $|\theta_1-\theta_2 \pmod{\pi}| \gg \sqrt{\delta}$
\[
|\Delta| = |z_1z_2||\sin (\arg z_2 - \arg z_1)| = (N_{11}N_{12})^{1/2}|\sin (\theta_2 - \theta_1)|(1 + O(\delta)) \gg  \sqrt{\delta} N .
\]
Hence, for any fixed tuple $(\bm{B},\bm{N},\bm{\theta})$ there is a constant $F(\bm{B},\bm{N},\bm{\theta})$ such that for all $z_i,b_i$ in the support of $\beta^{(1)}_{z_1}\beta^{(2)}_{z_2}\psi_{B_1}(b_1)\psi_{B_2}(b_2)$ we have (using $|\Delta| \gg \sqrt{\delta} N $)
\begin{equation} \label{ccremoval}
b_1 b_2 F_{M_1} \bigg( \bigg| \frac{z_2(b_1^2+1)-z_1(b_2^2+1)}{\Delta} \bigg|^2 \bigg) = B_1 B_2(F(\bm{B},\bm{N},\bm{\theta}) + O(\sqrt{\delta})).
\end{equation}
Accordingly we write 
\[
U^{(1)}_{111} =  U^{(1)}_{1111} + O(\sqrt{\delta} U^{(1)}_{1112}),
\]
where in $U^{(1)}_{1112}$ we have used the triangle inequality to put absolute values around $\beta_{z}$. We will show in Section \ref{ccremovalsubsection} that 
\begin{equation} \label{ccerrorbound}
U^{(1)}_{1112} \ll XN_1 \log^{O(1)} X, 
\end{equation}
so that by taking $C_1$ to be large enough in terms of $C$ we get
\[
\sqrt{\delta} U^{(1)}_{1112} \ll_C XN_1 \log^{-C} X.
\]
Before this we consider the main term $U^{(1)}_{1111}$, for which we need to bound sums of the form
\[
U_1^{(1)}(\bm{B},\bm{N},\bm{\theta}) :=  B_1B_2\sum_{\substack{z_1,z_2 \\ \Delta \neq 0 \\ (z_1,z_2)=1 \\ \max\{\Delta_2,(z_2-z_1,\Delta)\} \leq Y}} \beta^{(1)}_{z_1}\beta^{(2)}_{z_2} \sum_{\substack{b_1,b_2 \\ b_2^2+1 \equiv a(b_1^2+1) \, (|\Delta|)}} \psi_{B_1}(b_1)\psi_{B_2}(b_2).
\]
Since the cost from introducing smooth weights is bounded by $\ll \delta^6 \log^{6} X= \log^{O(1)} X$, to handle $U^{(1)}_{1111}$ it suffices to show that for any $C > 0$
\begin{equation} \label{umainclaim}
U_1^{(1)}(\bm{B},\bm{N},\bm{\theta}) \ll_C X N_1 \log^{-C} X.
\end{equation}

Let $H_j := X^\eps|\Delta|/B_j$ and define 
\[
\begin{split}
N_1(a; \Delta) &:= |\{x_1,x_2 \, (\Delta): x_1^2+1 \equiv a(x_2
^2+1) \, (|\Delta|)\}|, \\
S_1(a,\bm{h};|\Delta|)&:= \sum_{\substack{x_1,x_2 \, (\Delta) \\ x_1^2+1 \equiv a(x_2
^2+1) \, (|\Delta|)}} e_{|\Delta|}(\bm{h} \cdot \bm{x}),
\end{split}
\]
and
\[
\hat{U}_1^{(1)}(\bm{B},\bm{N},\bm{\theta})  := B_1 B_2 \sum_{\substack{z_1,z_2 \\ \Delta \neq 0 \\ (z_1,z_2)=1 \\\max\{\Delta_2,(z_2-z_1,\Delta)\} \leq Y}} |\beta^{(1)}_{z_1}\beta^{(2)}_{z_2} |\frac{1}{H_1 H_2} \sum_{\substack{|h_j| \ll H_j \\ (h_1,h_2) \neq (0,0)}} | S_1(a,\bm{h};\Delta)|.
\]
Then applying the Poisson summation formula (Lemma \ref{poissonlemma}) twice we get
\[
U_1^{(1)}(\bm{B},\bm{N},\bm{\theta}) = B_1^2B_2^2 \sum_{\substack{z_1,z_2 \\ \Delta \neq 0 \\ (z_1,z_2)=1 \\ \max\{\Delta_2,(z_2-z_1,\Delta)\} \leq Y}} \beta^{(1)}_{z_1}\beta^{(2)}_{z_2} \frac{\hat{\psi}(0)^2}{\Delta^2} N_1(a; \Delta) + O(X^{2 \eps}\hat{U}_1^{(1)}(\bm{B},\bm{N},\bm{\theta})) + O_\eps(X^{-10}).
\]
Note that here $\hat{\psi}(0) \ll \delta$.

\subsubsection{Bounding the main term} \label{maintermsection}
We have to show that for any $C >0$
\begin{align} \label{maintermclaim}
    \delta^2 B_1^2B_2^2 \sum_{\substack{z_1,z_2  \\ \Delta \neq 0 \\ (z_1,z_2)=1 \\\max\{\Delta_2,(z_2-z_1,\Delta)\} \leq Y}} \beta^{(1)}_{z_1}\beta^{(2)}_{z_2} \frac{1}{\Delta^2} N_1(a; \Delta)  \ll_C X N \log^{-C} X.
\end{align}
Recall that $(a,\Delta)=1$ for all $z_1,z_2$ appearing in the above sum and note that $(a-1,\Delta)= (z_2-z_1,\Delta).$ Using Lemma \ref{points1Dlemma} with $Y=\log^{C'} X$ as above, noting that $\max\{\Delta_2,(z_2-z_1,\Delta)\} \leq Y$, the left-hand side in \eqref{maintermclaim} is bounded by $T_1 + O(T_2) $ for
\[
\begin{split}
T_1 &:=  \delta^2B_1^2B_2^2 \sum_{\substack{z_1,z_2  \\ \Delta \neq 0 \\ (z_1,z_2)=1 \\ \max\{\Delta_2,(z_2-z_1,\Delta)\} \leq Y}} \beta^{(1)}_{z_1}\beta^{(2)}_{z_2} \frac{1}{\Delta} \sum_{\substack{d|\Delta \\ d \leq Y^4 }} \frac{\eps_d(a)}{d}, \\
T_2 &:= Y^{-1}\delta^2 B_1^2B_2^2 \sum_{\substack{z_1,z_2  \\ \Delta \neq 0\\ (z_1,z_2)=1}}|\beta^{(1)}_{z_1}\beta^{(2)}_{z_2} |\frac{1}{\Delta}\tau(\Delta)^{O(1)}  \\
\end{split}
\]
Using Lemma \ref{Deltaboundlemma} we see (taking $Y$ to be sufficiently large power of $\log X$)
\[
T_2 \ll_C B_1^2 B_2^2 N_1 \log^{-C} X \ll_C X N_1 \log^{-C} X.
\]
In $T_1$ the condition $d|\Delta$ is equivalent to $z_2 \equiv a z_1 \, (d)$ for some $a\,(d)$, so that
\[
T_1 = \delta^2 B_1^2 B_2^2 \sum_{d \leq Y^4} \sum_{(a,d)=1} \frac{\eps_d(a)}{d} \sum_{\substack{z_2 \equiv az_1 \,(d)\\ \Delta \neq 0 \\ (z_1,z_2)=1 \\ \max\{\Delta_2,(z_2-z_1,\Delta)\} \leq Y }} \beta^{(1)}_{z_1}\beta^{(2)}_{z_2} \frac{1}{\Delta}.
\]
The condition  $(z_1,z_2)=1$ can be dropped with a negligible error term by trivial bounds, recalling that $\beta(n)$ is supported on $(n,P(W))=1$. Similarly, the condition $\max\{\Delta_2,(z_2-z_1,\Delta)\} \leq Y$ may also be dropped with a negligible error term by crude bounds.

Recall that in the support of $\beta^{(1)}_{z_1}\beta^{(2)}_{z_2}$ we have for some constant $\sqrt{\delta} N_1 \ll D \leq 2N_1$
\[
\Delta = D(1+O(\sqrt{\delta})).
\]
The contribution from $O(\sqrt{\delta})$ is bounded by trivial bounds, recalling that  $\beta^{(\ell)}_{z}$ are supported on small polar boxes.  Applying \eqref{swmain0} (the Siegel-Walfisz property) we get 
\[
T_1 \ll_C B_1^2 B_2^2 N_1 (\delta^{-1/2}\log^{-C} X +  \delta^6 \sqrt{\delta} )  \ll_C X  N_1 (\log^{-C} X+ \delta^6 \sqrt{\delta}),
\]
which is sufficient for (\ref{umainclaim}) since the cost from introducing the smooth weights was $\delta^{-6}$.
\begin{remark}
At first it may seem surprising that in our problem the evaluation of the main term is vastly easier than in the situation in \cite{FI}. This is because here with $N(a;\Delta)$ we are computing solutions to a non-singular equation. This allowed us to bound the large moduli $d$ by using the Hasse-Weil bound (Lemma \ref{pointsweil1lemma}). 
\end{remark}
\subsubsection{Bounding the non-zero frequencies} \label{nonzerosection}
It suffices to show that for $N_1 \ll X^{1/3-\eta}$ we have for some $\eta >0$
\[
\hat{U}^{(1)}_1(\bm{B},\bm{N},,\bm{\theta}) \ll X^{1-\eta} N_1 .
\]
For $h_1 = 0$ we apply Lemma \ref{expsums1Dlemma} with $h_2 \neq 0$ to get a bound
\begin{align*}
 B_1 B_2& \sum_{\substack{z_1,z_2 \\ \max\{\Delta_2,(z_2-z_1,\Delta)\} \leq Y}}  \frac{1}{H_1 H_2} \sum_{\substack{0<|h_2| \ll H_2 }} | S_1(a,0,h_2;\Delta)| \\
 &\ll  Y^{O(1)}B_1 B_2 \sum_{\substack{z_1,z_2 }} \frac{|\Delta| \tau(\Delta)^{O(1)}}{\Delta^{1/2}} \frac{1}{H_1 H_2} \sum_{\substack{0<|h_2| \ll H_2 }} (h_2,\Delta_1)    
\end{align*}

Using $H_j=|\Delta|/B_j$ and Lemma \ref{Deltaboundlemma} the last expression is bounded by
\[
\ll  Y^{O(1)} \sum_{z_1,z_2 } \frac{|\Delta| \tau(\Delta)^{O(1)}}{\Delta^{1/2}} \frac{B_1 B_2}{H_1}
 \ll Y^{O(1)}B_1^2 B_2   \sum_{z_1,z_2 } \frac{\tau(\Delta)^{O(1)}}{\Delta^{1/2}} \ll Y^{O(1)} B_1^2 B_2 N_1^{3/2} \ll X^{1-\eta} N_1,
\]
since $B_j \ll X^{1/4}$ and $N \ll X^{1/3-\eta} < X^{1/2-\eta}$.
Similarly, we get a sufficient bound for the contribution from $h_2=0$ with $h_1 \neq 0$.
 
For the contribution from $(h_1,h_2) \neq  \bm{0}$ we get by Lemma 8 a bound
\[
\begin{split}
\ll Y^{O(1)} \frac{B_1 B_2}{H_1H_2} \sum_{z_1,z_2}\tau(\Delta)^{O(1)} \frac{|\Delta|}{\Delta^{1/2}} \sum_{0 < |h_j| \ll H_j} ((h_1,\Delta_1) +(h_1+h_2,\Delta_1)^{1/2}+(h_1-h_2,\Delta_1)^{1/2} ) \\
 \ll Y^{O(1)} B_1 B_2 \sum_{z_1,z_2} \tau(\Delta)^{O(1)}\bigg(\frac{|\Delta|}{\Delta^{1/2}}  + \frac{|\Delta|}{H_1} \bigg) \ll B_1 B_2 \sum_{z_1,z_2} \tau(\Delta)^{O(1)}\bigg(\frac{|\Delta|}{\Delta^{1/2}}  + B_1 \bigg)
\end{split}.
\]
By Lemma \ref{Deltaboundlemma}  the last expression is at most
\[
\ll X^{\eps}(B_1 B_2 N_1^{5/2} + B_1^2 B_2 N_1^{3/2}) \ll X^\eps( X^{1/2} N_1^{5/2} + X^{3/4} N_1^{3/2}) \ll X^{1-\eta} N_1
\]
by using $N_1  \ll X^{1/3-\eta}$.
\begin{remark} Morally speaking we  have in the above applied the bound $S_1(a;\Delta) \ll \Delta^{1/2}$. The corresponding exponential sum in \cite{FI} is bounded in terms of Ramanujan sums, so that in there one morally gets a bound $\ll 1$ for the exponential sums. This loss is the reason why our Type II range is narrower than in \cite{FI}.
\end{remark}
\subsubsection{Bounding the error from removing cross-conditions} \label{ccremovalsubsection}
 We now explain how to get (\ref{ccerrorbound}). To bound $U^{(1)}_{1112}$ we consider
\[
V_1^{(1)}(\bm{B},\bm{N},\bm{\theta}) := B_1B_2\sum_{\substack{|z_1|^2,|z_2|^2 \sim N_1 \\ \Delta \neq 0 \\ (z_1,z_2)=1 \\ \max\{\Delta_2,(z_2-z_1,\Delta)\} \leq Y}} |\beta^{(1)}_{z_1}||\beta^{(2)}_{z_2}| \sum_{\substack{b_1,b_2 \\ b_2^2+1 \equiv a(b_1^2+1) \, (|\Delta|)}} \psi_{B_1}(b_1)\psi_{B_2}(b_2)
\]
For fixed $\bm{B}$ and $\bm{N}$ we can write
\[
\begin{split}
V_1^{(1)}(\bm{B},\bm{N},\bm{\theta}) \ll  B_1B_2\sum_{\substack{|z_1|^2,|z_2|^2 \sim N_1 \\ \Delta \neq 0 \\ (z_1,z_2)=1\\ \max\{\Delta_2,(z_2-z_1,\Delta)\} \leq Y}}  \psi_{N_1}(|z_1|^2) \psi_{N_1}(|z_2|^2)&\psi_{\theta_1}(\arg z_1 + 2\pi) \psi_{\theta_2}(\arg z_2 + 2\pi)  \\
& \sum_{\substack{b_1,b_2 \\ b_2^2+1 \equiv a(b_1^2+1) \, (|\Delta|)}} \psi_{B_1}(b_1)\psi_{B_2}(b_2).
\end{split}
\]
Then applying Poisson summation and carrying out the same arguments as in Sections \ref{maintermsection} and \ref{nonzerosection}, replacing the use of the Siegel-Walfisz property with $\beta^{(1)}_{z_1}\beta^{(2)}_{z_2}$ by a trivial point count with $z_1$ and $z_2$ in small polar boxes, we can evaluate the sum to obtain for any $C>0$
\[
V_1^{(1)}(\bm{B},\bm{N},\bm{\theta}) \ll_C (\delta^6 X N_1 \log^{O(1)}  X + XN_1 \log^{-C} X).
\]
Since the integration over the tuple $(\bm{B},\bm{N},\bm{\theta})$ is weighted by $\delta^{-6} \log^{O(1)} X$ (as in Section \ref{smoothweightsection}), we get
\[ 
U^{(1)}_{1112}  \ll_C XN_1\log^{O(1)} X.
\]

\subsubsection{Bounding the contribution from $\arg z_1 = \arg z_2 + O(\sqrt{\delta}) \pmod{\pi}$} \label{smalldeltasection}
To complete the proof for $a^{(1)}_n$ we need to show the bound (\ref{pseudodiagerror}) for the contribution from $U^{(1)}_{112}$, where  $\arg z_1 = \arg z_2 + O(\sqrt{\delta}) \pmod{\pi}$. The idea is similar as in the previous section but we cannot ignore the weight
\[
F_{M_1} \bigg( \bigg| \frac{z_2(b_1^2+1)-z_1(b_2^2+1)}{\Delta} \bigg|^2 \bigg),
\] 
which now restricts the variables $b_j$ to a narrow subset. We split $U^{(1)}_{112}$ dyadically according to the size of $\Delta \sim \delta_1 N $ to get
\[
U^{(1)}_{112} \ll \sum_{\delta_1=2^{-j} \ll \sqrt{\delta}} U(\delta_1),
\]
where
\[
U(\delta_1) =  \sum_{\substack{|z_1|^2,|z_2|^2  \sim N_1 \\ (z_1,z_2)=1 \\ \max\{\Delta_2,(z_2-z_1,\Delta)\} \leq Y\\ \Delta \sim \delta_1 N_1}} |\beta_{z_1}||\beta_{z_2}| \sum_{\substack{b_1,b_2 \\ b_2^2+1 \equiv a(b_1^2+1) \, (\Delta)}} b_1b_2 F_{M_1} \bigg( \bigg| \frac{z_2(b_1^2+1)-z_1(b_2^2+1)}{\Delta} \bigg|^2 \bigg)
\]
We introduce finer-than-dyadic weights for $b_2$ using Section \ref{smoothweightsection} with $\delta_1$, denoting $\psi=\psi^{(\delta_1)}$, to get
\[
U(\delta_1) \ll  \int_{1/2}^{2 X^{1/4}}  | V(\delta_1,B_1)| \frac{d B_1}{B_1}, 
\]
where
\[
V(\delta_1,B_1)  = \frac{1}{\delta_1} \sum_{\substack{|z_1|^2,|z_2|^2  \sim N_1 \\ (z_1,z_2)=1 \\ \max\{\Delta_2,(z_2-z_1,\Delta)\} \leq Y \\ \Delta \sim \delta_1 N_1}} |\beta_{z_1}||\beta_{z_2}| \sum_{\substack{b_1,b_2 \\ b_2^2+1 \equiv a(b_1^2+1) \, (\Delta)}} b_1b_2\psi^{(\delta_1)}_{B_1}(b_1) F_{M_1} \bigg( \bigg| \frac{z_2(b_1^2+1)-z_1(b_2^2+1)}{\Delta} \bigg|^2 \bigg).
\]
The weight $F_{M_1}$ is supported on
\[
| b_2^2+1 -   z_2/z_1(b_1^2+1)| \ll \delta_1 M_1^{1/2} N_1^{1/2}  = \delta_1 X^{1/2},
\]
Denoting $B_2^2 :=  -1+z_2/z_1(B_1^2+1)$, for $b_1=B_1(1+O(\delta_1))$ this is contained in
\[
|b_2^2 - B_2^2| \, \ll \delta_1 X^{1/2} + \delta_1 B_1^2 \ll  \delta_1 X^{1/2},
\]
which is contained in
\[
|b_2-B_2| \,\ll \delta_1 X^{1/2}/ B_1.
\]
Thus, if we let $\psi^{(\delta_2)}_{B_2}(b_2)$ denote a suitable smooth majorant to this with
\[
\delta_2 = \min \bigg\{ 1, \delta_1 \frac{X^{1/2}}{B_1 B_2}\bigg\},
\]
we get
\[
V(\delta_1,B_1)  \ll \frac{1}{\delta_1} \sum_{\substack{|z_1|^2,|z_2|^2  \sim N_1 \\ (z_1,z_2)=1 \\ \max\{\Delta_2,(z_2-z_1,\Delta)\} \leq Y\\ \Delta \sim \delta_1 N_1}} |\beta_{z_1}||\beta_{z_2}| B_1 B_2 \sum_{\substack{b_1,b_2 \\ b_2^2+1 \equiv a(b_1^2+1) \, (\Delta)}} \psi^{(\delta_1)}_{B_1}(b_1)  \psi^{(\delta_2)}_{B_2}(b_2).
\]
We now evaluate the inner sum by Poisson summation (Lemma \ref{poissonlemma}) to get
\[
V(\delta_1,B_1) = \frac{1}{\delta_1} \sum_{\substack{|z_1|,|z_2|  \\ (z_1,z_2)=1 \\ \Delta \sim \delta_1 N_1}} |\beta_{z_1}||\beta_{z_2}|  B_1^2B_2^2\frac{\hat{\psi}^{(\delta_1)}(0)\hat{\psi}^{(\delta_2)}(0)}{\Delta^2} N_1(a; \Delta) + X^{2 \eps}\hat{V} + O_\eps(X^{-10})
\] 
The non-zero frequencies $\hat{V}$ are bounded just as before, making use of bound
\[
\widehat{\psi}^{(\delta_j)} \left( \frac{h B_1}{|\Delta|}\right) \ll \delta_j
\]
to cancel the factor $\frac{1}{\delta_1 \delta_2}$ and $\Delta \sim \delta_1 N_1$ to cancel the factor $\frac{1}{\delta_1}$. We get similarly as in Section \ref{nonzerosection}
\[
\hat{V} \ll X^{\eps}(B_1 B_2 N_1^{5/2} + B_1^2 B_2 N^{3/2}) \ll X^\eps( X^{1/2} N^{5/2} + X^{3/4} N_1^{3/2}) \ll X^{1-\eta} N_1
\]
by using $N_1  \ll X^{1/3-\eta}$.

For the main term we get using the trivial bound  $N_1(a; \Delta) \ll \tau(\Delta)^{O(1)} |\Delta|$
\[
\frac{1}{\delta_1} \sum_{\substack{|z_1|,|z_2|  \\ (z_1,z_2)=1 \\ \Delta \sim \delta_1 N_1}} |\beta_{z_1}||\beta_{z_2}| B_1^2B_2^2\frac{\hat{\psi}^{(\delta_1)}(0)\hat{\psi}^{(\delta_2)}(0)}{\Delta^2} N_1(a; \Delta)  \ll \delta_1 X N_1^2 \ll \sqrt{\delta} X N_1^2,
\]
which is $\ll XN_1 \log^{-C} X$ once we choose a sufficiently large $C_1$ in $\delta= \log^{-C_1} X$.
\begin{remark}
When we split $b_1$ into intervals of length $\delta_1 B_1$ it can happen that $\delta_1 B_1 < 1$ for small $\delta_1$, so that the sum over $b_1$  may be empty. This is not a problem since we are not trying to show that the error term $\hat{V}$ from the Poisson summation is smaller than the main term, only that both terms are smaller than $\ll XN_1 \log^{-C} X.$
\end{remark}
\subsubsection{Bounding $U^{(1)}_{11>}$}
\label{boundinglargeysection} Recall that we are aiming to show \eqref{baddeltaerror} for
\begin{align} \nonumber
    U_{11 > }^{(1)} = &\sum_{\substack{|z_1|^2,|z_2|^2  \sim N_1 \\ \Delta \neq 0 \\ (z_1,z_2)=1 \\ \max\{\Delta_2,(z_2-z_1,\Delta)\} > Y}} \beta_{z_1}\beta_{z_2} \sum_{\substack{b_1,b_2 \\ b_2^2+1 \equiv a(b_1^2+1) \, (|\Delta|)}} b_1 b_2 F_{M_1} \bigg( \bigg| \frac{z_2(b_1^2+1)-z_1(b_2^2+1)}{\Delta} \bigg|^2 \bigg)\\ \nonumber
    \ll & \,X^{1/2}\sum_{\substack{|z_1|^2,|z_2|^2  \sim N_1 \\ \Delta \neq 0  \\ (z_1,z_2)=1 \\ \Delta_2 > Y}}  \sum_{\substack{b_1,b_2 \ll X^{1/4} \\ z_1( b_2^2+1) \equiv z_2(b_1^2+1) \, (|\Delta|)}} 1  + X^{1/2}\sum_{\substack{|z_1|^2,|z_2|^2  \sim N_1 \\ \Delta \neq 0  \\ (z_1,z_2)=1 \\ (z_2-z_1,\Delta) > Y}}  \sum_{\substack{b_1,b_2 \ll X^{1/4} \\ z_1 (b_2^2+1) \equiv z_2(b_1^2+1) \, (|\Delta|)}} 1  \\
\nonumber =:& V^{(1)}_{11>} + W^{(1)}_{11>},
\end{align}
where $\Delta_2$ is the powerful part of $\Delta$. For both of the terms we will apply a similar argument as in Section \ref{diag2section}. 
We first consider $W^{(1)}_{11>}$, which is the more difficult of the two. Denoting $z_j = r_j+is_j$, since $\Delta = r_1s_2-r_2s_1\neq 0$, by symmetry we may assume that $s_1 \neq s_2$ and thus by symmetry we can further assume that $s_2\neq 0$.  Let $D_0 | \Delta$ stand for $(z_2-z_1,\Delta)$, so that we have $D_0 | s_2-s_1$ and $D_0|r_2-r_1$. Hence,
\begin{align} \label{w11first}
   W_{11 > }^{(1)}  \ll  X^{1/2}  \sum_{D_0 > Y} \sum_{b_1,b_2 \ll X^{1/4}} \sum_{\substack{ |s_1|,|s_2| \ll N_1^{1/2} \\ s_1\neq s_2 \neq 0  \\  s_2 (b_1^2+1) \equiv  s_1(b_2^2+1) \, (D_0)  \\ s_2\equiv s_1 \, (D_0)}} \sum_{\substack{\Delta \neq 0 \\ D_0|\Delta \\ \Delta |(s_2 (b_1^2+1) -  s_1(b_2^2+1))  }} \sum_{ \substack{|r_1|,|r_2| \ll N_1^{1/2} \\ r_1s_2-r_2s_1 = \Delta 
 \\ r_2\equiv r_1 \, (D_0) }} 1  
\end{align}

For fixed $s_1,s_2,\Delta$ with $s_2\neq 0, |s_2| \ll N_1^{1/2}$, denoting $r_1=r_2+r_3 D_0$ we see that
\[
\sum_{\substack{|r_1|,|r_2| \ll N_1^{1/2} \\ r_1s_2-r_2s_1 = \Delta 
 \\ r_2\equiv r_1 \, (D_0) }} 1  \ll \sum_{\substack{|r_2|,|r_3| \ll N_1^{1/2} \\ r_3D_0s_2-r_2s_1 + r_2s_2 = \Delta 
  }} 1  \ll \sum_{\substack{|r_2| \ll N_1^{1/2} \\ r_2(s_2-s_1) \equiv \Delta 
 \, (D_0 s_2)  }} 1 \ll 1+  \frac{N_1^{1/2}(D_0 s_2,(s_2-s_1))}{D_0 s_2}.
\]
Thus, denoting  $s_0 = (s_1,s_2)$ and $s_j=s_0 s_j'$ 
we have
\[
\frac{(D_0 s_2,(s_2-s_1))}{D_0 s_2}  = \frac{(D_0,(s'_2-s'_1))}{D_0 s'_2},
\]
so that
\begin{align} \nonumber 
        W_{11 > }^{(1)}  \ll &X^{1/2} \sum_{D_0 > Y} \sum_{b_1,b_2 \ll X^{1/4}} \sum_{s_0 \ll N_1^{1/2}} \tau(s_0) \\  \label{w11second}&\sum_{\substack{ |s'_1|,|s'_2| \ll N_1^{1/2}/s_0  \\(s_1',s_2')=1 \\   s'_1\neq s'_2 \neq 0 \\ s_0 s_2' (b_1^2+1) \equiv  s_0s_1'(b_2^2+1) \, (D_0) \\   D_0 | s_0 (s'_2-s'_1)  }} \bigg(1+  \frac{N_1^{1/2}(D_0,(s_2'-s_1'))}{D_0 s'_2} \bigg)\tau(s_2'(b_1^2+1)-s_1'(b_2^2+1))  .
\end{align}
By  using Lemma \ref{divisorlemma} to find a  divisor $d \leq X^{\eps}$ of $\frac{1}{D_0}(s_0s'_2(b_1^2+1) - s_0s'_1(b_2^2+1))$ with
\[
\tau\bigg(\frac{1}{D_0}(s_0s'_2(b_1^2+1) - s_0s'_1(b_2^2+1))\bigg) \ll \tau(d)^{O_\eps(1)}
\]
we get
\begin{align*}
        W_{11 > }^{(1)}  \ll \,&X^{1/2} \sum_{D_0 > Y} \tau(D_0) \sum_{d \leq X^\eps} \tau(d)^{O_\eps(1)}  \sum_{s_0 \ll N_1^{1/2}} \tau(s_0) \\&\sum_{\substack{ |s'_1|,|s'_2| \ll N_1^{1/2}/s_0 \\ (s_1',s_2')=1 \\ s'_1\neq s'_2 \neq 0 \\   D_0 | s_0 (s'_2-s'_1)  }} \bigg(1+  \frac{N_1^{1/2}(D_0,(s_2'-s_1'))}{D_0 s'_2} \bigg)  \sum_{\substack{b_1,b_2 \ll X^{1/4}   \\ s_0 s_2' (b_1^2+1) \equiv  s_0s_1'(b_2^2+1) \, (dD_0) }} 1.
        \end{align*}
Using the trivial bound
\[
\sum_{\substack{b_1 \ll X^{1/4}   \\ s_0 s_2' (b_1^2+1) \equiv  s_0s_1'(b_2^2+1) \, (dD_0) }} 1 \ll \tau(D_0)\tau(d) \bigg( 1 + \frac{X^{1/4}(dD_0,s_0s_2')}{dD_0}\bigg)
\]
we get        
        \begin{align*}
      W_{11 > }^{(1)}    \ll\, &X^{1/2} \sum_{D_0 > Y}\tau(D_0)^2 \sum_{d \leq X^\eps} \tau(d)^{O_\eps(1)}  \sum_{s_0 \ll N_1^{1/2}} \tau(s_0) \\ &\sum_{\substack{ |s'_1|,|s'_2| \ll N_1^{1/2}/s_0 \\ (s_1',s_2')=1 \\ s'_1\neq s'_2 \neq 0 \\   D_0 | s_0 (s'_2-s'_1)  }} \bigg(1+  \frac{N_1^{1/2}(D_0,(s_2'-s_1'))}{D_0 s'_2} \bigg)  X^{1/4} \bigg( 1 + \frac{X^{1/4}(dD_0,s_0s_2')}{dD_0}\bigg)  \\
        \ll\, &X^{3/4}  \sum_{d \leq X^\eps} \tau(d)^{O_\eps(1)}  \sum_{s_0 \ll N_1^{1/2}} \tau(s_0) \sum_{\substack{ |s'_1|,|s'_2| \ll N_1^{1/2}/s_0 \\ (s_1',s_2')=1 \\ s'_1\neq s'_2 \neq 0  }} \\&\sum_{\substack{D_0 > Y \\   D_0 | s_0 (s'_2-s'_1) }} \tau(D_0)^2 \bigg( 1 + \frac{X^{1/4}(dD_0,s_0s_2')}{dD_0} + \frac{N_1^{1/2}}{ s'_2} + \frac{N_1^{1/2}X^{1/4} (D_0,(s_2'-s_1'))(dD_0,s_0s_2')}{d D_0^2 s'_2}\bigg).
\end{align*}
Note that by $(s_1',s_2')=1$ and $D_0|s_0(s_2'-s_1')$ we have
\[
(dD_0,s_0s_2') \leq  (d,s_0s_2') (D_0,s_0).
\]
and
\[
(D_0,(s_2'-s_1'))(D_0,s_0s_2') =(D_0,(s_2'-s_1'))(D_0,s_0)  \leq D_0(s_0,(s_2'-s_1'))
\]
so that
\begin{align*}
      W_{11 > }^{(1)}   
        \ll\, &X^{3/4}  \sum_{d \leq X^\eps} \tau(d)^{O_\eps(1)}  \sum_{s_0 \ll N_1^{1/2}} \tau(s_0) \sum_{\substack{ |s'_1|,|s'_2| \ll N_1^{1/2}/s_0 \\ (s_1',s_2')=1 \\ s'_2 \neq 0  }} \\&\sum_{\substack{D_0 > Y \\   D_0 | s_0 (s'_2-s'_1) }} \tau(D_0)^2 \bigg( 1 + \frac{X^{1/4}(d,s_0s_1')(D_0,s_0)}{dD_0} + \frac{N_1^{1/2}}{ s'_2} + \frac{N_1^{1/2}X^{1/4} (d,s_0s_2') (s_0,(s_2'-s_1'))}{d D_0 s'_2}\bigg).
\end{align*}
By using the bounds
\[
\sum_{\substack{D_0 > Y \\   D_0 | s_0 (s'_2-s'_1) }} \frac{\tau(D_0)^2(D_0,s_0)}{D_0} \leq (Y^{-1/2} +  \mathbf{1}_{s_0 > Y^{1/2}} )\tau(s_0)^3 \tau(s_2'-s_1')^3
\] 
and
\[
\sum_{\substack{D_0 > Y \\   D_0 | s_0 (s'_2-s'_1) }} \frac{\tau(D_0)^2}{D_0} \leq Y^{-1}\tau(s_0)^3 \tau(s_2'-s_1')^3
\]
we get
\begin{align*}
    W_{11 > }^{(1)}   \ll\, &X^{3/4}   \sum_{d \leq X^\eps} \tau(d)^{O_\eps(1)} \sum_{s_0 \ll N_1^{1/2}} \tau(s_0)^4  \\
    & \sum_{\substack{ |s'_1|,|s'_2| \ll N_1^{1/2}/s_0 \\ (s_1',s_2')=1 \\ s'_1\neq s'_2 \neq 0  }}  \tau(s_2'-s_1')^3 \bigg( 1   + \frac{N_1^{1/2}}{ s'_2}  +\frac{X^{1/4}(d,s_0s_2')}{d}(Y^{-1/2} +  \mathbf{1}_{s_0 > Y^{1/2}} ) \\
    & \hspace{190pt}+ \frac{N_1^{1/2}X^{1/4} (s_0,(s_2'-s_1'))(d,s_0s_2')}{ Y d  s'_2} \bigg) \\
    \ll_\eps \, & X^{3/4+2\eps} N_1  + \frac{XN_1 \log^{O(1)} X}{Y^{1/2}},
\end{align*}
where the last line follows by first computing the sum over $d \leq X^\eps$ using the bound
\[
\sum_{d \sim D} \tau(d)^{O(1)} (d,s)  \leq \sum_{v| s} \sum_{\substack{d \sim D \\ d \equiv 0\, (v)}} \tau(d)^{O(1)}  \ll \tau(s)^{O(1)} D \log^{O(1)} D.
\]
Thus, by taking $\eps >0$ small and $Y= \log^{C'} X$ with $C'\gg C$ we get
\[
 W_{11 > }^{(1)} \ll XN_1 \log^{-C} X.
\]

To bound $V_{11 > }^{(1)}$ we  note that similarly as in \eqref{w11first} we have
\[ 
V_{11 > }^{(1)}  \ll  X^{1/2}  \sum_{\substack{Y< \Delta_2 \ll N_1  \\ \Delta_2 \, \text{powerful}}} \sum_{b_1,b_2 \ll X^{1/4}} \sum_{\substack{ |s_1|,|s_2| \ll N_1^{1/2} \\ s_1\neq s_2 \neq 0  \\  s_2 (b_1^2+1) \equiv  s_1(b_2^2+1) \, (\Delta_2)  }} \sum_{\substack{\Delta \neq 0 \\ D_0|\Delta \\ \Delta |(s_2 (b_1^2+1) -  s_1(b_2^2+1))  }} \sum_{ \substack{|r_1|,|r_2| \ll N_1^{1/2} \\ r_1s_2-r_2s_1 = \Delta 
 }} 1.
 \]
Thus, similar to \eqref{w11second} we get
\begin{align*}
     V_{11 > }^{(1)}  \ll &X^{1/2}\sum_{\substack{Y< \Delta_2 \ll N_1  \\ \Delta_2 \, \text{powerful}}}  \sum_{b_1,b_2 \ll X^{1/4}} \sum_{s_0 \ll N_1^{1/2}} \tau(s_0) \\&\sum_{\substack{ |s'_1|,|s'_2| \ll N_1^{1/2}/s_0 \\ (s_1',s_2')=1 \\   s'_1\neq s'_2 \neq 0 \\ s_0 s_2' (b_1^2+1) \equiv  s_0s_1'(b_2^2+1) \, (\Delta_2)  }} \bigg(1+  \frac{N_1^{1/2}}{ s'_2} \bigg)\tau(s_2'(b_1^2+1)-s_1'(b_2^2+1))  .
\end{align*}
By  using Lemma \ref{divisorlemma} to find a suitable divisor $d \leq X^{\eps}$ of $\frac{1}{\Delta_2}(s_0s'_2(b_1^2+1) - s_0s'_1(b_2^2+1))$ and using the trivial bound for the sum over $b_1,b_2$ we get
\begin{align*}
        V_{11 > }^{(1)}  \ll \,&X^{1/2} \sum_{\substack{Y< \Delta_2 \ll N_1  \\ \Delta_2 \, \text{powerful}}} \tau(\Delta_2) \sum_{d \leq X^\eps} \tau(d)^{O_\eps(1)}  \sum_{s_0 \ll N_1^{1/2}} \tau(s_0) \\&\sum_{\substack{ |s'_1|,|s'_2| \ll N_1^{1/2}/s_0 \\ (s_1',s_2')=1 \\ s'_1\neq s'_2 \neq 0   }}   \frac{N_1^{1/2}}{s'_2}  \sum_{\substack{b_1,b_2 \ll X^{1/4}   \\ s_0 s_2' (b_1^2+1) \equiv  s_0s_1'(b_2^2+1) \, (d\Delta_2) }} 1 \\
        \ll\, &X^{1/2} \sum_{\substack{Y< \Delta_2 \ll N_1  \\ \Delta_2 \, \text{powerful}}} \tau(\Delta_2)^2 \sum_{d \leq X^\eps} \tau(d)^{O_\eps(1)}  \sum_{s_0 \ll N_1^{1/2}} \tau(s_0) \\ &\sum_{\substack{ |s'_1|,|s'_2| \ll N_1^{1/2}/s_0 \\ (s_1',s_2')=1 \\ s'_1\neq s'_2 \neq 0   }} \frac{N_1^{1/2}}{s'_2}  X^{1/4} \bigg( 1 + \frac{X^{1/4}(d\Delta_2,s_0s_2')}{d\Delta_2}\bigg)  \\
        \ll\, &X^{3/4} \sum_{\substack{Y< \Delta_2 \ll N_1  \\ \Delta_2 \, \text{powerful}}} \tau(\Delta_2)^2  \sum_{d \leq X^\eps} \tau(d)^{O_\eps(1)}  \sum_{s_0 \ll N_1^{1/2}} \tau(s_0) \\& \sum_{\substack{ |s'_1|,|s'_2| \ll N_1^{1/2}/s_0 \\ (s_1',s_2')=1 \\ s'_1\neq s'_2 \neq 0  }} \bigg(  \frac{N_1^{1/2}}{ s'_2} + \frac{N_1^{1/2}X^{1/4} (d\Delta_2,s_0s_2')}{d \Delta_2 s'_2}\bigg).
        \end{align*}
Thus, we get
\begin{align} \nonumber
  V_{11 >}^{(1)} \ll\, & X^{3/4}(\log^{O(1)}  X )\sum_{\substack{Y< \Delta_2 \ll N_1  \\ \Delta_2 \, \text{powerful}}} \tau(\Delta_2)^{O(1)} \bigg( X^\eps N_1 + \frac{N_1 X^{1/2}}{\Delta_2} \bigg) \\ \label{v11final}
  \ll\, & X^{3/4+2 \eps} N_1^{3/2} + \frac{XN_1 \log^{O(1)}  X }{Y}.
\end{align}
Since $N_1 \leq X^{1/2-10 \eps}$, by taking $\eps >0$ small and $Y= \log^{C'} X$ with $C'\gg C$ we get
\[
 V_{11 > }^{(1)} \ll XN_1 \log^{-C} X.
\]

\subsection{The off-diagonal for $U^{(2)}_{1}$} \label{typeii2offdiag}
Here we apply the same arguments as with $U^{(1)}_{11}$, except that we count solutions to
\[
c_1^3+d_1^3 \equiv a(c_2^3+d_2^3) \ (|\Delta|),
\]
with $c_1,c_2,d_1,d_2 \ll X^{1/6}$, so that in place of Lemmas \ref{points1Dlemma} and \ref{expsums1Dlemma} we use Lemmas \ref{points2Dlemma} and \ref{expsums2Dlemma}. 
First we separate the part where $\Delta_2 > Y$ by writing
\[
U^{(2)}_{11} = U^{(2)}_{11 \leq} + U^{(2)}_{11 >}.
\]
By similar arguments as in Section \ref{boundinglargeysection}, in place of \eqref{v11final} we get
\[
U^{(2)}_{11 >} \ll X^{5/6+2 \eps} N_2^{3/2} + \frac{XN_2 \log^{O(1)}  X }{Y},
\]
which suffices since $N_2 \ll X^{1/3-10\eps}$. Note that it is not necessary to separate large values of $(z_2-z_1,\Delta)$ since the bounds from Section \ref{weilsection} we will use here do not care about the common factor $(a-1,D)$.

For $U^{(2)}_{11 \leq}$ the technical issues of removing the smooth cross-condition $F_{M_2}$ and bounding the part $\theta_1 = \theta_2 + O(\sqrt{\delta}) \pmod{\pi}$ are handled similarly as for $U^{(1)}_{11 \leq}$ (cf. Sections \ref{ccremovalsubsection} and \ref{smalldeltasection}), so that we need to consider sums of the form 
\[
\begin{split}
U_1^{(2)}(\bm{C},\bm{D},\bm{N},\bm{\theta}):=\Omega(C_1,D_1)&\Omega(C_2,D_2)\sum_{\substack{|z_1|,|z_2| \\ \Delta \neq 0  \\ (z_1,z_2)=1 \\ \max\{\Delta_2,(z_2-z_1,\Delta)\} \leq Y}} \beta^{(1)}_{z_1}\beta^{(2)}_{z_2}  \\
&\sum_{\substack{c_1,c_2,d_1,d_2 \\ c_2^3+d_2^3 \equiv a(c_1^3+d_1^3) \, (|\Delta|)}} \psi_{C_1}(c_1)\psi_{C_2}(c_2) \psi_{D_1}(d_1)\psi_{D_2}(d_2)
\end{split}
\]
with $C_1,C_2,D_1,D_2 \ll X^{1/6}$. Let $H_1 := X^\eps|\Delta|/C_1,$ $H_2 := X^\eps|\Delta|/C_2,$  $H_3 := X^\eps|\Delta|/D_{1},$ and $H_4 := X^\eps|\Delta|/D_{2}.$ Applying Poisson summation four times and denoting
\[
S_2(\bm{h};\Delta) :=  \sum_{\substack{x_1,x_2,
x_3,x_4 \, (\Delta) \\ x_1^3+x_2^3 \equiv a(x_3^3+x_4^3) \, (\Delta)}}e_{\Delta}(\bm{h}\cdot \bm{x}),
\]
the contribution from $(h_1,h_2,h_3,h_4) \neq \bm{0}$ is by Lemma \ref{expsums2Dlemma} bounded by (using $(a-1,\Delta) \leq Y$ and $(a,\Delta)=1$)
\[
\begin{split}
\Omega(C_1,D_1)&\Omega(C_2,D_2)\sum_{\substack{z_1,z_2 \\ \Delta \neq 0  \\ (z_1,z_2)=1 \\ \max\{\Delta_2,(z_2-z_1,\Delta)\} \leq Y }}  \frac{1}{H_1 H_2H_3H_4} \sum_{\substack{|h_j| \leq H_j \\ (h_1,h_2,h_3,h_4) \neq \bm{0}}} S_2(\bm{h};\Delta)  \\
&\ll  (X^{1/6})^2 \sum_{\substack{z_1,z_2 \\ \arg z_1 \neq \arg z_2 \\ (z_1,z_2)=1 \\ \max\{\Delta_2,(z_2-z_1,\Delta)\} \leq Y}}  \frac{Y^{O(1)}\tau(\Delta)^{O(1)}\Delta^3}{\Delta} \ll X^{1/3+\eps} N^4 \ll X^{1-\eta} N,
\end{split}
\]
since $N \ll X^{2/9-\eta}$.
For the main term (with $(h_1,h_2,h_3,h_4) = \bm{0}$)
\[
\Omega(C_1,D_1)\Omega(C_2,D_2)C_1C_2D_1D_2\sum_{\substack{z_1,z_2 \\ \Delta \neq 0  \\ (z_1,z_2)=1 \\  \max\{\Delta_2,(z_2-z_1,\Delta)\} \leq Y}} \beta^{(1)}_{z_1}\beta^{(2)}_{z_2} \frac{\hat{\psi}(0)^4}{\Delta^4} N_3(a; \Delta) 
\]
we get a sufficient bound by using Lemma \ref{points2Dlemma} and the Siegel-Walfisz property with main term 0 \eqref{swmain0} similarly as in Section \ref{maintermsection}. 

\subsection{The off-diagonal for $U^{(j)}_{2}$}
Here we can apply the same arguments as in the above sections, but the evaluation  is of course much easier. We are counting $b_1,b_2$ which satisfy the simple equation (with $(a,\Delta)=1$)
\[
b_2=ab_1 \, (|\Delta|).
\]
Here the lengths $B_j$ of $b_j$ satisfy $B_j > X^{1/2-\eta} > N X^\eta \geq |\Delta| X^{\eta}$, so that after Poisson summation we only get a contribution from the frequency $(0,0)$. The point count modulo $\Delta$ corresponding to $N_1(a;\Delta)$ is trivially equal to $|\Delta|$. Thus, applying the Siegel-Walfisz property \eqref{swmain0} we get
\[
U^{(j)}_{2} \ll_C XN \log^{-C} X.
\]
Again the technical issues of the contribution from $\max\{\Delta_2,(z_2-z_1,\Delta)\} > Y$, removing $F_{M_1}$, and a bounding the part $\theta_1 = \theta_2 + O(\sqrt{\delta}) \pmod{\pi}$ are handled similarly as above.
\section{The sieve argument for Theorems \ref{maintheorem1} and \ref{maintheorem2}} \label{sievesection}
In this section we give the proofs of Theorems \ref{maintheorem1} and \ref{maintheorem2} by applying Harman's sieve method with the arithmetic information given by Propositions \ref{flprop} and \ref{typeiiprop}. Define
\[
S(\A^{(j)}_d,Z) := \sum_{\substack{ n \sim X/d\\ (n,P(Z))=1}} a^{(j)}_{dn} \quad \text{and} \quad S(\B_d,Z) := \sum_{\substack{ n \sim X/d\\ (n,P(Z))=1}} b_{dn},
\]
and denote $\A^{(j)}_1=\A^{(j)}$ and $\B_1=\B$. Then Theorems \ref{maintheorem1} and \ref{maintheorem2} follow from the following quantitative version.  The lower bounds have not been optimized and could be improved by more Buchstab iterations, especially for $\A^{(2)}$. The upper bounds in Theorems \ref{maintheorem1} and \ref{maintheorem2} follow from a standard sieve upper bound, using Proposition \ref{typeiprop} or even just the trivial level of distribution $X^{1/2-\eta}$.
\begin{theorem} \label{quanttheorem}
For $X$ sufficiently large we have
\[
S(\A^{(1)}, 2 X^{1/2}) \geq 0.6 \cdot S(\B, 2 X^{1/2})
\]
and 
\[
S(\A^{(2)}, 2 X^{1/2}) \geq 0.1 \cdot S(\B, 2 X^{1/2}).
\]
\end{theorem}
Let $\eta >0$ be small and define the widths of the Type II ranges (Proposition \ref{typeiiprop})
\[
\gamma_1 := 1/3-1/4 - 2\eta=1/12 - 2\eta \quad \text{and} \quad \gamma_2 := 2/9-1/6 - 2\eta=1/18 - 2\eta.
\]
While our arithmetic information is not sufficient to give an asymptotic formula for primes, we can still give an asymptotic formula for certain sums of almost-primes with no prime factors below $X^{\gamma_j}$, as the following Proposition shows.
\begin{prop} \label{funprop}
Let $D_1 := X^{3/4-\eta}$ and $D_2 := X^{5/6-\eta}$ and let $U_j \leq D_j $. Let $W={X^{1/(\log\log x)^2}}$. Then for any bounded coefficients $\alpha(m)$ supported on $(m,P(W))=1$ we have
\[
\sum_{m \sim U_j} \alpha(m) S(\A^{(j)}_m,X^{\gamma_j}) = \sum_{m \sim U_j} \alpha(m) S(\B_m,X^{\gamma_j}) + O_{\eta,C}( X/\log^C X ).
\]
\end{prop}
\begin{proof}
Let $\CC=(c_n) \in \{\A^{(j)},\B\}$. Then by expanding using the M\"obius function we get
\[
\begin{split}\sum_{m \sim U_j}\alpha(m)S(\CC_m,X^{\gamma_j}) &= \sum_{m \sim U_j} \alpha(m)\sum_{\substack{ n \sim X/m\\ (n,P(X^{\gamma_j}))=1}} c_{mn} \\
&=   \sum_{m \sim U_j}\alpha(m) \sum_{d| P(W,X^{\gamma_j})} \mu (d) \sum_{\substack{ n \sim X/dm\\ (n,P(W))=1}} c_{dmn}.
\end{split}
\]
We split the sum in two parts, $dU_j \leq D_j$ and $dU_j > D_j$, and show that in each part we can get an asymptotic formula by Type I information and Type II information, respectively. 

For the first part we write 
\[
\alpha'(m') := \sum_{\substack{m'=dm \\ m \sim U_j \\d| P(W,X^{\gamma_j})\\ d U_j \leq D_j }} \alpha(m)\mu(d) =\alpha_1'(m') + \alpha_2'(m'),  
\]
where  for some large constant $C >0$
\[
\alpha_1'(m') =  \alpha'(m') \bm{1}_{\tau(m') \leq \log^C X} \quad \text{and} \quad \alpha_2'(m') =  \alpha'(m') \bm{1}_{\tau(m') > \log^C X}.
\]
Since $|\alpha'_1(m') | \leq \log^C X$, we get by Proposition \ref{flprop}
\[
\sum_{m' \leq 2D_j} \alpha'_1(m') \sum_{\substack{ n \sim X/m'\\ (n,P(W))=1}} a^{(j)}_{m'n} = \sum_{m' \leq 2D_j} \alpha'_1(m')\sum_{\substack{ n \sim X/m'\\ (n,P(W))=1}} b_{m'n} + O_C(X/\log^C X).
\]
For $\alpha'_2(m') \leq \tau(m')$ we get by Lemma \ref{divisorlemma} with $k=10$ and by trivial bounds
\begin{equation} \label{tautrick}
\begin{split}
&\bigg|\sum_{m' \leq 2D_j} \alpha'_2(m') \sum_{\substack{ n \sim X/m'\\ (n,P(W))=1}} c_{m'n}\bigg| \leq \sum_{\substack{e \leq X^{1/10} \\ \tau(e)^{10^3} > \log^C X}} \tau(e)^{10^3} \sum_{n \sim X/e} \tau(n) c_{en}  \\
&\leq  \sum_{\substack{e,f \leq X^{1/10} \\ \tau(e)^{10^3} > \log^C X}} \tau(e)^{O(1)}\tau(f)^{O(1)} \sum_{n \sim X/ef}  c_{efn} \ll X  \sum_{\substack{e,f \leq X^{1/10} \\ \tau(e)^{10^3} > \log^C X}} \frac{\tau(e)^{O(1)}\tau(f)^{O(1)}}{ef}  \\
&\ll  \frac{X}{\log^C X}  \sum_{\substack{e,f \leq X^{1/10} }} \frac{\tau(e)^{10^3+O(1)}\tau(f)^{O(1)}}{ef} \ll X \log^{-C +O(1)} X,
\end{split}
\end{equation}
where we have plugged in the factor $\tau(e)^{10^3} / \log^C X \geq 1$ to get the penultimate step.

Similarly, in the part $dU_j > D_j$ we get an asymptotic formula by Proposition \ref{typeiiprop}. To see this, let us write $d=p_1\cdots p_k$ to get  (noting that $k\neq 0$ since $d > D_j/U_j > 1$)
\[
\sum_{\substack{m \sim U_j \\ dU_j > D_j \\ d| P(W,X^{\gamma_j})}} \alpha(m)  \mu (d) \sum_{\substack{ n \sim X/dm\\ (n,P(W))=1}} c_{dmn} = \sum_{1 \leq k \ll \log X}(-1)^k\sum_{\substack{m \sim U_j}} \alpha(m) \sum_{\substack{W \leq p_1 < \cdots < p_k < X^{\gamma_j} \\ p_1\cdots p_k U_j > D_j}} \sum_{\substack{ n \sim X/p_1\cdots p_k m \\ (n,P(W))=1}} c_{p_1\cdots p_k mn}
\]
Since $p_1\cdots p_k U_j > D_j$, $U_j \leq D_j$ and $p_\ell < X^{\gamma_j}$, by the greedy algorithm there is a unique $0\leq \ell \leq k$ such that
\[
p_1\cdots p_\ell U_j \in [D_j X^{-\gamma_j},  D_j]  \quad \text{and} \quad p_1\cdots p_{\ell-1} U_j  < D_j X^{-\gamma_j} 
\]
Note that $[D_j X^{-\gamma_j},  D_j] $ is now exactly the admissible range for $M$ in Proposition \ref{typeiiprop}. We obtain that the part  $dU_j > D_j$ is partitioned into $\ll \log^2 X$ sums of the form
\[
\sum_{\substack{m \sim U_j}}\alpha(m) \sum_{\substack{W \leq p_1 < \cdots < p_k < X^{\gamma_j} \\ p_1\cdots p_k U_j > D_j \\ p_1\cdots p_\ell U_j \in [D_j X^{-\gamma_j},  D_j]  \\  p_1\cdots p_{\ell-1} U_j  < D_j X^{\gamma_j} }} \sum_{\substack{ n \sim X/p_1\cdots p_k m \\ (n,P(W))=1}} c_{p_1\cdots p_k mn}.
\]
The cross-conditions  $p_{\ell+1} > p_\ell$ and $p_1\cdots p_k U_j > D_j$ are easily removed by applying a finer-than-dyadic decomposition (Section \ref{smoothweightsection} with $\delta=\log^{-C} X$) to the four variables $p_\ell$, $p_{\ell+1}$, $p_1 \cdots p_\ell,$ and $p_{\ell+1} \cdots p_k$. Hence, writing
\[
m':= m p_1\cdots p_\ell \quad \text{and} \quad n' := n p_{\ell+1} \cdots p_k,
\]
we obtain for some coefficients $\alpha'(m')$ and $\beta'(n')$  supported on $(m'n',P(W))=1$ sums of the form
\[
\sum_{m'n' \sim X}  \alpha'(m') \beta'(n') c_{m'n'}.
\]
Here $\alpha'(m')$ is supported on $m' \in [D_j X^{-\gamma_j},  2 D_j]$ with $(m',P(W))=1$. The cross-condition $m'n' \sim X$ can be removed by a finer-than-dyadic decomposition and the coefficients $\alpha'(m'),\beta'(n')$ can be reduced to bounded coefficients by a similar argument as in (\ref{tautrick}).  Hence, by Propostion \ref{typeiiprop} we get
\[
\sum_{m'n' \sim X}  \alpha'(m') \beta'(n') a^{(j)}_{m'n'} = \sum_{m'n' \sim X}  \alpha'(m') \beta'(n') b_{m'n'} + O_C(X/\log^C X),
\]
provided that we show that $\beta'(|z|^2)$ satisfies the Siegel-Walfisz property (\ref{swproperty}). To this end we write for any $m'$
\[
\beta'(n') =  \sum_{p \geq W} \beta'(pn) \bm{1}_{P^+(n) \leq p} = \sum_{p \geq W} \beta'(pn)\bm{1}_{P^+(n) < p}  \bm{1}_{p\nmid m'n}+\sum_{p \geq W} \beta'(pn) \bm{1}_{P^+(n) \leq p}\bm{1}_{p| m'n}.
\]
In the first term we see by the support of $c_{m'n'}$ that $p$ is restricted to $p \equiv 1 \, (4)$ by the sum-of-two-squares theorem. Hence, by dropping $p\nmid m'n$ we get
\[
\beta'(n')=\sum_{\substack{p \geq W \\ p \equiv 1 \, (4)}}\bm{1}_{P^+(n) < p}  \beta'(pn) +O\bigg(\sum_{p \geq W} |\beta'(pn)| \bm{1}_{p| m'n} \bigg)
\]
The contribution from the second term is negligible by trivial bounds since we get a square factor $\geq W$. Similarly,  we may reduce to the case where  $\beta'(n')$ is supported on square-free integers.  By the construction of $\beta'(n')$ we can write 
\[
n' = pn= p q_1 \cdots q_K, \quad \text{for} \quad W < q_1 < \cdots < q_K < p.
\]
and remove any cross-conditions involving $p$ by a finer-than-dyadic decomposition (Section \ref{smoothweightsection} with $\delta$) applied to $p$ and at most $O(1)$ variables. Hence, for some $\beta''(n)$ the coefficient $\beta'(n')$ is up to a negligible error term replaced by coefficients of the form
\[
\beta''(n)\sum_{\substack{p \equiv 1 \, (4)  \\ p \nmid n}} \psi_P(p) 
\]
with $P \geq W/2$. The condition $p\nmid n$ may be droppped with a negligible error term as $p \gg W$. Let $\psi'$ be as in Section \ref{smoothweightsection} with $\delta'=\delta (\log X)^{-C}$. Then applying Section \ref{smoothweightsection} to the variable $n$ with $\psi'$ we get coefficients of the form
\[
\sum_{n'=pn} \beta''(n) \psi_{N_0}'(n)\sum_{\substack{p \equiv 1 \, (4)  }} \psi_P(p) ,
\]
and since $\delta'$ is small compared to $\delta$, this can be replaced with a negligible error term by
\[
\beta'''(n') := \psi_{PN_0} (n') \sum_{n'=pn} \beta''(n) \psi_{N_0}'(n) 
  \sum_{\substack{p \equiv 1 \, (4) }}  1 
\]
which is as in \eqref{betageneric} and thus satisfies the Siegel-Walfisz property \eqref{swproperty}.
\end{proof}

The general idea of Harman's sieve is to use Buchstab's identity to decompose the sum $S(\CC,2\sqrt{X})$  (in parallel for $\CC=\A^{(j)}$ and $\CC=\B$) into a sum of the form $\sum_k \epsilon_k S^{(j)}_k(\CC),$ where $\epsilon_k \in \{-1,1\}$ and  $S^{(j)}_k(\CC) \geq 0$ are sums over almost-primes.  Since wish to obtain a lower bound, for $\CC=\A^{(j)}$ we can insert the trivial estimate $S^{(j)}_k(\A^{(j)}) \geq 0$ for any $k$ such that the sign $\epsilon_k =1$ -- we say that these sums are \emph{discarded}. For the remaining $k$ we will obtain an asymptotic formula by using Propositions \ref{typeiiprop} and \ref{funprop}. That is, if $\mathcal{K}$ is the set of indices of the sums that are discarded, then we will show
\begin{align*}
S(\A^{(j)}, 2\sqrt{X})&= \sum_k \epsilon_k S^{(j)}_k(\A^{(j)}) \geq \sum_{k \notin \mathcal{K}} \epsilon_k S^{(j)}_k(\A^{(j)})  \\
&=(1+o(1))  \sum_{k \notin \mathcal{K}} \epsilon_k  S^{(j)}_k(\B) = (1+o(1)) S(\B,2\sqrt{X}) -   \sum_{k \in \mathcal{K}}  S^{(j)}_k(\B).
\end{align*} 
We are successful if we can then show that $\sum_{k \in \mathcal{K}} S^{(j)}_k(\B) < c_j S(\B, 2\sqrt{X})$ for $c_1 = 0.4$ and $c_2=0.9$. 

To bound the error terms corresponding to $k \in \mathcal{K}$ we need a lemma which converts sums over almost primes over $b_n$ into integrals which can be bounded numerically.  Let $\omega(u)$ denote the Buchstab function (see \cite[Chapter 1]{harman} for the properties below, for instance), so that by the Prime number theorem for $Y^{\epsilon} < Z < Y$ 
\[
\sum_{Y< n \leq 2Y} \bm{1}_{(n,P(Z))=1} = (1+o(1)) \omega \left(\frac{\log Y}{\log Z} \right) \frac{Y}{\log Z}.
\]
Similarly, using the Prime number theorem for primes $p\equiv 1 \,\, (4)$ we get
\begin{equation}\label{buchasymp}
\sum_{Y< n \leq 2Y } \bm{1}_{(n,P(Z))=1} \sum_{n=a^2+b^2} 1  = (1+o(1)) \omega \left(\frac{\log Y}{\log Z} \right) \frac{Y}{\log Z}.
\end{equation}
Note that for $1< u \leq 2$ we have $\omega(u)=1/u.$ In the numerical computations we will use the following upper bound for the Buchstab function (see \cite[Lemma 5]{hbjia}, for instance)
\[
\omega(u) \, \leq \begin{cases} 0, &u < 1 \\
1/u, & 1 \leq u < 2 \\
(1+\log(u-1))/u, &2 \leq u < 3 \\
0.5644, &  3 \leq u < 4 \\
0.5617, & u \geq 4.
\end{cases}
\]
In the lemma below we assume that the range $\mathcal{U}\subset [X^{2 \delta},X]^{k}$ is sufficiently well-behaved, for example, an intersection of sets of the type $\{ \boldsymbol{u}: u_i < u_j \}$ or $\{\boldsymbol{u}: V <  f(u_1, \dots,u_k) < W\}$ for some polynomial $f$ and some fixed $V,W.$
\begin{lemma}  \label{buchstablemma}Let $\mathcal{U} \subset [X^{2 \delta},X]^{k}.$ Then
\[
\sum_{(p_1, \dots , p_k) \in \mathcal{U}} S(\B_{p_1, \dots, p_k},p_k) = S(\B,2\sqrt{X})(1+o(1))\int \omega (\boldsymbol{\beta }) \frac{d\beta_1 \cdots d\beta_k}{\beta_1\cdots\beta_{k-1}\beta_k^2},
\]
where the integral is over the range 
\[
\{\boldsymbol{\beta}: \, (X^{\beta_1}, \dots, X^{\beta_k}) \in \mathcal{U}\}
\]
 and $\omega(\boldsymbol{\beta})= \omega(\beta_1,\dots,\beta_k):= \omega((1-\beta_1-\cdots 
 -\beta_k)/\beta_k)$.
\end{lemma}
\begin{proof}
Recall that
\[
b_n = \sum_{\substack{n=a^2+b^2 \\ (a,b)=1 \\ b > 0}} 1 = \frac{1}{2} \sum_{\substack{n=a^2+b^2 \\ (a,b)=1 }} 1.
\]
By (\ref{buchasymp}) and by the Prime number theorem for primes $p \equiv 1 \, \, (4)$, the left-hand side in the lemma is
\[
\begin{split}
& \sum_{(p_1, \dots , p_k) \in \mathcal{U}} \sum_{q\sim X/p_1\cdots p_k} \bm{1}_{(q,P(p_k))=1} b_{q p_1\cdots p_k} \\
&= (1+o(1))\frac{X}{2}\sum_{(p_1, \dots , p_k) \in \mathcal{U}} \frac{1}{p_1\cdots p_k \log p_k} \omega \left( \frac{\log(X/(p_1\cdots p_k))}{\log p_k} \right) \\
&= (1+o(1))\frac{X}{2} \hspace{-5pt}  \sum_{(n_1,\dots,n_k ) \in \mathcal{U}} \frac{1}{n_1\cdots n_k (\log n_1) \dots (\log n_{k-1} )\log^2 n_k} \omega \left( \frac{\log(X /(n_1\cdots n_k))}{\log n_k} \right) \\
&= (1+o(1))\frac{X}{2} \int_{\mathcal{U}}  \omega \left( \frac{\log(X/(u_1\cdots u_k))}{\log u_k} \right)   \frac{du_1\cdots du_k}{u_1\cdots u_k (\log u_1) \dots (\log u_{k-1} )\log^2 u_k}\\
&= (1+o(1))\frac{X}{2\log X} \int \omega (\boldsymbol{\beta }) \frac{d\beta_1 \cdots d\beta_k}{\beta_1\cdots\beta_{k-1}\beta_k^2} =(1+o(1))S(\B,2\sqrt{X}) \int \omega (\boldsymbol{\beta }) \frac{d\beta_1 \cdots d\beta_k}{\beta_1\cdots\beta_{k-1}\beta_k^2}
\end{split}
\]
by the change of variables $u_j=X^{\beta_j}$.
\end{proof}

\subsection{Buchstab iterations}
We are now ready to give the proof of Theorem \ref{quanttheorem}. We will apply similar Buchstab decompositions in both cases $j \in \{1,2\}$. Let $\CC=(c_n) \in \{\A^{(j)},\B\}$. We get by applying Buchstab's identity twice
\[\begin{split}
S(\CC,2 X^{1/2}) &=  S(\CC,X^{\gamma_j}) - \sum_{X^{\gamma_j} \leq p_1 <2X^{1/2}} S(\CC_{p_1},X^{\gamma_j})+  \sum_{X^{\gamma_j} \leq p_2< p_1< 2 X^{1/2}}S(\CC_{p_1p_2},p_2)  \\
&=: S_1(\CC) - S_2(\CC) + S_3(\CC).
\end{split}
\]
By Proposition \ref{funprop} we get for $j,k \in \{1,2\}$
\[
S_k(\A^{(j)}) = S_k(\B) + O_C( X / \log^C X).
\]

For the third sum we first separate the part where $p_1p_2^2 > 4X$. In this part $S(\CC_{p_1p_2},p_2)$ counts numbers $n$ of size $\leq 2X/p_1p_2 < p_2/2$ with $(n,P(p_2))=1$, so that we must have $n=1$ and $p_1p_2 \sim X$. Hence,  we get
\[
\sum_{\substack{X^{\gamma_j} \leq p_2< p_1< 2 X^{1/2} \\ p_1p_2^2 > 4X }}S(\B_{p_1p_2},p_2) \ll \sum_{p_1,p_2 \ll X^{1/2}} b_{p_1p_2} \ll X/ \log^2 X.
\]
and for $\CC= \A^{(j)}$ we use positivity and the trivial bound
\[
\sum_{\substack{X^{\gamma_j} \leq p_2< p_1< 2 X^{1/2} \\ p_1p_2^2 > 4X}}S(\A^{(j)}_{p_1p_2},p_2)  \geq 0.
\]

Let $\tilde{S}_3(\CC)$ denote the part of $S_3(\A)$ with $p_1p_2^2 \leq 4X$. We split it into two parts
\[
\tilde{S}_3(\CC) = S_{31}(\CC) + S_{32}(\CC),
\]
where
\[
\begin{split}
S_{31}(\CC):=\sum_{\substack{X^{\gamma_j} \leq p_2< p_1< 2 X^{1/2} \\ p_1p_2^2 \leq D_j }}S(\CC_{p_1p_2},p_2) \quad \text{and} \quad
S_{32}(\CC):=\sum_{\substack{X^{\gamma_j} \leq p_2< p_1< 2 X^{1/2} \\  D_j < p_1p_2^2 \leq 4X }}S(\CC_{p_1p_2},p_2).
\end{split}
\]

Consider first $S_{32}(\CC).$ Let  $D_j := X^{\alpha_j}$ and denote the Type II range (Proposition \ref{typeiiprop}) by
\[
\begin{split}
I(j) &:= [X^{1-\alpha_j},X^{1-\alpha_j+\gamma_j}] \cup [X^{\alpha_j-\gamma_j} \cup X^{\alpha_j}] \quad \text{\and} \quad\\
 J(j)&:= \frac{\log I(j)}{\log X} = [1-\alpha_j,1-\alpha_j+\gamma_j] \cup [\alpha_j-\gamma_j, \alpha_j].
\end{split}\]
By Proposition \ref{typeiiprop} (after removal of cross-conditions) we get an asymptotic formula whenever a combination of the variables lies in $I(j)$, that is,
\[
\sum_{\substack{X^{\gamma_j} \leq p_2< p_1< 2 X^{1/2} \\  D_j < p_1p_2^2 \leq 4X  \\ p_1,p_2, \,\text{or} \, p_1p_2 \in I(j) }}S(\A^{(j)}_{p_1p_2},p_2) = \sum_{\substack{X^{\gamma_j} \leq p_2< p_1< 2 X^{1/2} \\  D_j < p_1p_2^2 \leq 4X  \\ p_1,p_2, \,\text{or} \, p_1p_2 \in I(j) }}S(\B_{p_1p_2},p_2)  + O_C(X/\log^C X).
\]
For the remaining part we write by positivity $S_{32}(\A^{(j)}) \geq 0$ and by Lemma \ref{buchstablemma}
\begin{equation} \label{2dimintegral}
\sum_{\substack{X^{\gamma_j} \leq p_2< p_1< 2 X^{1/2} \\  D_j < p_1p_2^2 \leq 4X  \\ p_1,p_2, \,\text{and} \, p_1p_2 \not \in I(j) }}S(\B_{p_1p_2},p_2)  = (\Omega_2(j)+o(1)) S(\B,2 X^{1/2})
\end{equation}
for
\[
\Omega_2(j):= \iint_{U_2(j)} \omega((1-\beta_2-\beta_2)/\beta_2) \frac{d \beta_1 d \beta_2}{\beta_1\beta_2^2}
\]
with the two dimensional range
\[
U_2(j):= \{ (\beta_1,\beta_2): \gamma_j < \beta_2 <\beta_1 < 1/2,\,\, \alpha_j <\beta_1  + 2 \beta_2 < 1,\, \beta_1,\beta_2,\beta_1+\beta_2 \not \in J(j) \}.
\]

In the part $S_{31}(\CC)$ we can apply Buchstab's identity to get
\[
S_{31}(\CC)= \sum_{\substack{X^{\gamma_j} \leq p_2< p_1< 2 X^{1/2} \\ p_1p_2^2 \leq D_j }}S(\CC_{p_1p_2},X^{\gamma_j})  - \sum_{\substack{X^{\gamma_j} \leq  p_3<p_2< p_1< 2 X^{1/2} \\ p_1p_2^2 \leq D_j }}S(\CC_{p_1p_2p_3},p_3) 
\]
If $p_1p_2p_3^2 > 4X$ then $S(\CC_{p_1p_2p_3},p_3) =0$ except if $p_1p_2 p_3 \sim X$, which contradicts the fact that $p_1p_2 p_3 < p_1p_2^2 \leq D_j$. Thus, in the second sum we always have  $p_1p_2p_3^2 \leq 4X$. Applying Buchstab's identity once more we get
\[
S_{31}(\CC) =  S_{311}(\CC) - S_{312}(\CC) + S_{313}(\CC)
\]
with
\[
\begin{split} S_{311}(\CC) &:= \sum_{\substack{X^{\gamma_j} \leq p_2< p_1< 2 X^{1/2} \\ p_1p_2^2 \leq D_j }}S(\CC_{p_1p_2},X^{\gamma_j}), \\
S_{312}(\CC) &:= \sum_{\substack{X^{\gamma_j} \leq  p_3<p_2< p_1< 2 X^{1/2} \\ p_1p_2^2 \leq D_j  \\ p_1p_2p_3^2 \leq 4X}}S(\CC_{p_1p_2p_3},X^{\gamma_j}),  \quad \text{and} \\
S_{313}(\CC) & := \sum_{\substack{X^{\gamma_j} \leq p_4 < p_3<p_2< p_1< 2 X^{1/2} \\ p_1p_2^2 \leq D_j \\  p_1p_2p_3^2 \leq 4X }}S(\CC_{p_1p_2p_3p_4},p_4).
\end{split}
\]
Since $p_1p_2p_3 \leq p_1p_2^2 \leq D_j$, we have asymptotic formulas for the first two sums, that is, for $j,k \in \{1,2\}$
\[
S_{31k}(\A^{(j)}) = S_{31k}(\B) + O_C( X / \log^C X).
\]

For the third sum we use a similar treatment as for $S_{32}(\CC)$ above. First, we separate the contribution from $p_1p_2p_3p_4^2 > 4X$, which implies $p_1p_2p_3p_4 \sim X$. By $p_1p_2p_3^2 \leq 4X$ and $p_4 < p_3$ we infer $p_4=(1+O(1))p_3$. Hence, we have 
\[
\sum_{\substack{X^{\gamma_j} \leq p_4 < p_3<p_2< p_1< 2 X^{1/2} \\ p_1p_2^2 \leq D_j \\ p_1p_2p_3^2 \leq 4X\\ p_1p_2p_3p_4^2 > 4X }}S(\B_{p_1p_2p_3p_4},p_4) \ll \sum_{\substack{ p_1p_2p_3p_4 \sim X \\   X^{\gamma_j} \leq p_4 < p_3<p_2< p_1< 2 X^{1/2} \\ p_4 =(1+O(1))p_3} } b_{p_1p_2p_3p_4} \ll X/ \log^2 X,
\]
so that we may restrict to $p_1p_2p_3p_4^2 \leq 4X$. Here we use Proposition \ref{typeiiprop} to give an asymptotic formula whenever a combination of the variables is in the Type II range $I(j)$ and we discard the rest. This gives us an error contribution (by Lemma \ref{buchstablemma})
\begin{equation} \label{4dimintegral}
\sum_{\substack{X^{\gamma_j} \leq p_4 < p_3<p_2< p_1< 2 X^{1/2} \\ p_1p_2^2 \leq D_j \\ p_1p_2p_3^2 \leq 4X\\ p_1p_2p_3p_4^2 \leq 4X \\ \prod_{i \in I} p_i \notin I(j) }}S(\B_{p_1p_2p_3p_4},p_4) = (\Omega_4(j) + o(1)) S(\B,2 X^{1/2}),
\end{equation}
where 
\[
\Omega_4(j):= \iiiint_{U_4(j)} \omega((1-\beta_1-\beta_2-\beta_3-\beta_4)/\beta_4) \frac{d \beta_1 d \beta_2 d \beta_3 d \beta_4}{\beta_1\beta_2 \beta_3 \beta_4^2}
\]
with the four dimensional range $U_4(j)$ defined by
\[
\begin{split}
\gamma_j <\beta_4< \beta_3< \beta_2 <\beta_1 < 1/2,\quad \beta_1  + 2 \beta_2 < \alpha_j,\quad \beta_1+\beta_2+2\beta_3 < 1, \\ \beta_1+\beta_2+\beta_3+2\beta_4 < 1,  \quad \text{and} \quad \sum_{i \in I} \beta_i \not \in J(j) \,\, \forall I \subseteq\{1,2,3,4\} .
\end{split}
\]

Combining the above asymptotic evaluations and the error terms (\ref{2dimintegral}) and (\ref{4dimintegral}), we have shown that for $j \in \{1,2\}$
\[
S(\A^{(j)},2 X^{1/2}) \geq (1-\Omega_2(j)-\Omega_4(j)) S(\B,2 X^{1/2}).
\]
To complete the proof of Theorem \ref{quanttheorem}, we compute the numerical upper bounds
\[
\Omega_2(1)+\Omega_4(1) \leq 0.38 + 0.017 < 0.4
\]
and
\[
\Omega_2(2)+\Omega_4(2) \leq 0.38 + 0.49 < 0.9.
\]
Python codes which compute rigorous upper bounds for these integrals can be found at the following links.

\begin{tabular}{ c c c }
&\hspace{50pt} $\Omega_2(1)$ & \quad \quad \url{http://codepad.org/lO5xHGuO} \\
&\hspace{50pt} $\Omega_4(1)$ &\quad \quad \url{http://codepad.org/Ygi8kNUF}  \\
&\hspace{50pt} $\Omega_2(2)$ & \quad \quad\url{http://codepad.org/6hNTZtYt}  \\
&\hspace{50pt} $\Omega_4(2)$ &\quad \quad \url{http://codepad.org/spA3ENVz} 
\end{tabular}

This completes the proof of Theorem \ref{quanttheorem}, and as mentioned above, we then get Theorems \ref{maintheorem1} and \ref{maintheorem2}.
\qed
 
\section{The set-up for Theorem \ref{maintheorem3}} \label{setup3section}
\subsection{The set-up}
Let $K/\QQ$ be a Galois extension of degree $k$ and fix a basis $\omega_1,\cdots, \omega_k$ for the ring of integers $O_K$. Define the form
\[
N(b_1,\dots,b_k) := N_{K/\QQ}(b_1 \omega_1+\cdots + b_k \omega_k)
\]
and the incomplete form
\[
N(b_1,\dots, b_{k-1}) := N(b_1,\dots,b_{k-1},0).
\]
Let $\psi$ be as in Section \ref{smoothweightsection} with $\delta= \log^{-C} X$ for some large $C > 0$. We define
 \begin{align*}
 \Omega(b_1,\cdots,b_{k-1}) :=&\mathbf{1}_{0 \leq b_i \leq   X^{1/(2k)}}
 \frac{1}{\delta} \\
 &\times \int_{1/2 \leq B  \leq \frac{2\sqrt{X}}{C_k}} \psi_B(N(b_1,\cdots,b_{k-1})) \frac{\int \psi_{B}(u) du}{\int_{\mathbf{u} \in [0,X^{1/(2k)}]^{k-1}}  \psi_B(N(u_1,\dots,u_{k-1})) d \mathbf{u}} \frac{dB}{B},     
 \end{align*}

where $C_k > 0$ is a large enough constant so that
 \[
\int_{\mathbf{u} \in [0,C^{1/k}_k ]^{k-1}}  \psi_2(N(u_1,\dots,u_{k-1})) d \mathbf{u} \neq 0, 
 \]
 so that the denominator  in  the definition of $\Omega(b_1,\cdots,b_{k-1})$ is always non-zero. We also define the arithmetic factor
 \[
\kappa_k := \sum_{c \geq 1} \frac{\mu(c) }{c^2} \bigg(\sum_{c \geq 1} \frac{\mu(c) \rho_k(c)}{c^{k}} \bigg)^{-1}, 
 \]
where
\[
\rho_k(c) = |\{b_1,\cdots,b_{k-1} \, (c):  N(b_1,\cdots,b_{k-1}) \equiv 0 \, (c) \}|
\]
We set
 \[
a_n^{(k)} := \kappa_k \sum_{\substack{n= a^2+N(b_1,\cdots,b_{k-1})^2  \\ (a,N(b_1,\cdots,b_{k-1}))=1}}  \Omega(b_1,\cdots,b_{k-1}),
\]
  For typical $b_1,\dots,b_{k-1}$ we have
 \[
  \Omega(b_1,\cdots,b_{k-1}) \asymp X^{1/(2k)},
 \]
 where the upper bound holds for all $b_1,\dots,b_{k-1}$,
 so that
 \[
 \sum_{n \sim X} a_n^{(k)} \asymp X.
 \]
 For the comparison sequence we set
 \[ 
b_n := \sum_{\substack{n=a^2+b^2 \\ (a,b)=1 }}\frac{1}{\delta}\int_{1/2 \leq B  \leq \frac{2\sqrt{X}}{C_k}}\psi_B(b)  \frac{dB}{B},
 \]
 which is the same as before except that we essentially restrict to $b \leq 2 C_k^{-1} \sqrt{X}$. When we write $a_n^{(k)}$ we have suppressed the fact that the sequence depends also on $K$ and the choice of basis $\omega_j$. We will use similar convention below to other quantities, and with the exception of Section \ref{sieve3section} all implied constants are allowed to depend on $K$ and the $\omega_j$.
\subsection{Lemmas}
We need the following lemma for handling the diagonal terms after Cauchy-Schwarz with the Type I and Type II sums. Basically, this ensures that the density of the numbers $a^2+N(b_1,\cdots,b_{k-1})^2 \sim X$ is $ \asymp X^{-1/(2k)+o(1)},$ as expected, so that we get an optimal control for the diagonal terms. Here we do not need to restrict to $b_k=0$, and the lemma is essentially an analogue of the divisor bound $\tau(n) \ll_\eps n^\eps$ for the the number field $K$.  
 \begin{lemma} \label{normreplemma}
The number of representations
\[
N(b_1,\dots,b_k) =b
\]
with $|b_j|\, \leq B$ is $\ll_\eps B^\eps$, where the implied constant may depend on $K$ and the choice of  $\omega_j$.
\end{lemma}
\begin{proof}
Let $b_1,\dots,b_k$ be such that $N(b_1,\dots,b_k) =b$. The principal ideal factorizes uniquely into prime ideals
\[
(b_1\omega_1+\cdots+b_k\omega_k)O_K = \mathfrak{p}_1 \cdots \mathfrak{p}_m.
\]
We have 
\[
N_{K/\QQ}(b'_1\omega_1+\cdots+b'_k\omega_k) =b
\]
only if
\[
(b'_1\omega_1+\cdots+b'_k\omega_k)O_K = \sigma_1(\mathfrak{p}_1) \cdots \sigma_m(\mathfrak{p}_m)
\]
for some $\sigma_j \in \text{Gal}(K/\QQ)$.
The number of choices for $\sigma_j$ which give a different ideal is at most
\[
\leq \prod_{p^\ell || b} (\ell+k)^k \ll \tau(b)^{O_k(1)} \ll_\eps B^\eps 
\]
by a divisor bound. To see this, by multiplicativity of $N_{K/\QQ}$ it suffices to consider the case that
\begin{align*}
    b =  p^{\ell}
\end{align*}
so that $m \leq \ell$ and for some $efg=k$ we have $p O_K= P_1^e \cdots P_g^e$ for some prime ideals $P_i \subseteq O_K$. Then the number of choices for $\sigma_j$ which give a distinct ideal is bounded by the number of ways choosing $\ell$ (possibly repeating) elements from a set of size $k$, which is
\[
\leq \binom{\ell+k-1}{\ell} \leq (\ell+k)^k.
\]

Thus, it suffices to show that the number of units $\eps_0$ of the ring of integers $O_K$ with
\[
\eps_0 (b_1\omega_1+\cdots+b_k\omega_k) = b'_1\omega_1+\cdots+b'_k\omega_k, \quad |b_j'| \,\leq B
\]
is $\ll_\eps B^\eps$. Consider the Minkowski embedding
\[
M:\beta \mapsto (\log |\sigma(\beta) |)_\sigma  \in \R^r,
\]
indexed by the $r =r_1+r_2$ complex embeddings $\sigma: K \hookrightarrow \C$. For every such embedding we have
\[
|\sigma(b_1\omega_1+\cdots+b_k\omega_k)| = |b_1\sigma(\omega_1)+\cdots+b_k\sigma(\omega_k)| \,\ll B,
\]
and 
\[
|\sigma(b_1\omega_1+\cdots+b_k\omega_k)|= \frac{b}{\prod_{\sigma' \neq \sigma} |\sigma'(b_1\omega_1+\cdots+b_k\omega_k)|} \,\gg B^{1-k}.
\]
Thus, our task is reduced to enumerating units $\eps_0$ such that
\[
\|M(\eps_0)\| \,\ll \log B.
\]
The set of units of $O_K$ is mapped by $M$ to a rank $r-1$ lattice in $\R^r$ of determinant $\ll_K 1$, so that by considering a Minkowski-reduced basis (see \cite[Lemma 4.1]{maynard}, for instance) we see that the number of such units is
\[
\ll \log^{r-1} B \ll \log^k B \ll_\eps B^{\eps}.
\]
\end{proof}

For a  matrix $A$ denote by $A^{ij}$ the matrix which is obtained by deleting $i$th row and $j$th column.  For the exponential sums in the next section we need the following two lemmas.
\begin{lemma} \label{Alemma}
Let $\text{\emph{Gal}}(K/\QQ)= \{\sigma_1,\dots,\sigma_k\}$ and define
\[
A:= (a_{ij}), \quad a_{ij} = \sigma_{i}(\omega_j).
\]
Then $\det A \neq 0$, that is, $A:K^k \to K^k$ is invertible. Furthermore, for any fixed $i$ the numbers $\det A^{ij} \in K, \,j \leq k$ are linearly independent over $\QQ$, so that in particular $\det A^{ij} \neq 0$ for all $i,j \leq k$.
\end{lemma}
\begin{proof}
We have $\det(A) \neq 0$ since $\det(A)^2$ is the discriminant of the number field $K/\QQ$, which is always a non-zero integer (cf. \cite[Chapter III.3]{lang}, for instance).

For the second claim suppose that for some $i$ there are $q_1,\dots,q_k \in \QQ$ not all zero such that
\[
\sum_{j \leq k} q_j \det A^{ij} = 0.
\]
For all $\ell \leq k$ there is some $\sigma^{(i,\ell)}  \in \text{Gal}(K/\QQ)$ and $\eps_{i\ell} \in \{\pm \}$ such that 
\[\det(\sigma^{(i,\ell)}( A^{i j })) = \eps_{i\ell} \det A^{\ell j} , \quad j\leq k,\] 
since any
 $\sigma\in \text{Gal}(K/\QQ)$ the  permutes the row vectors of $A$ and so we may choose $\sigma^{(i,\ell)} = \sigma_\ell \sigma_i^{-1}$. Note that $\eps_{i\ell}$ is independent of $j$ as it only depends on the sign of the permutation of rows defined by $\sigma^{(i,\ell)}$. Thus, we get
\[
\sum_{j \leq k}  q_j \eps_{i\ell} \det A^{\ell j } = \sum_{j \leq k}q_j    \det (\sigma^{(i,\ell)}(A^{ij})) = \sigma^{(i,\ell)} \bigg(\sum_{j \leq k} q_j \det A^{ij}\bigg)  = 0.
\]
Then the column vectors of the  matrix
$D =(d_{\ell j})= (\eps_{i\ell}\det A^{\ell j})$ are linearly dependent, that is, for  $\bm{c}_j=(\eps_{i \ell }\det A^{ij})_{\ell \leq k}$ we have
\[
\sum_{j \leq k} q_j \bm{c}_j = \bm{0}.
\] 
Hence $\det D = 0$, which implies  that also for the cofactor matrix $C=((-1)^{\ell+j} \det A^{\ell j})$ we have $\det C = 0$, since $C$ can be obtained from $D$ by the row/column operations of multiplying the rows by $\eps_{i\ell} (-1)^{\ell}$ and the columns by $(-1)^j$. Then $\det A^{-1} =  (\det A)^{-1}\det C = 0$, which is a contradiction, since in the above we have shown that $\det A \neq 0$.  
\end{proof}
We include the following lemma for completeness, even though we will not need it.
\begin{lemma} \label{Flemma}
Let $F \in K[X_1,\dots,X_k]$ be the form defined by
\[
F(X_1,\dots,X_k) := \prod_{i \leq k} \bigg(\sum_{j \leq K} (-1)^{i+j} X_j \det A^{ij} \bigg).
\]
Then $F \in \Z[X_1,\dots,X_k]$.
\end{lemma}
\begin{proof}
We have  $\text{Gal}(K/\QQ) \cong \text{Gal}(K(X_1,\dots,X_k)/\QQ(X_1,\dots,X_k))$, where the isomorphism is given by letting $\sigma \in \text{Gal}(K/\QQ)$ act on the rational functions $P/Q \in K(X_1,\dots,X_k)$ coefficient-wise.
By the definition of $A$ we have
\[
F(X_1,\dots,X_k) = \prod_{\sigma \in \text{Gal}(K/\QQ)}  \bigg(\sum_{j \leq k} (-1)^{i+j}X_j \sigma(\det A^{1j}) \bigg).
\]
Hence, for any $\sigma \in  \text{Gal}(K(X_1,\dots,X_k)/\QQ(X_1,\dots,X_k))$
\[
\sigma F = F,
\]
so that $F \in \QQ[X_1,\dots,X_k]$.  Clearly then $F \in \Z[X_1,\dots,X_k].$
\end{proof}
\section{A Deligne-type bound}
Throughout this section we denote
\[
\bm{h}_{2k}=(h_1,\dots,h_{2k}) \quad \text{and} \quad  \bm{h}=(h_1,\dots,h_{k-1},h_{k+1},\dots,h_{2k-1}).
\]
We shall first consider the exponential sum with the complete norm form with $k$ variables
\[
\sum_{\substack{b_1,c_1,\dots,b_{k},c_{k} \, (p) \\ N(b_1,\dots,b_{k}) \equiv aN(c_1,\cdots,c_{k}) \, (p)}} e_p ( \bm{h}_{2k} \cdot (\bm{b},\bm{c})),
\]
and in the proof of Lemma \ref{deligneboundDlemma} we reduce the case of incomplete forms with $k-1$ variables to this. The square-root bound would be $\ll p^{k-3/2}$, whereas we lose a factor of $p^{3/2}$ and get only $\ll p^{k}$ for the generic $\bm{h}$. This does not matter in the application, where we take a large $k$ so that the relative loss is quite small.

As is often the case with such exponential sums, it turns out to be helpful to consider a more general sum over finite fields $\FF_{p^m}$, where the additive character $e_p(x)$ on $\FF_p$ is replaced by  the additive character on $\FF_{p^m}$ 
\[
e_p( \text{Tr}_{\FF_{p^m}/\FF_p} (x)).
\] 
We need the following lemma, which is equivalent to the rationality of the $L$-function associated to the exponential sum. 
\begin{lemma} \label{rationalitylemma}
\[
\sum_{\substack{b_1,c_1,\dots,b_{k},c_{k} \in \FF_{p^m} \\ N(b_1,\dots,b_{k}) = aN(c_1,\dots,c_{k}) }} e_p ( \text{\emph{Tr}}_{\FF_{p^m}/\FF_p} (\bm{h}_{2k} \cdot (\bm{b},\bm{c})))  = \sum_{j\leq g} \epsilon_j  \lambda_j^m,
\]
where for some fixed $g$ depending only on $k,a, N, \bm{h}_{2k}$ and $\lambda_j \in \C, \epsilon_j \in \{\pm 1\}$ depending only on  $k,p,a, N, \bm{h}_{2k}$.
\end{lemma}
\begin{proof}
We partition the sum into
\[
\sum_{I,J \subseteq \{ 1,\dots, k\}} S(I,J),
\]
where 
\[
S(I,J) =  \sum_{\substack{b_1,c_1,\dots,b_{k},c_{k} \in \FF_{p^m} \\ N(b_1,\dots,b_{k}) = aN(c_1,\dots,c_{k}) \\ b_i=c_j=0 \Leftrightarrow (i,j) \in I\times J}} e_p ( \text{Tr}_{\FF_{p^m}/\FF_p} (\bm{h}_{2k} \cdot (\bm{b},\bm{c}))).
\]
Here $b_i=c_j=0$ means that these variables are $0$ as elements of the field $\FF_{p^m}$ so that for $(i,j) \not \in I \times J$ we have $b_j,c_j \in \FF_{p^m}^\times$. 
We expand the condition $ N(b_1,\dots,b_{k}) = aN(c_1,\cdots,c_{k})$ to get
\[
\frac{1}{p^m} \sum_{w \in \FF_{p^m}} \sum_{\substack{b_1,c_1,\dots,b_{k},c_{k} \in \FF_{p^m}  \\ b_i=c_j=0 \Leftrightarrow (i,j) \in I\times J}} e_p ( \text{Tr}_{\FF_{p^m}/\FF_p} (w(N(b_1,\dots,b_{k}) -aN(c_1,\dots,c_{k}))+ \bm{h}_{2k} \cdot (\bm{b},\bm{c}))).
\]
The contribution from $w=0$ is 
\[
\frac{1}{p^m} \sum_{I,J \subseteq \{ 1,\dots, k\}}\sum_{\substack{b_1,c_1,\dots,b_{k},c_{k} \in \FF_{p^m} \\   b_i=c_j=0 \Leftrightarrow (i,j) \in I\times J }} e_p ( \text{Tr}_{\FF_{p^m}/\FF_p} (\bm{h}_{2k} \cdot (\bm{b},\bm{c}))).
\]
which is 0 or $p^{m(2k-1)}$ if $p| h_j$ for all $j \leq 2k$, so that it is of a suitable form. For the contribution from $w \neq 0$ we have by \cite[Theorem 1]{bombieri} (since the $b_i,c_j$ with $(i,j) \not \in I\times J$ run over  $\FF_{p^m}^\times$)
\[
\begin{split}
& \frac{1}{p^m} \sum_{w \in \FF_{p^m}^\times} \sum_{\substack{b_1,c_1,\dots,b_{k},c_{k} \in \FF_{p^m} \\ b_i=c_j=0 \Leftrightarrow (i,j) \in I\times J }} e_p ( \text{Tr}_{\FF_{p^m}/\FF_p} (w(N(b_1,\dots,b_{k}) -aN(c_1,\dots,c_{k}))+ \bm{h}_{2k} \cdot (\bm{b},\bm{c}))) \\
&= \sum_{i=1}^r \omega_i^m - \sum_{j=1}^s \eta_j^m
\end{split}
\]
for some fixed $r,s$ depending only on  $k,a, N, \bm{h}_{2k}, I,J$ and $\omega_i,\eta_j \in \C$ depending only on  $k,p,a, N, \bm{h}_{2k},I,J$. 
\end{proof}
We also need the following standard lemma (see \cite[Lemma 4.15]{kowalskinotes}, for instance).
\begin{lemma} \label{lambdalemma} Let $n,g \geq 1$ and let $\lambda_1,\dots,\lambda_g \in \C$. Suppose that there are constants $A,B > 0$ such that for every $m \geq 1$ we have for 
\[
\bigg|\sum_{j \leq g} \lambda_j^{mn} \bigg|\leq A B^{mn}.
\]
Then $|\lambda_j| \leq B$ for all $j \in \{1,\dots,g\}$.
\end{lemma}
The purpose of the above two lemmas is that they allow us to reduce bounding an exponential sum over $\FF_p$ to bounding the corresponding sum over $\FF_{q^n}$ for some fixed $q=p^f$. The benefit of this is that after a suitable extension $\FF_q/\FF_p$ the norm form $N(b_1,\dots,b_k)$ factors into a product of $k$ linear forms. 

We need to set up some notations for the next lemma. Recall that for any prime $p$ there are integers $e,f,g$ with $efg=k$ such that
\[
pO_K = (P_1 \dots P_g)^e
\]
for some prime ideals $P_j$ of $K$. The integer $f$ is called the inertia degree of $p$, and we have for all $j \leq g$
\[
O_K/P_j \cong \FF_{p^f}.
\]
We will denote $q = p^f$ and choose some $P_j$ for each prime $p$ (the exact choice is not important, see Remark \ref{choiceremark}). For each prime $p$ we denote the reduction map by
\begin{align} \label{reductionmodpmap}
\pi_p: O_K \to \FF_q.    
\end{align}
We then identify $\FF_p =\pi_p(\Z)$ as a subfield of $\FF_q$.

We have the following Deligne-type bound, which is an explicit version for a special case of the more general stratification result of Fouvry and Katz \cite[Theorem 1.2]{fouvrykatz}. 
 \begin{lemma} \label{deligneboundlemma}
Let $p$ be sufficiently large in terms of $K$ and $\omega_j$. Let $a \in \FF_p$, $t \in \FF_p^\times$, and $q=p^f$. Let $A$ be as in Lemma \ref{Alemma}. Let $F$ be as in Lemma \ref{Flemma} and let $Y_1 \subseteq \mathbb{A}_{\Z}^{2k}$ be defined by
\begin{equation} \label{Fequation}
F(x_1,\dots,x_k) = (-1)^{k-1}a F(x_{k+1},\dots,x_{2k}).
\end{equation}
Let $\bm{h}_{2k} = (\bm{h}_{2k,1},\bm{h}_{2k,2}) \in \FF_p^{2k}$. 

There are algebraic sets $X_j$ with $ \mathbb{A}_{K}^{2k} = X_0 \supseteq X_1 \supseteq \cdots \supseteq X_{2k} \supseteq X_{2k+1} = \emptyset $ such that the following hold. 

(1) For $j \geq 1$ and $\bm{h}_{2k} \in X_j(\FF_q) \setminus  X_{j+1}(\FF_q)$ we have 
 \[
\sum_{\substack{b_1,c_1,\dots,b_{k},c_{k} \in \FF_p \\ N(b_1,\dots,b_{k}) = aN(c_1,\cdots,c_{k}) \, (p)}} e_p ( t \bm{h}_{2k} \cdot (\bm{b},\bm{c})) \ll (a,p)^{1/2}p^{k-1+j/2}.
 \]
For $\bm{h}_{2k} \in X_0(\FF_q) \setminus (Y_1(\FF_q) \cup X_1(\FF_q)) $ we have
 \[
\sum_{\substack{b_1,c_1,\dots,b_{k},c_{k} \in \FF_p \\ N(b_1,\dots,b_{k}) = aN(c_1,\cdots,c_{k}) \, (p)}} e_p ( t \bm{h}_{2k} \cdot (\bm{b},\bm{c})) \ll (a,p)^{1/2} p^{k-1/2}.
 \]
For $\bm{h}_{2k} \in Y_1(\FF_q) \setminus  X_{1}(\FF_q)$ we have
 \[
\sum_{\substack{b_1,c_1,\dots,b_{k},c_{k} \in \FF_p \\ N(b_1,\dots,b_{k}) = aN(c_1,\cdots,c_{k}) \, (p)}} e_p (t \bm{h}_{2k} \cdot (\bm{b},\bm{c})) \ll p^{k}.
 \]

(2) For each $j \geq 1$ we have

\[
X_j = \bigcup_{\substack{j=r+s \\ r,s \leq k}} Z_{r} \times Z_{s},
\] 
where $Z_{r} \subseteq  \mathbb{A}_{K}^{k}$ is a union of the $\binom{k}{r}$ hyperplanes of dimension $k-r$ defined by the $r$ independent linear equations
\begin{equation} \label{linearequations}
\sum_{j \leq k} (-1)^{j} \det(A^{i_s j}) h_j = 0
\end{equation}
for some distinct $i_1,\dots,i_r \in \{1,\dots,k\}$ with $A^{ij}$ as in Lemma \ref{Alemma}.
\end{lemma}
 \begin{proof}
Suppose first that $a \neq 0$. It suffices to consider  
\[
\sum_{\substack{b_1,c_1,\dots,b_{k},c_{k} \in \FF_{q^n} \\ N(b_1,\dots,b_{k}) = aN(c_1,\dots,c_{k})}} e_p ( t\text{Tr}_{\FF_{q^n}/\FF_p} (\bm{h} \cdot (\bm{b},\bm{c})))
 \]
for all $n \geq 1$ with $q=p^f$, since by Lemma \ref{rationalitylemma} for some $g \geq 1$
\[
\sum_{\substack{b_1,c_1,\dots,b_{k},c_{k} \in \FF_{q^n}\\ N(b_1,\dots,b_{k}) = aN(c_1,\dots,c_{k})}} e_p ( t\text{Tr}_{\FF_{q^n}/\FF_p} (\bm{h} \cdot (\bm{b},\bm{c}))) = \sum_{j\leq g} \epsilon_j \lambda_j^{fn}
\]
for some $\lambda_j \in \C, \epsilon_j \in \{\pm 1\}$ depending only on $p$, $a$, $t\bm{h}_{2k}$, and $N$, so that the bound for $\FF_p$ follows by Lemma \ref{lambdalemma} from the corresponding bound for all $n \geq 1$. 

Since $N(b_1,\cdots,b_k)$ is a norm form, it splits into linear factors over $K$ and hence over $\FF_{q}$
 \[
N(b_1,\cdots,b_k) = \prod_{j=1}^k L_j(b_1,\cdots,b_k),
 \]
where $L_j(b_1,\dots,b_k)$ is the $j$th coordinate of $A\bm{b}$ (with $A$ as in Lemma \ref{Alemma}). Note that the linear map 
\begin{equation}
\label{linearmap}
\bm{b} \mapsto A \bm{b}
\end{equation}
is invertible for $p$ sufficiently large, since $\det A \neq 0$ has finitely many prime factors. Thus, after the linear change of variables $(\bm{b},\bm{c}) \mapsto (A\bm{b},A\bm{c})=: (\bm{b'},\bm{c'}),$ it suffices to consider 
  \[
 \sum_{\substack{b'_1,c'_1,\dots,b'_{k},c'_{k} \in \FF_{q^n} \\ b'_1\cdots b'_{k} = a  c'_1\cdots c'_{k}}} e_p ( t\text{Tr}_{\FF_{q^n}/\FF_p} (\bm{h'} \cdot (\bm{b'},\bm{c'}))) 
 \] 
 where $\bm{h'} = (\bm{h'}_1,\bm{h'}_2)$ is obtained from $\bm{h}_{2k}=(\bm{h}_{2k,1},\bm{h}_{2k,2})$ by the linear map 
 \[
 \bm{h'}_j  = (A^{-1})^T \bm{h}_{2k,j},
 \]  
 since
 \[
 \bm{h'}_1 \cdot \bm{b'} =    (\bm{h'}_1)^T \bm{b}' =  ((A^{-1})^T \bm{h}_{2k,1})^T A \bm{b} = \bm{h}_{2k,1} \cdot \bm{b}.
 \]
 We split into separate cases depending on whether $b'_1\cdots b'_{k} =0$ or $\neq 0$.  In the latter case the contribution is
 \[
\sum_{x \in \FF_{q^n}^\times }  \sum_{\substack{b'_1,\dots,b'_{k} \in \FF_{q^n} \\ b'_1\cdots b'_{k} = a  x}} e_p ( t\text{Tr}_{\FF_{q^n}/\FF_p} (\bm{h'}_1 \cdot \bm{b'})) \sum_{\substack{c'_1,\dots,c'_{k} \in \FF_{q^n} \\ c'_1\cdots c'_{k} = x}} e_p ( t\text{Tr}_{\FF_{q^n}/\FF_p} (\bm{h'}_2 \cdot \bm{c'}))
 \]
Suppose first that all of the coordinates of $\bm{h'}_1$ and $\bm{h'}_2$ are $\neq 0$. Denote
\[
a'= a \prod_{\substack{i } } h'_{1,i} \prod_{\substack{i} } (h'_{2,i})^{-1}
\]
and $\psi(y)=e_p ( t\text{Tr}_{\FF_{q^n}/\FF_p}(y))$. After a change of variables the above sum is 
\[
 \begin{split}
 \sum_{x \in \FF_{q^n}^\times }  & \sum_{\substack{b''_1,\dots,b''_{k} \in \FF_{q^n} \\ b''_1\cdots b''_{k} = a'x}} \psi(b''_1 + \cdots +b''_{k}) \sum_{\substack{c'_1,\dots,c'_{k} \in \FF_{q^n} \\ c'_1\cdots c'_{k} = x}}\psi(c''_{1} + \cdots +c''_{k-j_2})  \\
 &= \sum_{x \in \FF_{q^n}^\times } \text{Kl}_{k}(a'x;q^n) \text{Kl}_{k}(x;q^n).
 \end{split}.
\]
 If $a' \neq (-1)^{k-1}$, then this is $\ll (q^n)^{k-1/2}$ by \cite[(6.12)]{polymath}. If $a' =(-1)^{k-1}$, then we use a point-wise bound for the hyper-Kloosterman sums and bound the sum over $x$ trivially, which gives $\ll (q^n)^{k}$. The equation $a' =(-1)^{k-1}$ holds precisely when (\ref{Fequation}) holds, so that this is covered by $Y_1$.
 
Suppose then that exactly $j_1$ and $j_2$ of the coordinates in $\bm{h'}_1$ and $\bm{h'}_2$ are 0, and denote the sets of such indices $J_1$ and $J_2$.  By symmetry we may suppose that $j_1 \geq 1.$ Then after a change of variables and denoting
\[
a'= a \prod_{\substack{i \notin J_1} } h'_{1,i} \prod_{\substack{i \notin J_2} } (h'_{2,i})^{-1},
\]
we get
 \[
 \begin{split}
 \sum_{x \in \FF_{q^n}^\times }  \sum_{\substack{b''_1,\dots,b''_{j_1} \in \FF_{q^n}}}\sum_{\substack{b''_1,\dots,b''_{k-j_1} \in \FF_{q^n} \\ b''_1\cdots b''_{j_1} b''_{j_1+1}\cdots b''_{k} = a'x}} \psi(b''_{j_{1}+1} + \cdots +b''_{k}) \sum_{\substack{c''_1,\dots,c''_{j_1} \in \FF_{q^n}}}\sum_{\substack{c''_1,\dots,c''_{k-j_1} \in \FF_{q^n} \\ c''_1\cdots c''_{j_2} c''_{j_2+1}\cdots c''_{k} = x}} \psi(c''_{j_{2}+1} + \cdots +c''_{k}) 
 \end{split}
 \] 
Making the change of variables
\[
b''_1 \mapsto \frac{a'x }{b_1'' b_2'' \cdots b_{j_1}''}, \quad x \mapsto x c_1'' \cdots c''_{j_2}
\]
the sum becomes
\[
(q^n)^{j_1+j_2-1} \bigg(\sum_{b_1'' \in \FF_{q^n}^\times} \text{Kl}_{k-j_1}(b_1'';q^n)\bigg)\bigg(\sum_{x \in \FF_{q^n}^\times} \text{Kl}_{k-j_2}(x;q^n)\bigg),
\] 
which by \cite[(6.11)]{polymath} is 
\[
\ll(q^n)^{j_1+j_2-1} (q^n)^{(k-j_1)/2}(q^n)^{(k-j_2)/2} \ll (q^{n})^{k-1+(j_1+j_2)/2}.
\]
We now note that (with $A^{ij}$ as in Lemma \ref{Alemma}) 
\[
(A^{-1})^T = \frac{1}{\det A} ((-1)^{i+j} \det A^{ij})_{i,j \leq k},
\]
so that the equations $h'_i = 0$ are precisely of the form (\ref{linearequations}). By Lemma \ref{Alemma} the equations are independent (for $p$ sufficiently large), so that $Z_r$ has dimension $k-r$ for each $r$.
 
Consider then the part where $b'_1\cdots b'_{k}=0$. If $a \neq 0$,  we also have $c'_1\cdots c'_{k}=0.$ Thus, this part of the sum is 
\[
\sum_{\substack{I,J \subseteq \{1,\dots,k\} \\ I,J \neq \emptyset }} \sum_{\substack{b'_1,c'_1,\dots,b'_{k},c'_{k} \in \FF_{q^n} \\ b'_i=c'_j=0 \Leftrightarrow (i,j) \in I\times J}} e_p(t \text{Tr}_{\FF_{q^n}/\FF_p} (\bm{h'} \cdot (\bm{b'},\bm{c'})) \ll \min\{(p^{mn})^{2k-2} ,\prod_{i \leq 2k} (h'_i,p)^{mn} \},
\]
which is bounded by $\max\{p^{k-1/2},p^{k-1+j/2}\}$. Finally, if $a = 0$, then $b'_1\cdots b'_{k}=0$ and $c_1',\dots,c_k'$ are free and we have
\[
\sum_{\substack{I \subseteq \{1,\dots,k\} \\ I \neq \emptyset }} \sum_{\substack{b'_1,c'_1,\dots,b'_{k},c'_{k} \in \FF_{q^n} \\ b'_i=0 \Leftrightarrow i \in I}} e_p(t \text{Tr}_{\FF_{q^n}/\FF_p} (\bm{h'} \cdot (\bm{b'},\bm{c'})) \ll \min\{(p^{mn})^{2k-1} ,\prod_{i \leq 2k} (h'_i,p)^{mn} \},
\]
which is
\[
\ll (a,p)^{mn/2}\max\{(p^{mn})^{k-1/2},(p^{mn})^{k-1+j/2}\}
\]
with equality at $j=2k-1, 2k$.
\end{proof}
\begin{remark} \label{choiceremark}
The coefficients $\det A^{ij}$ of the equation (\ref{linearequations}) are in $\FF_q$ by the reduction to $O_K/P_j$, where $pO_K= P_1 \cdots P_g$ and $q=p^f$. If $f=1$, then they are in $\FF_p$ and we get a system of $r$ linear equations for $\bm{h} \in \FF_p^k$. Since any $\sigma \in \text{Gal}(K/\QQ)$ permutes the primes $P_1,\dots, P_g$, we see that changing $P_j$ in $\FF_q \cong O_K/P_j$ simply changes the set of $r$ indices $j \leq k$ such that $h'_j=0$, so that the choice of the reduction modulus $P_j$ is inconsequential for our application, since we get the same bound for any set of $r$ indices. We can eliminate the use of the (countable) axiom of choice simply by averaging over all choices $P_j \in \{P_1,\dots,P_g\}$. If $f=k$, then by Lemma \ref{Alemma} if one of the coordinates of $\bm{h}'_i$ is 0, then all are 0 (for $p \gg 1$), so that the only solution $\bm{h} \in \FF_p^k$ is $\bm{h}=\bm{0}$.
\end{remark}
\begin{remark}
Using the Hasse-Davenport relation one can show that the function
\[
K(a;q) := \sum_{\substack{b_1,\dots,b_k \in \FF_q \\ N_{\FF_q / \FF_p}(b_1,\dots,b_{k}) \equiv a \, (p)}} \psi(b_1+\cdots + b_k)
\]
is a constant multiple of the hyper-Kloosterman sum $\text{Kl}_{k}(a;q)$ (see \cite[Applications de la formule des traces aux sommes trigonom\'etriques (7.2.5)]{deligne}, thanks to Emmanuel Kowalski for pointing this out).
\end{remark}

We need to set up some more notations for the following lemma, which uses the Chinese remainder theorem to combine Lemma \ref{deligneboundlemma} for different primes $p$. Then it may happen that $\bm{h}$ lie in different $X_j$ for different primes $p$, where $X_j$ are as in Lemma \ref{deligneboundlemma}. For any prime $p$ we extend the reduction map $\pi_p:O_k \to \FF_q$ defined in \eqref{reductionmodpmap} to a map
\[
\bm{\pi}_p: O_K^{2k} \to \FF_q^{2k}: \quad (z_1,\dots, z_{2k}) \mapsto (\pi_p(z_1),\dots,\pi_p(z_{2k})).
\]
Let $D_1=p_1\dots p_\ell$ be a square-free integer and denote $q_i=p_i^{f_i}$. For $(j_1,\dots,j_\ell)$, $j_i \leq k$, we define
\begin{align*}
  X_{(j_1,\dots,j_\ell)}(D_1) := \{\bm{z} \in O_{K}^{2k} :  \bm{\pi}_{p_{i}}(\bm{z}) \in X_{j_i}(\FF_{q_i}) \, \forall i \leq \ell \}  
\end{align*}
That is, we have
\[
X_{(j_1,\dots,j_\ell)}(D_1)= \bigcap_{i \leq \ell} \bm{\pi}_{p_i}^{-1} X_{j_i}(\FF_{q_i}).
\]
Then we have the following lemma.
\begin{lemma} \label{deligneboundDlemma}
Let $h_1,\dots,h_{k-1},h_{k+1},\dots,h_{2k-1} \in \Z$. Let $D_1=p_1\cdots p_\ell$ denote the largest square-free divisor such that $(D/D_1,D_1)=1$ and  let $X_j$ be as in Lemma \ref{deligneboundlemma}. For $j \in \Z$ denote
\[
j^+ := \max\{2,j\}.
\]
Then
 \[
 \begin{split}
& \sum_{\substack{b_1,c_1,\dots,b_{k-1},c_{k-1} \, (D) \\ N(b_1,\dots,b_{k-1}) \equiv aN(c_1,\cdots,c_{k-1}) \, (D)}} e_{D} ( \bm{h} \cdot (\bm{b},\bm{c})) \\
& \ll (a,D_1)^{1/2}\tau(D)^{O(1)}\bigg(\frac{D}{D_1}\bigg)^{2k-2}  D_1^{k-1} \sum_{0 \leq j_1,\dots,j_\ell \leq 2k} p_1^{j_1^+ /2} \cdots p_\ell^{j_\ell^+ / 2} \frac{1}{D_1^2}\sum_{\substack{h_{k},h_{2k} \, (D_1) \\ \bm{h}_{2k} \in X_{(j_1,\dots,j_\ell)}(D_1)}} 1 .
\end{split}
 \]
\end{lemma}
\begin{proof}
By the Chinese remainder theorem we get
\[
\prod_{p^m|| D} \sum_{\substack{b_1,c_1,\dots,b_{k-1},c_{k-1} \, (p^m) \\ N(b_1,\dots,b_{k-1}) \equiv aN(c_1,\cdots,c_{k-1}) \, (p^m)}} e_{p^m} \bigg( \bigg(\frac{D}{p^m}\bigg)^{-1}\bm{h} \cdot (\bm{b},\bm{c})\bigg),
\]
where the inverse $(\frac{D}{p^m})^{-1}$ is computed modulo $p^m$. By expanding the conditions $b_k \equiv c_k \equiv 0 \, (D_1)$, this is bounded by
\[
\begin{split}
\frac{1}{D_1^2} \sum_{h_{k},h_{2k} \, (D_1)} \bigg( &\prod_{p^m|| D/D_1} \sum_{\substack{b_1,c_1,\dots,b_{k-1},c_{k-1} \, (p^m) \\ N(b_1,\dots,b_{k-1}) \equiv aN(c_1,\cdots,c_{k-1}) \, (p^m)}} 1 \bigg)  \\
&\times \bigg( \prod_{p| D_1} \bigg|\sum_{\substack{b_1,c_1,\dots,b_{k},c_{k} \, (p) \\ N(b_1,\dots,b_{k}) \equiv aN(c_1,\cdots,c_{k}) \, (p)}} e_{p} \bigg( \bigg(\frac{D}{p}\bigg)^{-1}\bm{h}_{2k} \cdot (\bm{b},\bm{c})\bigg)\bigg| \bigg).
\end{split}
\]
For $p^m|| D/D_1$ we use the trivial bound
\[
\sum_{\substack{b_1,c_1,\dots,b_{k-1},c_{k-1} \, (p^m) \\ N(b_1,\dots,b_{k-1}) \equiv aN(c_1,\cdots,c_{k-1}) \, (p^m)}} 1  \ll (p^m)^{2k-2}
\]
For $p| D_1$ we apply Lemma \ref{deligneboundlemma} (with $t= (D/p)^{-1}$) to get for some $j=j(p)$ with $\bm{\pi}_p(\bm{h}_{2k}) \in X_j(\FF_q)$ that
\begin{align*}
\bigg|\sum_{\substack{b_1,c_1,\dots,b_{k},c_{k} \, (p) \\ N(b_1,\dots,b_{k}) \equiv aN(c_1,\cdots,c_{k}) \, (p)}} e_{p} \bigg( \bigg(\frac{D}{p}\bigg)^{-1}\bm{h}_{2k} \cdot (\bm{b},\bm{c})\bigg)\bigg| 
\ll  (a,p)^{1/2}  p^{k-1+j^+/2}.
\end{align*}
Thus, for some $(j_1,\dots,j_\ell)$ with $\bm{h}_{2k} \in X_{(j_1,\dots,j_\ell)}(D_1)$ we have
\begin{align*}
    \bigg( \prod_{p| D_1} \bigg|\sum_{\substack{b_1,c_1,\dots,b_{k},c_{k} \, (p) \\ N(b_1,\dots,b_{k}) \equiv aN(c_1,\cdots,c_{k}) \, (p)}} e_{p} \bigg( \bigg(\frac{D}{p}\bigg)^{-1}\bm{h}_{2k} \cdot (\bm{b},\bm{c})\bigg)\bigg| \bigg) \ll \tau(D)^{O(1)} \prod_{p_i|D_1}  p_i^{k-1+j_i^+/2} \\
    \ll\tau(D)^{O(1)}  D_1^{k-1}\prod_{p_i|D_1} p_i^{j_i^+/2}.
\end{align*}
\end{proof}
\begin{remark}
Note that here we have not used the fact that in $X_0\setminus Y_1$ we get the superior bound $D_1^{k-1/2}$. This is not crucial for our application and using this weaker bound simplifies greatly the upcoming computations. We have included the characterization of $Y_1$ in Lemma \ref{deligneboundlemma} just for the sake of completeness. Making use of this would only marginally widen our Type II range (Proposition \ref{typeiikprop}).
\end{remark}
Let $H_r$ denote any of the hyper-planes defining $Z_r$ in Lemma \ref{deligneboundlemma}. We need the following bound for the number of $(h_1,\dots,h_k)\in \Z^k \subseteq O_K^k$ which are in $H_r(\FF_q)$ after reduction modulo some $P_j$, where as before $p O_K =P_1\cdots P_g$ and $q = p^f$. For a square-free integer $D_1=p_1\cdots p_{\ell}$ and $\bm{r}=(r_1,\dots,r_\ell) \in \Z_{\geq 0}^{\ell}$, denote 
\begin{align*}
    H_{(r_1,\dots,r_\ell)}(D_1) := \{\bm{z} \in O_{K}^{k} : \bm{\pi}_{p_{i}}(\bm{z}) \in H_{r_i}(\FF_{q_i}) \, \forall i \leq \ell \}  
\end{align*}
Then the following lemma is simply a generalization of the elementary bound
\[
\sum_{h \leq H } \bm{1}_{h \equiv a \, (q)} \ll 1 +\frac{H}{q}.
\]
\begin{lemma} \label{h1boundlemma}
Let $H \geq 1$ and let $D_1=p_1\cdots p_\ell$ be square-free. Let $0 \leq r_1,\dots,r_\ell \leq k$. Let 
\[
D_1(\bm{r}) :=  \prod_{\substack{i \leq \ell \\ r_i =0}} p_i
\]
 and define
\[
\bm{r}' =(r_1',\dots,r_\ell'), \quad r_i' = \max\{r_i-1,0\}.
\] 
Then for all $i \leq \ell$ there are $r'_{ij} \in\{0,1\}, \, j  \leq k-1$ such that
\begin{equation} \label{r'condition}
\sum_{j \leq k-1} r'_{ij} = r_i'
\end{equation}
and
\begin{equation*} 
  \sum_{\substack{ h_1,\dots,h_{k}\\ |h_i| \leq H, \, i \leq k-1 \\ h_k \leq D_1}} \mathbf{1}_{\bm{h} \in H_{(r_1,\dots,r_\ell)}(D_1)} \ll D_1(\bm{r})  \prod_{j \leq k-1}\bigg( 1 + \frac{H}{p_1^{r'_{1j}}\dots p_\ell^{r'_{\ell j}}} \bigg).  
\end{equation*}
Furthermore, if $I$ is the set of indices such that for $i \in I$ we have $r_i=k$, then
\[
\sum_{\substack{ h_1,\dots,h_{k}\\ |h_i| \leq H, \, i \leq k-1 \\ h_k \leq D_1 \\ (h_1,\dots,h_{k-1}) \neq \bm{0}}} \mathbf{1}_{\bm{h} \in H_{(r_1,\dots,r_\ell)}(D_1)} \ll  \mathbf{1}_{H \geq \prod_{i \in I} p_i} D_1(\bm{r})  \prod_{j \leq k-1}  \bigg( 1 + \frac{H}{p_1^{r'_{1j}}\dots p_\ell^{r'_{\ell j}}} \bigg),
\]
where we have $r'_{ij}=1$ for $i \in I$.
\end{lemma}

\begin{proof}
If we fix $h_1,\dots,h_{k-1}$, then  for every $i$ such that $r_i \neq 0$ and $p_i \gg 1$ there is at most one $h_k \in \FF_{q_i}$ with $(h_1,\dots,h_k)  \in H_{r_i}(\FF_{q_i})$, since by Lemma \ref{Alemma}  the coefficients $\det A^{ik} \neq 0$ for $p_i\gg 1$. For $p_i \ll 1$ the number of  $h_k \in \FF_{q_i}$ is  trivially $\ll 1$. Thus, the number of $h_k \leq D_1$ is $\ll D_1(\bm{r})$, and the numbers $h_1,\dots,h_{k-1}$ are restricted to a set of type
\[
H'_{\bm{r}'}(D_1) \subseteq O_K^{k-1},
\]
which is a set such that  $\bm{z} \in H'_{\bm{r}'}(D_1)$ if for every $i \leq \ell$
\[
\pi_{p_i}(\bm{z})   \in H'_{r'_i}(\FF_{q_i})
\]
for some $H'_{r'_i} \subseteq \mathbb{A}^{k-1}_K$  which is a hyperplane of dimension $k-r_i'$. That is, in this first step we have simply solved for $h_k$ in terms of $h_1,\dots,h_{k-1}$ in one of the equations and substituted this to the remaining equations to get $r_i'-1$  independent equations in $k-1$ variables.

It then suffices to show that
\[
\sum_{\substack{ h_1,\dots,h_{k-1}\\ |h_i| \leq H}} \mathbf{1}_{\bm{h} \in H'_{\bm{r}'}(D_1)} \ll \prod_{j \leq k}\bigg( 1 + \frac{H}{p_1^{r'_{1j}}\dots p_\ell^{r'_{\ell j}}} \bigg),
\]
which we will do using induction on the number of summation variables $k':=k-1$. The case of $k'=1$ is trivially true, so consider the case of general $k'$. Fixing $h_1,\dots, h_{k'-1}$, for each $i \leq \ell$ there are two possibilities, the set of $h_{k'} \in \FF_{q_i}$  such that 
\[
(h_1,\dots,h_{k'}) \in H'_{r'_i}(\FF_{q_i})
\]
(that is, the fibre) is either zero dimensional (a point) or one dimensional (a line) as a subvariety of $\mathbb{A}^{1}_{\FF_{q_i}}$. Denoting this dimension by $d_{ik'},$ we let $r'_{ik'}=1-d_{ik'}$, and then the number of $|h_{k'}| \leq H$  with $(h_1,\dots,h_{k'}) \in H'_{r'_i}(\FF_{q_i})$ is
\[
\ll 1 + \frac{H}{p_1^{r'_{1k'}}\dots p_\ell^{r'_{\ell k'}}},
\]
and then $(h_1,\dots,h_{k'-1})$ live in a  hyperplane 
\[
H''_{\bm{r}''}(D_1) \subseteq O_K^{k'-1},
\]
where $\bm{r}'' = (r_1'-r'_{1k'},\dots r_\ell'-r'_{\ell k'})$, and we can apply the induction hypothesis. Clearly we get (\ref{r'condition}), since if $k' \leq r'_i$ we must have $d_{ik'}=0$ as we have $r_i'$ independent equations.

To get the second bound we note that $r_i=k$ means that $h'_j \equiv 0 \, (p_i)$ for all $j \leq k$, which implies that  $h_j \equiv 0 \, (p_i)$ for all $j \leq k$. We sum over $(h_1,\dots,h_{k-1}) \neq \bm{0}$, so suppose by symmetry that $h_{k-1} \neq 0$. Then $h_{k-1} \equiv 0 \, ( \prod_{i\in I} p_i)$ is non-trivial and there are no solutions if $H <  \prod_{i \in I} p_i$. Hence, the number of $h_{k-1}$ for fixed $h_1,\dots,h_{k-2}$ is
\[
\ll  \mathbf{1}_{H \geq \prod_{i \in I} p_i}  \bigg( 1 + \frac{H}{p_1^{r'_{1j}}\dots p_\ell^{r'_{\ell j}}} \bigg).
\]
\end{proof}
\begin{remark}
It is important for our applications that we get full savings in the longest sum $h_k \leq D_1$, so that for large $r_i$ we can cancel the losses from  expanding the condition $b_k =0$ in Lemma \ref{deligneboundDlemma}. 
\end{remark}
\begin{lemma} \label{productboundlemma}
Let $D_1=p_1\cdots p_\ell$ and  $\beta:= \log H/ \log D_1$. With the same notations as above, we have
    \begin{align*}
        \prod_{j \leq k-1}\bigg( 1 + \frac{H}{p_1^{r'_{1j}}\dots p_\ell^{r'_{\ell j}}} \bigg) \ll \tau(D_1)^{O(1)} H^{k-1} \prod_{i \leq  \ell} p_i^{-\min\{1,\beta\} r_i' }.
    \end{align*}
\end{lemma}
\begin{proof}
    Let $\alpha_i = \log p_i / \log D_1$, so that
\[
\alpha_1 + \cdots + \alpha_\ell = 1.
\]
We  have  (by $\sum_{j \leq k-1} r_{ij}' = r_i'$)
\[
\begin{split}
&\prod_{j \leq k-1}\bigg( 1 + \frac{H}{p_1^{r'_{1j}}\dots p_\ell^{r'_{\ell j}}} \bigg) =\prod_{j \leq k-1}\bigg( 1 + \frac{H^{\alpha_1 + \cdots + \alpha_\ell }}{p_1^{r'_{1j}}\dots p_\ell^{r'_{\ell j}}} \bigg) \leq \prod_{j \leq k-1}\prod_{i \leq \ell}\bigg( 1 + \frac{H^{\alpha_i}}{p_i^{r'_{ij}}} \bigg) \\ 
& = \prod_{i \leq \ell} \prod_{j \leq k-1}\bigg( 1 + \frac{H^{\alpha_i}}{p_i^{r'_{ij}}} \bigg)  = \prod_{i \leq \ell} \bigg( 1+ \frac{H^{\alpha_i}}{p}\bigg)^{r_i'}\bigg( 1+ H^{\alpha_i}\bigg)^{k-1-r_i'}
\\
&\ll \tau(D_1)^{O(1)} \prod_{i \leq \ell} \bigg( 1 + \frac{H^{\alpha_i r_i'} }{p_i^{r'_{i}}} \bigg) H^{\alpha_i(k-1-r_i')} =  \tau(D_1)^{O(1)} \prod_{i \leq \ell} \bigg( H^{\alpha_i(k-1-r_i')} + \frac{H^{\alpha_i (k-1)} }{p_i^{r'_{i}}} \bigg)  \\
&=\tau(D_1)^{O(1)}\prod_{i \leq \ell} \bigg( p_i^{\beta (k-1-r_i')} + p_i^{\beta(k-1) - r'_{i}} \bigg)  \ll \tau(D_1)^{O(1)} H^{k-1} \prod_{i \leq \ell} p_i^{- \min \{1,\beta\}r_i'}.
\end{split}
\]
\end{proof}

Combining Lemmas \ref{deligneboundDlemma}, \ref{h1boundlemma}, and \ref{productboundlemma} we finally get our key bound, using the fact that if $\bm{h} \neq \bm{0}$ then $(h_1,\dots,h_{k-1}) \neq \bm{0}$ or $(h_{k+1},\dots,h_{2k-1}) \neq \bm{0}$.
\begin{lemma}\label{sumoverhlemma}
Let $\beta :=\log H / \log D_1$. We have
\[
\begin{split}
&\sum_{\substack{|h_j| \ll H , \, j \in \{1,\dots,k-1,k+1,\dots,2k-1\}\\ \bm{h} \neq \bm{0}}} \bigg| \sum_{\substack{b_1,c_1,\dots,b_{k-1},c_{k-1} \, (D) \\ N(b_1,\dots,b_{k-1}) \equiv aN(c_1,\cdots,c_{k-1}) \, (D)}} e_{D} ( \bm{h} \cdot (\bm{b},\bm{c}))\bigg| \\
\ll &  (a,D_1)^{1/2}\tau(D_1)^{O(1)} \bigg(\frac{D}{D_1}\bigg)^{2k-2}  H^{2k-2} D_1^{k-1}    \\
& \times (D_1 + D_1^{k(1/2-\beta)+\beta-1} + D_1^{(2k-1)(1/2-\beta) +2\beta- 2} H^{1/2-\beta} \\
& \hspace{50pt} + D_1 H^{2k(1/2-\beta)+2\beta-3}  + D_1^{k(1/2-\beta)+\beta-1} H^{k(1/2-\beta) + \beta-1} )  .
\end{split}
\]
\end{lemma}
\begin{proof}
  By applying Lemma \ref{deligneboundDlemma} it suffices to show that
 \begin{align*}
     &\sum_{\substack{|h_j| \ll H , \, j \in \{1,\dots,k-1,k+1,\dots,2k-1\}\\ \bm{h} \neq \bm{0}}}  \sum_{0 \leq j_1,\dots,j_\ell \leq 2k} p_1^{j_1^+ /2} \cdots p_\ell^{j_\ell^+ / 2} \frac{1}{D_1^2}\sum_{\substack{h_{k},h_{2k} \, (D_1) \\ \bm{h}_{2k} \in X_{(j_1,\dots,j_\ell)}(D_1)}} 1 \\
     \ll &\tau(D_1)^{O(1)} H^{2k-2} H^{\max\{1,k(1/2-\beta)+\beta-1,2k(1/2-\beta) +2\beta-2\}} (D_1/H)^{\max\{1,k(1/2-\beta)+\beta-1,(2k-1)(1/2-\beta) +2\beta- 2\}}.
 \end{align*}
The left-hand side is bounded by
\begin{align*}
\max_{\substack{0 \leq r_1,\dots,r_\ell,s_1,\dots,s_\ell \leq k \\ H_{r_i},H_{s_i} }} \frac{\tau(D_1)^{O(1)}}{D_1^2} \prod_{i \leq \ell} p_i^{(r_i+s_i)^+/2} \sum_{\substack{ h_1,\dots,h_{2k}\\ |h_i| \leq H, \, i \neq k,2k \\ h_k,h_{2k} \leq D_1 \\ (h_1,\dots,h_{k-1},h_{k+1},\dots,h_{2k-1}) \neq \bm{0}}} \mathbf{1}_{\bm{h}_{2k,1} \in H_{(r_1,\dots,r_\ell)}(D_1)}  \mathbf{1}_{\bm{h}_{2k,2} \in H_{(s_1,\dots,s_\ell)}(D_1)} 
\end{align*}
where the maximum runs over hyperplanes $H_{r_i},H_{s_i}$ which define $Z_{r_i} \times Z_{s_i}$ as in \eqref{linearequations} when applying Lemma \ref{deligneboundlemma} with $p=p_i$. By symmetry we may assume that $(h_1,\dots,h_{k-1}) \neq \bm{0}$. By applying Lemma \ref{h1boundlemma} this is then bounded by
\begin{align*}
     \max_{\substack{0 \leq r_1,\dots,r_\ell,s_1,\dots,s_\ell \leq k  }} \mathbf{1}_{H \geq \prod_{i\in I} p_i }\frac{\tau(D_1)^{O(1)}}{D_1^2} \prod_{i \leq \ell} p_i^{(r_i+s_i)^+/2} D_1(\bm{r}) D_1(\bm{s}) \prod_{j \leq k-1}\bigg( 1 + \frac{H}{p_1^{r'_{1j}}\dots p_\ell^{r'_{\ell j}}} \bigg) \\
     \prod_{j \leq k-1}\bigg( 1 + \frac{H}{p_1^{s'_{1j}}\dots p_\ell^{s'_{\ell j}}} \bigg) ,
\end{align*}
where $I$ denotes the set of indices $i \leq \ell$ such that $r_i=k$. Let 
\[
\hat{r}_i = \mathbf{1}_{r_i=0}, \quad \hat{s}_i :=\mathbf{1}_{s_i=0}, 
\]
so that 
\begin{align*}
D_1(\bm{r}) D_1(\bm{s})  = \prod_{i \leq \ell} p_i^{\hat{r}_i+\hat{s}_i}
\end{align*}
By Lemma \ref{productboundlemma} we reduce to bounding
\begin{align*}
      \max_{\substack{0 \leq r_1,\dots,r_\ell,s_1,\dots,s_\ell \leq k  }}\tau(D_1)^{O(1)} H^{2k-2} \mathbf{1}_{H \geq \prod_{i\in I} p_i }  \prod_{i \leq \ell} p_i^{(r_i+s_i)^+/2+ \hat{r}_i+\hat{s}_i - (r_i'+s_i') \min\{1,\beta\} -2}. 
\end{align*}
For $\beta > 1$ the exponent of $p_i$ is 
\begin{align*}
(r_i+s_i)^+/2+ \hat{r}_i+\hat{s}_i - (r_i'+s_i')  -2 \leq 1   
\end{align*}
 and we get
\[
 \ll \tau(D_1)^{O(1)} H^{2k-2} D_1
\]
which is sufficient. Thus, let $\beta \leq 1$. For $r_i+s_i \leq 2$ we have
\begin{align*}
    (r_i+s_i)^+/2+ \hat{r}_i+\hat{s}_i - (r_i'+s_i') \beta -2 \leq 1
\end{align*}
For $r_i+s_i > 2$  and $r_i=0$ or $s_i=0$ we get
\[
 (r_i+s_i)^+/2+ \hat{r}_i+\hat{s}_i - (r_i'+s_i') \beta -2 \leq \max\{r_i(1/2-\beta) +\beta- 1,s_i(1/2-\beta) +\beta- 1\}.
\]
For $r_i+s_i > 2$  and  $r_i \neq 0 \neq s_i$ we get
\[
 (r_i+s_i)^+/2+ \hat{r}_i+\hat{s}_i - (r_i'+s_i') \beta -2 = (r_i+s_i)(1/2-\beta)  +2\beta - 2.
\]
Thus, we get
\begin{align*}
 \ll    \max_{\substack{0 \leq r_1,\dots,r_\ell,s_1,\dots,s_\ell \leq k  }}\tau(D_1)^{O(1)} H^{2k-2} \mathbf{1}_{H \geq \prod_{i\in I} p_i }  \prod_{i \leq \ell} p_i^{\max\{1,r_i(1/2-\beta)+\beta - 1,s_i(1/2-\beta) +\beta- 1,(r_i+s_i)(1/2-\beta) +2\beta- 2\}}. 
\end{align*}
If $\beta \in [1/2,1],$ then the exponent is $\leq 1$ so we get
\begin{align*}
   \tau(D_1)^{O(1)} H^{2k-2} D_1.
     \end{align*}
For $\beta < 1/2$ we may assume that $s_i \geq r_i$ (since $H \geq \prod_{i\in I} p_i$ with $I=\{i:r_i=k\}$ ) and we get
     \begin{align*}
       \max_{\substack{0 \leq r_1,\dots,r_\ell,s_1,\dots,s_\ell \leq k  }}\tau(D_1)^{O(1)} H^{2k-2} \mathbf{1}_{H \geq \prod_{i\in I} p_i }  \prod_{i \leq \ell} p_i^{\max\{1,s_i(1/2-\beta) +\beta- 1,(r_i+s_i)(1/2-\beta)+2\beta - 2\}}    \\
       \leq   \max_{\substack{k-1 \leq r_1,\dots,r_\ell \leq k  }}\tau(D_1)^{O(1)} H^{2k-2}  \mathbf{1}_{H \geq \prod_{i\in I} p_i }  \prod_{i \leq \ell} p_i^{\max\{1,k(1/2-\beta)+\beta-1,(r_i+k)(1/2-\beta)+2\beta - 2\}},
     \end{align*}
     since for $\beta< 1/2$ the first bound is increasing in $s_i$. Recalling that $i \in I$ iff $r_i=k$, the above is maximized if we choose $I_0$ so that $\prod_{i\in I_0} p_i$ is as close as possible to $H$ without exceeding it and for $i \not \in I_0$ we have $r_i=k-1$, so that we get
\begin{align*}
    &\ll \tau(D_1)^{O(1)} H^{2k-2} \prod_{i \in I_0}p_i^{\max\{1,k(1/2-\beta)+\beta-1,2k(1/2-\beta)+2\beta - 2\}}\prod_{i  \not \in I_0}p_i^{\max\{1,k(1/2-\beta)+\beta-1,(2k-1)(1/2-\beta) +2\beta- 2\}} \\
  &  \ll  \tau(D_1)^{O(1)} H^{2k-2} H^{\max\{1,k(1/2-\beta)+\beta-1,2k(1/2-\beta) +2\beta-2\}} (D_1/H)^{\max\{1,k(1/2-\beta)+\beta-1,(2k-1)(1/2-\beta) +2\beta- 2\}}.
\end{align*}
\end{proof}

We also require the following Lang-Weil type bound for the number of points on the variety. 
\begin{lemma} \label{langweillemma}
For $a \not \equiv 0 \, (p)$ We have
\[
\sum_{\substack{b_1,c_1,\dots,b_{k-1},c_{k-1} \, (p) \\ N(b_1,\dots,b_{k-1}) \equiv aN(c_1,\cdots,c_{k-1}) \, (p)}} 1 = p^{2k-3} (1+ O(p^{-1/2})).
\]
\end{lemma}
\begin{proof}
By Lemma \ref{rationalitylemma} it suffices to show that for $q=p^f$ we have for all $n \geq 1$
\[
\sum_{\substack{b_1,c_1,\dots,b_{k-1},c_{k-1} \in \FF_{q^n} \\ N(b_1,\dots,b_{k-1}) = aN(c_1,\cdots,c_{k-1})}} 1 = (q^n)^{2k-3} (1+ O((q^n)^{-1/2}))
\]
Th form $N(b_{1},\dots,b_k)$ splits into linear factors of $\FF_q$, so that after a linear change of variables we are counting solutions to
\begin{equation} \label{curvekeq}
b_1' \cdots b'_{k-1} L(b_1',\dots,b'_{k-1}) = a c_1' \cdots c'_{k-1} L(c_1',\dots,c'_{k-1})
\end{equation}
for
\[
L(b_1',\dots,b_{k-1}') = \alpha_1 b_1' + \cdots \alpha_{k-1} b'_{k-1}
\]
with some $\alpha_j \neq  0.$ The number of points $(b_1',\dots,b_{k-1}',c_1',\dots,c_{k-1}')$ where $b_2' \cdots b'_{k-1} = 0$ or $c_2' \cdots c'_{k-1} = 0$ is clearly $\ll (q^n)^{2k-4}$, so we may assume that $b_2' \cdots b'_{k-1} \neq 0$ and $c_2' \cdots c'_{k-1} \neq 0$. Fix $b_2',\dots,b_{k-1}'$ and $c_2',\dots,c_{k-1}'$ consider the affine curve defined by the equation (\ref{curvekeq}) in variables $b_1',c_1'$. If the curve is non-singular, then by the Weil bound the number of points $(b_1',c_1')$ is $q^{n}+O((q^n)^{-1/2})$. Thus, it suffices to show that the curve is non-singular for for all but $\ll (q^{n})^{2k-4-1/2}$ of the variables $b_2',\dots,b_{k-1}',c_2',\dots,c_{k-1}'$. By the product rule we see that a point is singular only if
\[
b_2' \cdots b'_{k-1}(\alpha_1 b_1' + L(b_1',\dots,b'_{k-1})) =  c_2' \cdots c'_{k-1}(\alpha_1 c_1' + L(c_1',\dots,c'_{k-1})) =0.
\]
Since $b_2' \cdots b'_{k-1} \neq 0$ and $c_2' \cdots c'_{k-1} \neq 0$, a point is singular only if
\[
\alpha_1 b_1' + L(b_1',\dots,b'_{k-1}) =  \alpha_1 c_1' + L(c_1',\dots,c'_{k-1}) = 0,
\]
so that $b'_1 = L_1(b_2',\dots,b'_{k-1})$ and $c'_1 = L_1(c_2',\dots,c'_{k-1})$ for some linear form $L_1.$ Substituting these into (\ref{curvekeq}) we get
\[
L(b_2',\dots,b'_{k-1})^2 b_2' \cdots b'_{k-1} = a L(c_2',\dots,c'_{k-1})^2 c_2' \cdots c'_{k-1}, 
\]
so that the number of such $b_2',\dots,b_{k-1}',c_2',\dots,c_{k-1}'$ is $\ll (q^{n})^{2k-5}$.
\end{proof}
Similarly as in Section \ref{weilsection}, using Lemma \ref{langweillemma} with the Chinese remainder theorem we get the following.
\begin{lemma} \label{langweilDlemma}
Assume that $D/D_1 \leq Y$ and $(a,D) \leq Y$. Then we have for some $\eps^{(k)}_d(a) \ll Y d \tau(d)^{O(1)}$
\[
\sum_{\substack{b_1,c_1,\dots,b_{k-1},c_{k-1} \, (D) \\ N(b_1,\dots,b_{k-1}) \equiv aN(c_1,\cdots,c_{k-1}) \, (D)}} 1 = \sum_{\substack{d|D \\ d <Y^8 }} \frac{D^{2k-3}}{d} \eps^{(k)}_d(a) + O\bigg(\frac{D^{2k-3}\tau(D)^{O(1)}}{ Y} \bigg).
\]
\end{lemma}
\begin{proof}
    By the Chinese remainder theorem we have
   \begin{align*}  \sum_{\substack{b_1,c_1,\dots,b_{k-1},c_{k-1} \, (D) \\ N(b_1,\dots,b_{k-1}) \equiv aN(c_1,\cdots,c_{k-1}) \, (D)}} 1  = \prod_{p^{m} || D}\sum_{\substack{b_1,c_1,\dots,b_{k-1},c_{k-1} \, (p^m) \\ N(b_1,\dots,b_{k-1}) \equiv aN(c_1,\cdots,c_{k-1}) \, (p^m)}} 1 .
   \end{align*}
   The claim follows by using Lemma \ref{langweillemma} for $m= 1$ and the trivial bound $ \ll (p^m)^{2k-2}$ for $m \geq 2$, noting that
   \begin{align*}
       (D/D_1)^{2k-2} \leq Y (D/D_1)^{2k-3} .
   \end{align*}.
\end{proof}
\section{The arithmetic information for Theorem \ref{maintheorem3}} \label{ai3section}
 \subsection{Type I sums}
 We have the following Type I information, similar to Section \ref{typeisection}.
 \begin{prop}\emph{(Type I estimate).} \label{typeikprop}  Let $k \geq 3$ and let $\eta,\eta' > 0$ and suppose that $\eta$ is sufficiently small in terms of $\eta'$. Let $D \leq X^{1-1/(2k)-\eta'}$.  Then for any bounded coefficients $\alpha(d)$ we have
\[
\sum_{\substack{ d \sim D \\ n \sim X/d}} \alpha(d) (a^{(k)}_{dn} - b_{dn}) \ll_\eta X^{1-\eta}.
\]
\end{prop}
\begin{prop}\emph{(Fundamental lemma of the sieve).} \label{flkprop} Let $\eta,\eta' > 0$ and suppose that $\eta$ is sufficiently small in terms of $\eta'$. Let $W:=X^{1/(\log\log X)^2}$. Let $D \leq X^{1-1/(2k)-\eta'}$. Then for any bounded coefficients $\alpha(d)$ we have for any $C > 0$
\[
\sum_{\substack{ d \sim D \\ n \sim X/d}} \alpha(d)\bm{1}_{(n,P(W))=1} (a^{(k)}_{dn} - b_{dn}) \ll_{\eta,C} X \log^{-C} X.
\]
where $\eta > 0$ depends on $\eta' >0.$ 
\end{prop}
\begin{proof}
The proofs of these Propositions are analogous to the proofs of Propositions \ref{typeiprop} and \ref{flprop}, using Lemma \ref{normreplemma} to bound the number of representations as $N(b_1,\dots,b_{k-1})$ at the end of the large sieve bound, similarly as in Section \ref{typeicubesection}. The only thing we have to check is that the main terms match for $h=0$ after applying Poisson summation to sum over $a$, which follows by the construction of the weight $\Omega(b_1,\cdots,b_{k-1})$. To see this, we expand the definition of $\Omega(b_1,\cdots,b_{k-1})$ to get a finer-than-dyadic split for $N(b_1,\cdots,b_{k-1})$, so that the main term corresponding to Section \ref{typeimaintermsection} looks like
\[
 \kappa_k\int_{1/2 \leq B  \leq \frac{2\sqrt{X}}{C_k}} \sum_{\substack{b_1,\cdots, b_{k-1}  \in [0,X^{1/(2k)}]\\ N(b_1,\cdots,b_{k-1}) \equiv 0 \, (c)}}\psi_B(N(b_1,\cdots,b_{k-1})) \frac{\int \psi_{B}(u) du}{\int_{\mathbf{u} \in [0,X^{1/(2k)}]^{k-1}}  \psi_B(N(u_1,\dots,u_{k-1})) d \mathbf{u}} \frac{dB}{B}.
\]
The conditions $b_i \in [0,X^{1/(2k)}]$ are handled by using a finer-than-dyadic decomposition (Section \ref{smoothweightsection}) to each of the variables $b_i$, so that by Poisson summation (Lemma \ref{poissonlemma}) the above sum is up to a negligible error term
\[
\begin{split}
&\kappa_k \frac{\rho_k(c)}{c^{k-1}}\int_{1/2 \leq B  \leq \frac{2\sqrt{X}}{C_k}} \int_{\mathbf{u} \in [0,X^{1/(2k)}]^{k-1}}\psi_B(N(u_1,\dots,u_{k-1}))  d \mathbf{u} \\
& \hspace{180pt} \frac{\int \psi_{B}(u) du}{\int_{\mathbf{u} \in [0,X^{1/(2k)}]^{k-1}}  \psi_B(N(u_1,\dots,u_{k-1})) d \mathbf{u}} \frac{dB}{B} \\
&\hspace{100pt}=\kappa_k \frac{\rho_k(c)}{c^{k-1}}\int_{1/2 \leq B  \leq \frac{2\sqrt{X}}{C_k}}  \int \psi_{B}(u) du  \frac{dB}{B},
 \end{split}
\]
which matches the term coming from the comparison sequence $b_n$ after summing over $c$ by the definition of the arithmetic factor $\kappa_k$.
\end{proof}
\subsection{Type II sums}
Our type II information is given by the following.
\begin{prop}\emph{(Type II estimate).}\label{typeiikprop}
Let $M N\asymp X$ and for any small $\eta >0$ let
\[
X^{1/(2k)+\eta} \ll N \ll X^{1/k-2/(k(k+1))-\eta}
\]
Let $\alpha(m)$ and $\beta(n)$ be bounded coefficients.  Assume that $\beta(n)$ is supported on square-free numbers with $(n,P(W))=1$ and satisfies the Siegel-Walfisz property (\ref{swproperty}) and assume that $\alpha(m)$ is supported on $(m,P(W))=1$. Then for any $C > 0$
\[
\sum_{\substack{ m \sim M \\ n \sim N}} \alpha(m) \beta(n) (a^{(k)}_{mn} - b_{mn}) \ll_{\eta,C} X \log^{-C} X.
\]
\end{prop}
Similarly as in Section \ref{typeiisection}, we can use Proposition \ref{flkprop} to reduce this to the following.
\begin{prop} \label{typeiikproper}
Let $W=X^{1/(\log \log X)^2}$. Let $M N\asymp X$ and for any small $\eta >0$ let
\[
X^{1/(2k)+\eta} \ll N \ll X^{1/k-2/(k(k+1))-\eta}
\]
Let $\alpha(m)$ and $\beta(n)$ be bounded coefficients. Assume that $\beta(n)$  is supported on square-free numbers with $(n,P(W))=1$ and satisfies the Siegel-Walfisz property with main term equal to 0, that is, (\ref{swmain0}).  Assume that $\alpha(m)$ is supported on $(m,P(W))=1$.  Then for any $C > 0$
\[
\sum_{\substack{ m \sim M \\ n \sim N}} \alpha(m) \beta(n) a^{(k)}_{mn}  \ll_{\eta,C} X \log^{-C} X
\]
and
\[
\sum_{\substack{ m \sim M \\ n \sim N}} \alpha(m) \beta(n) b_{mn}  \ll_C X \log^{-C} X.
\]
\end{prop}
 \begin{proof}
 The arguments are essentially the same as in Section \ref{typeiiproofsection}.
 Denoting
 \[
 S(w,z) := \sum_{\Re(\overline{w}z)= N(b_1,\dots,b_{k-1})} \Omega(b_1,\dots,b_{k-1})
 \]
 and
 \begin{align*}
     U^{(k)}_1:= \sum_{|z_1|^2,|z_2|^2 \sim N} \beta_{z_1}\beta_{z_2} \sum_{w} F_{M}(|w|^2)  S (w,z_1) S (w,z_2),
 \end{align*}
 we need to show
\begin{align*}
    U_1^{(k)} \ll_C  X N \log^{-C} X.
\end{align*}
The contribution from the diagonal $\Delta=0$ is handled using $N \gg X^{1/(2k) + \eta}$ by the same argument as in Section \ref{diag1section},
using Lemma \ref{normreplemma} to bound the number of representations of any give integer by the form $N(b_1,\cdots,b_{k-1})$.
The contribution from the diagonal $(z_1,z_2) > 1$ is bounded by essentially the same argument as in Section \ref{diag2section}.

Thus, we are reduced to bounding the off-diagonal contributions
\begin{align*}
 U_{11}^{(k)} :=   \sum_{\substack{|z_1|^2,|z_2|^2  \sim N \\ \Delta \neq 0 \\ (z_1,z_2)=1}} \beta_{z_1}\beta_{z_2} \sum_{\substack{b_1,\dots,b_{k-1}, b_{k+1},\dots, b_{2k-1}\\  N(b_{k+1},\dots,b_{2k-1}) \equiv a N(b_1,\dots,b_{k-1})\, (\Delta)}} \Omega(b_1,\dots,b_{k-1})\Omega(b_{k+1},\dots,b_{2k-1})  \\
 F_{M} \bigg( \bigg| \frac{z_2 N(b_1,\dots,b_{k-1})-z_1N(b_{k+1},\dots,b_{2k-1}) }{\Delta} \bigg|^2 \bigg). 
\end{align*}
Let $Y=\log^{C'} X$ and let $\Delta_2$ denote the powerful part of $\Delta$. We separate the contribution from $\Delta_2 > Y$ to get
\[
 U_{11}^{(k)} =  U_{11 \leq}^{(k)} +  U_{11 >}^{(k)}.
\]
By similar arguments as in Section \ref{boundinglargeysection} for $V^{(1)}_{11 >}$, in place of \eqref{v11final} we get
\[
U_{11 >}^{(k)} \ll  X^{1-1/(2k)+2 \eps} N^{3/2} + \frac{XN \log^{O(1)}  X }{Y},
\]
which is sufficient since $N \ll X^{1/k - 10\eps}$.
 
For the main contribution  $U_{11 \leq}^{(k)}$, similarly as in Section \ref{offdiagsection}, we introduce smooth weights $\psi_{B_j}$ for the variables $b_j$ using Section \ref{smoothweightsection} with $X^{1/(2k)-\eta'} \ll  B_j \ll X^{1/(2k)}$ and denote
\[
H_j := X^\eps |\Delta| B_j^{-1}, \quad j \in  \{1,\dots,k-1,\dots,k+1,\dots,2k-1\}.
\] 
We may restrict to $ B_j \gg X^{1/(2k)-\eta'}$, for some $\eta'$ small compared to $\eta$, by estimating trivially the contribution from any $b_j \ll X^{1/(2k)-\eta'}$ before the application of Cauchy-Schwarz.  This ensures that $H_j$ are all within $X^{\eta'}$ of each other and we set 
\[
H:= \max_{j \in \{1,\dots,k-1,\dots,k+1,\dots,2k-1\}} H_j.
\]
Similarly in Section \ref{offdiagsection} we introduce a smooth partition for $|z_j|$ and $\arg z_j$ and we restrict to  $|\theta_1-\theta_2 \pmod{\pi}| > \sqrt{\delta}$ (by bounding the complement as in Section \ref{smalldeltasection}), which implies $|\Delta| > \sqrt{\delta} N$ and therefore $H=X^{O(\eta'+\eps)} N X^{-1/(2k)}$. Note that by $\Delta_2 \leq Y$ we then have $H < |\Delta_1|$ always. After this the smooth cross-condition $F_M$ may be removed with a negligible error term (by similar arguments as in Section  \ref{ccremovalsubsection}).  Note that
\[
\Omega(B_1,\dots,B_{k-1})\Omega(B_{k+1},\dots,B_{2k-1}) \ll X^{1/k}.
\]
Thus, denoting
\[
U_1^{(1)}(\bm{B},\bm{N},\bm{\theta}) :=  X^{1/k} \sum_{\substack{z_1,z_2 \\ \Delta \neq 0 \\ (z_1,z_2)=1 \\ \Delta_2 \leq Y}} \beta^{(1)}_{z_1}\beta^{(2)}_{z_2} \sum_{\substack{b_1,\dots,b_{k-1}, b_{k+1},\dots, b_{2k-1}\\  N(b_{k+1},\dots,b_{2k-1}) \equiv a N(b_1,\dots,b_{k-1})\, (\Delta)}} \psi_{B_1}(b_1) \cdots \psi_{B_{2k-1}}(b_{2k-1})
\]
we need to show that
\[
U_1^{(1)}(\bm{B},\bm{N},\bm{\theta}) \ll_C  XN \log^{-C} X.
\]
By Poisson summation (Lemma \ref{poissonlemma}) we obtain
\begin{align*}
  U_1^{(1)}(\bm{B},\bm{N},\bm{\theta}) = X^{1/k} \widehat{\psi}(0)^{2k-2} \prod_{j = 1}^{k-1} B_j B_{j+k}  \sum_{\substack{z_1,z_2 \\ \Delta \neq 0 \\ (z_1,z_2)=1 \\ \Delta_2 \leq Y}} \beta^{(1)}_{z_1}\beta^{(2)}_{z_2} \frac{N_k(a;\Delta)}{|\Delta|^{2k-2}} \\
  + O(X^{O(\eps+\eta')}  \widehat{U}_1^{(1)}(\bm{B},\bm{N},\bm{\theta})) + O_\eps(X^{-100}),
\end{align*}
where
\[
\widehat{U}_1^{(1)}(\bm{B},\bm{N},\bm{\theta}) := X^{1/k} \sum_{\substack{z_1,z_2 \\ \Delta \neq 0 \\ (z_1,z_2)=1 \\ \Delta_2 \leq Y}}   \frac{1}{H^{2k-2}} \sum_{\substack{|h_j| \ll H \\ \bm{h} \neq \bm{0}}} | S_k(a,\bm{h};\Delta)|
 \]
 with
 \begin{align*} 
 N_k(a;\Delta) &:=\sum_{\substack{b_1,\dots,b_{k-1}, b_{k+1},\cdots,b_{2k-1} \, (\Delta) \\ N(b_1,\dots,b_{k-1}) \equiv aN(b_{k+1},\cdots,b_{2k-1}) \, (\Delta)}} 1,
 \\
  S_k(a,\bm{h};\Delta) &:= \sum_{\substack{b_1,\dots,b_{k-1}, b_{k+1},\cdots,b_{2k-1} \, (\Delta) \\ N(b_1,\dots,b_{k-1}) \equiv aN(b_{k+1},\cdots,b_{2k-1}) \, (\Delta)}} e_{\Delta} ( \bm{h} \cdot \bm{b}).    
 \end{align*}

The main term from the Poisson summation is bounded similarly as in Section \ref{maintermsection}, using Lemma \ref{langweilDlemma} to evaluate $N_k(a,\Delta)$ and using the Siegel-Walfisz property of $\beta_z$. For the error term by using Lemma \ref{sumoverhlemma} with $\Delta/\Delta_1=\Delta_2 \leq Y$, noting that for $N=X^\alpha$ with $1/(2k) +\eta \leq \alpha \leq 1/k-2/(k(k+1))-\eta$ we have
\[
\beta = \frac{\alpha-1/(2k)}{\alpha} + O(\eta'+\eps)
\]
 and we get
 \[\begin{split}
\widehat{U}_1^{(1)}(\bm{B},\bm{N},\bm{\theta}) &\ll_\eps X^{O(\eps+\eta')} X^{1/k} N^{2+k-1}\bigg (N (1+ H^{2k(1/2-\beta)+2\beta-3})\\
&\hspace{50pt}+ N^{k(1/2-\beta)+\beta-1}(1+H^{k(1/2-\beta) + \beta-1})  + N^{(2k-1)(1/2-\beta) +2\beta- 2} H^{1/2-\beta} \bigg) 
\end{split}
\]
By $1 \leq H \leq N$  we have
\begin{align*}
N (1+ H^{2k(1/2-\beta)+2\beta-3}) & \ll  N +  N^{(2k-1)(1/2-\beta) +2\beta- 2} H^{1/2-\beta}, \\
  N^{k(1/2-\beta)+\beta-1}(1+H^{k(1/2-\beta) + \beta-1})  &\ll  N^{k(1/2-\beta)+\beta-1}+  N^{(2k-1)(1/2-\beta) +2\beta- 2} H^{1/2-\beta}
\end{align*}
and we get (using $N^{\beta}= N X^{-1/(2k)+ O(\eta'+\eps)}$ )
 \[\begin{split}
\widehat{U}_1^{(1)}(\bm{B},\bm{N},\bm{\theta})  &\ll_\eps X^{O(\eps+\eta')} X^{1/k} N^{2+k-1}\bigg (N +N^{k(1/2-\beta)+\beta-1}  + N^{(2k-1)(1/2-\beta) +2\beta- 2} H^{1/2-\beta} \bigg) \\
&\ll_\eps X^{O(\eps+\eta')} X^{1/k} \bigg(N^{2+k}    + N^{3k/2  - (k-1)\beta} + N^{2k-3/2 - (2k-1)\beta}H^{1/2-\beta}\bigg)  \\
&\ll_\eps X^{O(\eps+\eta')} \bigg(X^{1/k} N^{2+k}  + X^{1/k+(k-1)/(2k)}N^{k/2 -1} + X^{1/k+(2k-1)/(2k))}N^{-1/2} H^{1/2-\beta}\bigg)  \\
&\ll_\eps X^{O(\eps+\eta')} \bigg(X^{1/k} N^{2+k}+ X^{1/2+1/(2k)}N^{k/2 -1} + X^{1+1/(2k)}N^{-1/2}H^{1/2-\beta}\bigg), 
\end{split}
 \]
 which is $\ll X^{1-\eta} N$ since  $1/(2k) +\eta \leq \alpha \leq 1/k-2/(k(k+1))-\eta$ and since $\eta'$ is small compared to $\eta$.

 \end{proof}
\section{The sieve argument for Theorem \ref{maintheorem3}} \label{sieve3section}
Theorem \ref{maintheorem3} is a corollary of the following quantitative version, which approaches to an asymptotic formula for large $k$. We could apply Vaughan's identity if we had Type II information with the range $X^{1/(2k)} < N < X^{1/k}$, so that by Proposition \ref{typeiikprop} we are only missing a width of $2/(k(k+1))$ from this. 
\begin{theorem} \label{quantktheorem}
For some small $\eta > 0$ we have
\[
S(\A^{(k)},2\sqrt{X}) = (1+ O((\log k) ^{-\eta \log k}) ) S(\B,2\sqrt{X}).
\]
\end{theorem}
\begin{proof}
 Using the same argument as in Proposition \ref{funprop}, we get the following.
 \begin{prop} \label{funkprop}
Let $D := X^{1-1/(2k)-\eta}$, 
\[
\gamma := \frac{1}{2k} -\frac{2}{k(k+1)} - 2\eta,
\]
 and let $U \leq D $. Let $W={X^{1/(\log\log x)^2}}$. Then for any bounded coefficients $\alpha(m)$ supported on $(m,P(W))=1$ we have
\[
\sum_{m \sim U} \alpha(m) S(\A^{(k)}_m,X^{\gamma}) = \sum_{m \sim U} \alpha(m) S(\B_m,X^{\gamma}) + O_{\eta,C}( X/\log^C X ).
\]
\end{prop}

Denote $Z=X^\gamma$. When we iterate Buchstab's identity, we can insert the conditions
\[
p_1 \cdots p_{j-1} p_j^2 \leq 4 X
\]
with negligible error terms $\ll X/\log^2 X$ similarly as in Section \ref{sievesection}. Thus, after $2J$ applications of Buchstab's identity, we get
\[
\begin{split}
S(\CC,2 X^{1/2}) = \sum_{0 \leq j \leq 2J-1} (-1)^j \sum_{\substack{Z \leq p_j < \cdots < p_1 < 2 \sqrt{X} \\ p_1 \cdots p_{i-1} p_i^2 \leq 4 X,\, i \leq j}} S(\CC_{q_j},Z)  +  \sum_{\substack{Z \leq p_{2J} < \cdots < p_1 < 2 \sqrt{X} \\ p_1 \cdots p_{i-1} p_i^2 \leq 4 X,\, i \leq 2J}} S(\CC_{q_{2J}},p_{2J}) \\
+ O(X/\log^2 X),
\end{split}
\]
where we have denoted $q_j :=  p_1\cdots p_j$. By induction we see that the conditions  $p_1 \cdots p_{i-1} p_i^2 \leq 4 X$ imply that
\[
p_1 \cdots  p_i \leq 4 X^{1-2 ^{-i}}.
\]
Indeed, this is true for $i=1$ and by the induction hypothesis we have
\[
p_1 \cdots p_i = (p_1^2\cdots p_i^2)^{1/2} = (p_1\cdots p_{i-1})^{1/2} (p_1\cdots p_{i-1}p_i^2)^{1/2}\leq (4X^{1-2^{-i-1}})^{1/2} (4X)^{1/2} \leq 4 X^{1-2^{-i}}.
\]
Hence, we can give an asymptotic formula for the terms $j \leq 2J-1$ by Proposition \ref{funkprop} if for some small $\eta'>0$ we take
\[
J : = \eta' \log k,
\]
and our task is reduced to  estimating
\[
S_{2J}(\CC):=\sum_{\substack{Z \leq p_{2J} < \cdots < p_1 < 2 \sqrt{X} \\ p_1 \cdots p_{i-1} p_i^2 \leq 4 X,\, i \leq 2J}} S(\CC_{q_{2J}},p_{2J}).
\]
Here we could iterate Buchstab's identity provided that at each stage $p_1 \cdots p_{j-1} p_j^2 \leq D$, but this is not necessary for us. Note also that this sum is too large to be dropped completely, since by Lemma \ref{buchasymp} and Stirling's approximation
\begin{align*}
    \frac{S_{2J}(\B)}{S(\B,2\sqrt{X})} = & (1+o(1)) \int \omega(\bm{\beta}) \frac{d\beta_1 \cdots d \beta_{2J}}{\beta_1\cdots \beta_{2J-1} \beta_{2J}^2} \\
    \gg & \int_{ \substack{\gamma < \beta_{2J} < \cdots  < \beta_1 < 1/2 \\ \beta_1 + \cdots + \beta_{i-1} + 2 \beta_{i} < 1, \, i \leq 2J} } \frac{d\beta_1 \cdots d \beta_{2J}}{\beta_1\cdots \beta_{2J-1} \beta_{2J}^2} \\
    \gg & \frac{1}{(2J)!}\int_{ \substack{\gamma < \beta_{2J}, \dots, \beta_1 < 1/{2J+1} } } \frac{d\beta_1 \cdots d \beta_{2J}}{\beta_1\cdots \beta_{2J}} \\
    = & \frac{1}{(2J)!} \log \bigg( \frac{1}{(2J+1) \gamma} \bigg)^{2J} \\
     \gg & \frac{1}{(2J)!} (1-\eps)^{-2J} (\log  k)^{2J} \\
     \gg & \frac{1}{\sqrt{2J}} \bigg(  \frac{e \log k}{ (1-\eps)2 J}\bigg)^{2J}  \\
     \gg &  \frac{1}{\sqrt{2J}}\bigg( \frac{e \log 2}{(1-\eps)}\bigg)^{2J}\to \infty \quad \text{as} \quad k \to\infty,
\end{align*}
since the largest $J$  we can take   is $J= \lfloor  \frac{\log k}{2\log 2} \rfloor+1$. The problem is that the primes $p_i$ range over too long intervals and thus we seek to replace them by variables with shorter ranges.  To do so we will use the so-called reversal of roles to "break" any prime variable
\[ 
p_i \geq V:=2X/D= 2 X^{1/(2k)+\eta}=: X^{\beta}.
\]
By breaking a prime variable we mean that we will be able to replace the sum by a similar sum where the prime variable is replaced by a product of variables each smaller than $V$. To see how the argument goes, suppose we want to estimate a sum of the form
\[
\sum_{m \sim M} \alpha(m) \sum_{p_1,p_2 \geq V} S(\CC_{m p_1 p_2},z).
\]
for some function $z= z(mp_1p_2)$ (we will have $z(n)= P^{-}(n)$, the smallest prime factor). We apply inclusion-exclusion (Buchstab's identity) twice, first to the prime $p_1$ and  then to the prime $p_2$ to get
\[
\begin{split}
\sum_{m \sim M} \alpha(m) \sum_{\substack{n_1,p_2 \geq V  \\ (n_1,P(Z))=1} } &S(\CC_{m n_1 p_2},z) - \sum_{m \sim M} \alpha(m) \sum_{\substack{p_1'n_1,n_2 \geq V  \\ Z \leq p_1' < 2 \sqrt{n_1}  \\(n_1,P(Z))=1 \\(n_2,P(Z))=1 } } S(\CC_{m p'_1n_1 n_2},z)  \\
&+  \sum_{m \sim M} \alpha(m) \sum_{\substack{p_1'n_1,p_2'n_2 \geq V  \\ \\ Z \leq p_1' < 2 \sqrt{n_1} \\ Z \leq p_2' < 2 \sqrt{n_2}  \\(n,P(p_1'))=1 \\(n_2,P(p_2'))=1 } } S(\CC_{m p'_1n_1 p'_2 n_2},z).
\end{split}
\]
In the first sum we have an asymptotic formula by Proposition \ref{funkprop} after re-arranging the variables, since $n_1 > V$ implies that the rest of the variables are $\leq 2X/V = D$. Similarly, in the second sum we have an asymptotic formula by Proposition \ref{funkprop}, since $n_2 > V$. The third sum is similar to the original sum but we have replaced the large primes $p_1,p_2 > V$ by products of smaller variables and we can iterate the process for the new variables $n_i$ and any unbroken $p_i$ until all but possibly one of the variables $n_i$ satisfy $n_i \leq V$. We will then show that these remaining sums may be discarded and the error terms are sufficiently small by Lemma \ref{buchstablemma}.

To make the above sketch rigorous, we begin by writing
\[
S_{2J}(\CC) = \sum_{0 \leq j \leq 2J } R_j(\CC),
\]
where
\[
\begin{split}
R_j(\CC) &:= \sum_{\substack{Z \leq p_{2J} < \cdots < p_1 < 2 \sqrt{X} \\ p_1 \cdots p_{i-1} p_i^2 \leq 4 X,\, i \leq 2J \\ p_{j+1} \leq V < p_j}} S(\CC_{q_{2J}},p_{2J})  \\
&= \frac{1}{j!(2J-j)!} \sum_{\substack{Z \leq p_{2J}, \dots, p_1 < 2 \sqrt{X} \\ p_1 \cdots p_{i-1} p_i^2 \leq 4 X,\, i \leq 2J \\ p_{j+1} \leq V < p_j}} S(\CC_{q_{2J}},p_{\min}) + O(X/\log^{10} X),
\end{split}
\]
by using crude estimates and Type I information to bound the part where $p_i=p_j$ for some $i \neq  j$. For $j \in \{0,1\}$ we cannot apply reversal of roles, so we leave these as is. By Lemma \ref{buchstablemma} the contribution from $j=1$ is
\[
\begin{split}
& \ll \frac{1}{(2J-1)!}\int_{[\gamma,1] \times [\gamma,\beta]^{2J-1}}  \frac{d\beta_1 \cdots d\beta_k}{\beta_1\cdots\beta_{2J} \beta_{\min}} \\
& \ll \frac{k^2 }{(2J-1)!} \bigg(\frac{\beta-\gamma}{\gamma} \bigg)^{2J-1} \ll \frac{k^2 }{(2J-1)!} \bigg(\frac{10}{k+1} \bigg)^{2J-1}
\end{split}
\]
for $k$ sufficiently large. The contribution from $j=0$  is bounded similarly. Here and below the implied constants do not depend on $k$.

For $j \geq 2$ we iterate the reversal of roles to get terms with asymptotic formula and the remaining sums 
\[
T_j(\CC) := \frac{1}{j!(J-j)!}\sum_{\substack{k_1,\dots,k_j \geq 0  \\ 2 | (k_1+\cdots + k_j)}} \sum_{Z \leq p_{2J}, \dots, p_j \leq V} \sideset{}{'}\sum_{\substack{ V < p'_{i1} \cdots p'_{ik_i} n_i \leq 2 \sqrt{X} \\  Z \leq p'_{ik_i} < \cdots <  p'_{i1} < 2\sqrt{n_i p'_{ik_i} \cdots p'_{i2}} \\p'_{ik_i} n_i  > V \\ (n_i,P(p'_{ik_i} ))=1}} S(\CC_{q},p_{\min}),
\]
where $\Sigma'$ means that all but possibly one of the variables $n_i$ are $\leq V$.  Note that since $V < Z^2$ and $(n_i,P(p'_{ik_i} ))=1$, the $n_i$ must be primes. Denoting  $K= k_1+\dots+k_j$ and using Lemma \ref{buchstablemma}, we get that the deficiency from $T_j(\CC)$ is
\[
\ll \Delta_j := \frac{1}{j!(J-j)!} \sum_{\substack{ k_1,\dots,k_j  \geq 0 \\ 2 | (k_1+\cdots + k_j)  }} \frac{1}{k_1!\cdots k_j!}  \int_{\U_{2J+K}}  \frac{d\bm{\beta}}{\beta_1\cdots\beta_{2J+K} \beta_{\min}},
\]
where
\[
 \U_{2J+K} = [\gamma,\beta]^{2J-j}\times [\gamma,\beta]^{j-1} \times [\gamma,1]^{K+1} .
\]
We obtain
\[
\begin{split}
 \Delta_j &\ll  \frac{1}{j!(J-j)!} \bigg(\frac{10}{k+1} \bigg)^{2J-1} k^2 \sum_{\substack{k_1,\dots,k_j  \geq 0  }} \frac{(\log k)^{k_1+\cdots +k_j} }{k_1!\cdots k_j!}  \\
 & \ll \frac{1}{j!(J-j)!} \bigg(\frac{10}{k+1} \bigg)^{2J-1}  k^{j+2} 
 \end{split}
\]
Summing over $j$ we get a total deficiency
\[
\begin{split}
\sum_{j \leq 2J} \Delta_j &\ll \frac{k^2}{(2J)!}\bigg(\frac{10}{k+1} \bigg)^{2J-1} \sum_{j \leq 2J} \frac{J!}{j!(J-j)!}   k^{j}  \\
& \ll   \frac{k^2}{(2J)!}\bigg(\frac{10}{k+1} \bigg)^{2J-1} (1+k)^{2J} \ll (\log k)^{-\eta \log k}
\end{split}
\]
for some $\eta>0$, by using $J =   \eta' \log k $. This completes the proof of Theorem \ref{quantktheorem}.

Note that we could get a stronger bound by iterating reversal of roles to the primes $p'_{ij} > V$ as well, which would improve the deficiency to $\ll k^{-\eta \log k}$.
\end{proof}
 \bibliography{polyprimesbib}
\bibliographystyle{abbrv}
\end{document}